\documentclass{article}

\usepackage{amssymb}
\usepackage{amsmath}
\usepackage{amsthm}

\usepackage{subfigure} 
\usepackage{mathtools} 
\mathtoolsset{showonlyrefs}
\usepackage{dsfont} 
\usepackage{multicol} 
\usepackage{hyperref}
\usepackage[a4paper,left=3cm,right=3cm,top=3cm,bottom=3cm,]{geometry}

\newtheorem{theorem}{Theorem}[section]
\newtheorem{proposition}[theorem]{Proposition}
\newtheorem{definition}[theorem]{Definition}
\newtheorem{lemma}[theorem]{Lemma}
\newtheorem{corollary}[theorem]{Corollary}
\theoremstyle{definition}
\newtheorem{example}[theorem]{Example}
\newtheorem{remark}[theorem]{Remark}

\numberwithin{equation}{section}

\newcommand{\ith}[2]{#1^{ ( #2 ) } }
\renewcommand{\epsilon}{\varepsilon}
\renewcommand{\phi}{\varphi}
\newcommand{\Ide}{\ith I2}
\newcommand{\irl}{I}
\newcommand{\idv}{J}
\newcommand{\ent}{H}
\newcommand{\spaces}{S}
\newcommand{\alphabet}{\spaces^2}
\newcommand{\indic}{\mathds{1}}
\newcommand{\proba}{\mathcal P(\spaces^2)}
\newcommand{\probaun}{\mathcal P(\spaces)}
\newcommand{\balanced}{\mathcal P_\mathrm{bal}}
\newcommand{\abscont}{{\mathcal D}}
\newcommand{\bounded}{\mathcal B}
\newcommand{\expbounded}{\mathcal E}
\newcommand{\continuous}{\mathcal C}
\newcommand{\unifcont}{\mathcal C_{\mathrm u}}
\newcommand{\deltaa}{\delta}

\newcommand{\Ab}{\mathcal A}
\newcommand{\Ade}{\ith {\mathcal A} 2}
\newcommand{\Aun}{\ith {\mathcal A} 1}
\newcommand{\Ak}{\ith {\mathcal A}k}
\newcommand{\Adebal}{\ith {\mathcal A} 2_\mathrm{bal}}
\newcommand{\Abal}{\mathcal A_\mathrm{bal}}
\newcommand{\conv}{\mathrm{co}\,}
\newcommand{\sigmaco}{\sigma\text{-}\mathrm{co}\,}

\DeclareMathOperator{\cl}{\mathrm{cl}}
\DeclareMathOperator{\supp}{\mathrm{supp}}

\newcommand{\tvnorm}[1]{\left|#1\right|_{\mathrm{TV}}}

\newcommand{\tvnormm}[1]{\big|#1\big|_{\mathrm{TV}}}
\newcommand{\tvnorml}[1]{\bigg|#1\bigg|_{\mathrm{TV}}}
\newcommand{\tvnormxl}[1]{\Bigg|#1\Bigg|_{\mathrm{TV}}}

\newcommand{\uf}{f}
 
\newcommand{\Kdelta}{\spaces_1}
\newcommand{\Kun}{\spaces_2}
\newcommand{\words}{\spaces_{\mathrm{fin}}}
\newcommand{\wordspos}{\spaces_{\mathrm{fin,+}}}
\newcommand{\N}{{\mathbb N}}

\renewcommand{\d}{\mathrm d}
\renewcommand{\Omega}{\spaces^\N}

\newcommand{\closeword}{W}
\renewcommand{\hat}{}
\newcommand{\csli}{c_\mathrm{sl}}
\newcommand{\csti}{c_\mathrm{st}}

\newcommand{\stitchable}{A^{(t)}}
\newcommand{\kv}[1]{k_{\underline #1}}

\begin{document}

\title{Large deviations for {possibly reducible} Markov chains on discrete state spaces}

\author{Léo Daures\footnote{Université Paris Cité and Sorbonne Université, CNRS, Laboratoire de Probabilités, Statistique et Modélisation, Paris, France F-75205 CEDEX 13}} 
\date{}

\maketitle
\begin{abstract}
	We study the large deviations of Markov chains under the sole assumption that the state space is discrete. In particular, we do not require any of the usual irreducibility
	and exponential tightness assumptions. Using subadditive arguments, we provide an
	elementary and self-contained proof of the weak level-2 and level-3 large deviation principles. Due to the possible reducibility of the Markov chain, the rate functions may be nonconvex and may differ, outside a specific set, from the Donsker-Varadhan entropy and other classical rate functions.
\end{abstract}

\vspace{0.5cm}
\tableofcontents
\section{Introduction}
Let $X=(X_n)_{n\geq 1}$ be a Markov chain on a discrete countable state space $\spaces$ associated with a stochastic kernel $p$ and an initial distribution $\beta$, and let $\mathbb P$ denote the law of $X$.\label{not: X}\label{not: S}\label{not: p}\label{not: beta}\label{not: P}
From the seminal work of Donsker and Varadhan \cite{DV1,DV2,DV3,DV4} to the recent publication of a comprehensive monograph by de Acosta \cite{deacosta2022}, large deviations of Markov chains have become a standard subject in probability theory. 
A \emph{Large Deviation Principle} (LDP) describes an exponential rate of decay of probabilities. This work primarily focuses on \emph{weak LDPs}. We recall that weak LDPs differ from standard LDPs in that the upper bound is only required to hold for all compact (instead of all closed) sets; see Definition~\ref{def: weak and full LDP}. Following standard terminology, we will occasionally refer to the standard LDP as the {\em full LDP} to emphasize the contrast with the weak LDP.

For Markov chains, it is usual to study the large deviations of the time average of a given function $f:\spaces\to \mathbb R^d$, of the occupation times, and of the empirical process, respectively defined as
\label{not: Lninfty}\label{not: Ln1}
\begin{equation}\label{eq: definition empirical measures level 1 2 3}
	A_nf=\frac 1n\sum_{i=1}^nf(X_i),\qquad
	\ith L1_n=\frac 1n\sum_{i=1}^n\delta_{X_i},\qquad 
	\ith L\infty_n=\frac 1n\sum_{i=1}^n\delta_{T^i(X)},\qquad n\geq 1,
\end{equation}
where $T$ is the shift operator acting on $\spaces^\N$ by $T(x_1, x_2, \dots) = (x_2, x_3, \dots)$.
Here, $A_nf$ is a $d$-dimensional vector, $\ith L1n$ is a probability measure on $\spaces$, and $\ith L\infty_n$ is a probability measure on $\spaces^\N$.
Following the terminology of~\cite{ellis2006}, we refer to the (weak) LDPs of the three sequences in \eqref{eq: definition empirical measures level 1 2 3} as \emph{level-1}, \emph{level-2}, and \emph{level-3} (weak) LDPs respectively. 

The level-1, level-2, and level-3 full LDPs for Markov chains have been extensively studied in the literature. 
Most existing results are stated under assumptions that ensure two properties: one of \emph{irreducibility} and one of \emph{exponential tightness} (see Section~\ref{section: definition and main results} for definitions).

The irreducibility property generally provides the basis of the proofs of the LDP lower bound. 
We may add that irreducibility also typically yields the convexity of the rate functions;
simple examples of Markov chains that lack irreducibility properties, and which satisfy LDPs with nonconvex rate functions were already exhibited in~\cite{dinwoodie1993}.
In order to work in irreducible setups, some authors (historical references \cite{DV3,neynummelin1987, deacosta1990} as well as more recent works \cite{deacosta2013,deacosta2022}) straightforwardly require that the kernel $p$ be \emph{matrix-irreducible} or \emph{$\varphi$-irreducible},\footnote{Most authors simply call~\emph{irreducibility} the notion of $\varphi$-irreducibility. Even if it may be improper we will keep the name $\varphi$-irreducibility to avoid any ambiguity.}
but good irreducibility properties can also be consequences of broader assumptions, such as Assumption {\bf (U)} in \cite{deuschel1989,DZ}.\footnote{Originally introduced as Hypothesis 1.1(b') in \cite{ellis1988}.} 
Even though such irreducibility assumptions are often considered mild, we stress that non-irreducible Markov chains arise in many applications; models involving extinction events constitute just one prominent example.\footnote{One may think, among others, of some ecological models \cite{biobranchingprocess}, discretized birth-and-death processes, $N$-player ruin problems, or some chemical reaction networks models.}
We recall one of these models in Example~\ref{ex: wright fisher}. In such applications, reducibility is structural and not just the result of the existence of one absorbing state.


The property of exponential tightness, for its part, may be assumed through various hypothesis; one is {\bf(H*)}, originally introduced by Donsker and Varadhan in~\cite{DV3}, but assumption {\bf (U)} of \cite{deuschel1989,DZ} is also sufficient.
Broadly speaking, exponential tightness guarantees that the LDPs are full~---~which naturally associates exponential tightness to the upper bound of the LDP in the literature~---~and that the rate functions are good. 
In the absence of exponential tightness, the validity of the full level-2 and level-3 LDPs, as well as that of the (weak) level-1 LDP, seems to depend on the fine details of the Markov chain.
We recall in Appendix~\ref{section: examples and counterexamples} some well-known\footnote{See for instance Section~3 of~\cite{baxter1991}, Example~1 of \cite{brycdembo1996}, Exercises~13.14 and~13.15 of~\cite{rassoul} and Example 10.3 of~\cite{deacosta2022}.} examples of Markov chains (on countable spaces) which lack exponential tightness or goodness of the rate function, and for which only the level-2 and level-3 {\em weak} LDPs hold, and which even fail to satisfy the level-1 weak LDP.

In this paper, we prove that $(X_n)$ always satisfies the level-2 and level-3 weak LDP; see Theorems~\ref{theorem: intro weak LDP for Ln1} and~\ref{theorem: intro weak LDP for Lninfty} below.  Besides discreteness of the space $\spaces$, no assumption is made on the kernel $p$ or on the initial distribution $\beta$, and in particular, the results hold without exponential tightness or irreducibility. In view of the above discussion, the stated weak LDPs are the best achievable results in this situation.
These LDPs will be obtained as consequences of an auxiliary result, which is interesting in its own right, namely the weak LDP for the {\em pair empirical measures}; see  Theorem~\ref{theorem: intro weak LDP for Ln2}.
{Like in \cite{denhollander2008}, the study of the pair empirical measures is a key to a better understanding of the simple, usual, empirical measures; see Remarks~\ref{remark: intro defence of L2n} and~\ref{rem: defence of L2n}.}

In the derivation of the above weak LDPs, we will introduce the sets $\Ak$ of \emph{admissible measures}; see Definition~\ref{def: admissible measures}.
These sets contain the measures that are relevant for the LDPs, and the rate functions are trivially infinite on their complement.
Since LDPs describe exponential decay of probabilities, we stress that measures that give positive probability to transient states may be admissible and have a finite rate function. Examples are provided in Appendix~\ref{section: examples and counterexamples}.

The conclusions of Theorems~\ref{theorem: intro weak LDP for Ln2},~\ref{theorem: intro weak LDP for Ln1}, and \ref{theorem: intro weak LDP for Lninfty} are not new in themselves. 
In 2002, de La Fortelle and Fayolle \cite{fortelle2002} claimed the weak LDP for the pair empirical measure with modified version of the relative entropy as rate function; unfortunately their conclusions {are false as stated, and their proof appears to contain some gaps.} 
In 2005, Jiang and Wu \cite{wu2005} proved the level-2 and level-3 weak LDPs with modified versions of the Donsker-Varadhan entropy as rate functions. This proof was based on some functional analysis arguments from a previous rich article by Wu \cite{wu2000}.
Eventually, in 2015, the level-2 weak LDP made its way to the book of Rassoul-Agha and Seppäläinen \cite{rassoul}, where a proof is outlined as an adaptation of the proof for the irreducible case. 
All these works define a notion of admissible measure. 
We believe that the novelty of the present article lies in two major features.
First, we provide a proof based on subadditivity, which is self-contained and purely probabilistic.
Second, we provide an exhaustive identification of the LDP rate functions.

More precisely, the first feature consists in the use of the \emph{subadditive method}, which requires significant adaptations in order to handle reducible Markov chains.
In large deviations theory, the subadditive method consists in applying a version of the subadditive lemma (see for example Lemma 6.1.11 in \cite{DZ}), in order to prove the existence of the \emph{Ruelle-Lanford function} (RL function), which, in turn, implies the desired weak LDP; see Section~\ref{section: ruelle lanford functions}. 
Subadditive techniques have been employed extensively in the study of large deviations of irreducible Markov chains, leading to concise and elegant proofs of weak LDPs. 
The most classical version of the subadditive method in the case of irreducible Markov chains is the one presented in Chapter 4 of \cite{deuschel1989} and Chapter 6 of \cite{DZ} where the level-2 and level-3 LDPs with abstract rate functions are derived under Assumption~{\bf (U)}. The authors are then able to identify the rate functions with the desired ones. Recently, the method also appeared in~\cite{papierdecouple} under some {\em decoupling} assumptions which cover irreducible Markov chains on finite state spaces.
Some authors \cite{dinwoodieney1995,deacosta1998} also have extracted the main arguments of the subadditive method to use them separately in the derivation of the LDP lower bound under some irreducibility assumptions~---~the upper bound being derived by other means.
The most recent version of this argument is presented in Chapter 2 of \cite{deacosta2022}.
The classical subadditive method, and the variations used in all the above references, inherently relies on irreducibility and always produces convex rate functions. The improvements to the method that we provide involve carefully {\em slicing} and {\em stitching} finite-length trajectories, in a way that accomodates (possible) reducibility. To the best of the author's knowledge, this work provides the first instance of a subadditive argument beyond irreducibility and convexity. In addition to offering an elegant and self-contained proof of the level-2 and level-3 weak LDPs, we hope that these developments will pave the way for future extensions and applications of the subadditive method.

The second feature of the present article is the comprehensive identification of the rate functions. 
We identify the rate functions with all known standard expressions on the set of admissible measures $\Ak$, and show that they are infinite outside $\Ak$. The latter property is at the root of the above-mentioned nonconvexity of the rate functions in non-irreducible cases. In addition to the possible lack of convexity, the rate functions that we obtain may not be good, and the LDPs may not be full. The lack of these properties makes the identification of the rate functions more delicate than usual. 

This article proceeds as follows.
In the remainder of the current section, we provide some notation and definitions, and state the main results of this article.
In Section~\ref{section: weak LDP by subadditive method}, we present the modifications of the subadditive method required to accommodate the generality of our setup and derive the weak LDP for the \emph{pair empirical measure}, defined in~\eqref{eq: definition Ln2}. Section~\ref{section: identification rate function} is dedicated to the computation of the rate function of this weak LDP. Sections~\ref{section: contraction} and~\ref{section: weak LDP on the process level} use this weak LDP to derive respectively the weak LDP for the occupation times, and the weak LDP for the empirical process.

\subsection{Definitions and main results}
\label{section: definition and main results}
Let us recall the definition of exponential tightness, goodness, and of the weak and full LDP.
\begin{definition}
	\label{def: weak and full LDP}
	Let $(Y_n)=(Y_n)_{n\in \N}$ be a random process on a Hausdorff space $\mathcal X$, endowed with its Borel $\sigma$-algebra.
	\begin{enumerate}
		\item 
		We say that $(Y_n)$ is exponentially tight if for all $\alpha<\infty$, there exists a compact set $K\subseteq \mathcal X$ such that 
		\begin{equation*}
			\limsup_{n\to \infty}\frac 1n \log \mathbb P(Y_n\notin K)<-\alpha.
		\end{equation*} 
		\item 
		We say that $(Y_n)$ satisfies the weak (respectively, full) LDP if there exists a lower semicontinuous function $I:\mathcal X\to[0,+\infty]$ such that
		\begin{equation}
			\label{eq: LDP upper bound}
			\limsup_{n\to \infty} \frac 1n\log \mathbb P(Y_n\in K)\leq -\inf_{x\in K} I(x),
		\end{equation}
		for every compact (respectively, closed) set $K$ in $\mathcal X$ and
		\begin{equation}
			\label{eq: LDP lower bound}
			\liminf_{n\to \infty} \frac 1n\log \mathbb P(Y_n\in U)\geq -\inf_{x\in U} I(x),
		\end{equation}
		for every open set $U$ of $\mathcal X$. If so, $I$ is called the \emph{rate function} of the LDP. We say that the rate function is \emph{good} if its level sets are compact, {\it i.e.}~if $\{x\in \mathcal X: I(x)\leq C\}$ is compact for every $C\in \mathbb R$.
		We refer to~\eqref{eq: LDP lower bound} as the \emph{lower bound} and to~\eqref{eq: LDP upper bound} as the \emph{upper bound} of the (weak) LDP. 
	\end{enumerate}
\end{definition}
When a random process satisfies a full LDP or a weak LDP on a Polish space, its rate function is unique: see Lemma~4.1.4 and Exercise~4.1.30 of~\cite{DZ}.

Let $k\in \N:=\{1,2,\ldots\}$. Let $\mathcal P(\spaces^k)$ denote the set of probability measures on $\spaces^k$, embedded in the vector space $\mathcal M(\spaces^k)$ of finite signed measures on $\spaces^k$.\label{not: Pcal}\label{not: Mcal} Let $\bounded(\spaces^k)$ denote the set of bounded real-valued functions on $\spaces^k$.\footnote{We do not mention measurability nor continuity since $\spaces^k$ is discrete!}\label{not: Bounded} We define the dual pairing 
\label{not: pairing}
\begin{equation}
	\label{eq: dual pairing}
	(\mu,V)\mapsto \langle \mu, V\rangle= \sum_{u\in \spaces^k}V(u)\mu(u),
	\qquad V\in \bounded(\spaces^k),\ \mu\in \mathcal M(\spaces^k),
\end{equation}
and we call \emph{weak topology} the coarsest topology on $\mathcal M(\spaces^k)$ making $\langle \cdot,V\rangle$ continuous for all $V\in \bounded(\spaces^k)$.\footnote{The \emph{dual} pairing of~\eqref{eq: dual pairing} can also be used to define a weak topology on $\bounded(\spaces^k)$; see Appendix~\ref{section: duality}.}
We equip $\mathcal M(\spaces^k)$ with the weak topology.
A convenient metric to work with is the total variation (TV) distance, denoted $\tvnorm{\cdot}$.\label{not: TV}
Since $\spaces$ is discrete, $\tvnorm{\cdot}$ metrizes the weak topology on $\mathcal P(\spaces^k)$; this is because $\tvnorm{\cdot}$ can be expressed as half the $\ell ^1$ norm on $\mathcal P(\spaces^k)$.
In $\mathcal P(\spaces^k)$, the open $\tvnorm{\cdot}$-ball for of radius $\rho$ centered at $\mu$ is denoted by $\mathcal B(\mu,\rho)$.\label{not: ball}
When $\Kdelta$ is a subset of $\spaces$, we identify $\mathcal P(\Kdelta^k)$ with the set $\{\mu\in \mathcal P(\spaces ^k)\ |\ \supp\mu\subseteq \Kdelta^k\}$, where $\supp \mu = \{u\in \spaces^k\ |\  \mu(u)>0\}$ is the support of $\mu$.\label{not: supp}
The space $\mathcal P(\spaces^k)$ is complete; see for instance Theorem~6.5 of~\cite{parthasarathy2005}. 

Most of the work in this paper is done in $\mathcal P(\spaces^2)$, with the pair empirical measure $L_n$ defined as
\label{not: Ln2}
\begin{equation}
	\label{eq: definition Ln2}
	L_n=\ith L2_n
	=\frac1{n}\sum_{k=1}^{n}\delta_{(X_k,X_{k+1})}\in \proba, \qquad n\geq 1.
\end{equation}
The measure $L_n$, which counts the transitions between each pair of states, carries more useful information than the occupation times $L_n^{(1)}$.
Following \cite{fortelle2002}, we say that a finite measure $\mu$ on $\spaces^2$ is \emph{balanced} if its first and second marginal are equal \emph{i.e}~if
\begin{equation}
	\label{eq: definition mu balanced in S2}
	\mu(A\times \spaces) = \mu(\spaces\times A), \quad A\subseteq \spaces.
\end{equation}
We call $\mu^{(1)}$ the measure on $\spaces$ for which $\ith \mu 1(A)$ is given by the left-hand side of~\eqref{eq: definition mu balanced in S2} (or either side of~\eqref{eq: definition mu balanced in S2} when $\mu$ is balanced).\label{not: mu1}
The set of all balanced probability measures is closed and we denote it by $\balanced(\spaces ^2)$.\label{not: Pbal} It will play a central role in the weak LDP of Theorem~\ref{theorem: intro weak LDP for Ln2} below because, by definition, $L_n$ is $\frac2n$-close to a balanced measure. 

As we consider Markov chains that may be reducible, the decomposition of $\spaces$ into irreducible classes plays a central role in the following.
The stochastic kernel $p$ defines a relation $\leadsto$ on $\spaces$ with $x\leadsto y$ if there exists a finite sequence $x=x_1,x_2,\ldots, x_{n+1}=y$ in $\spaces$ satisfying $p(x_i,x_{i+1})>0$ for all $1\leq i\leq n$.\footnote{Note that, unlike other definitions of communication, this one requires the sequence to have at least two elements. In particular, we do not have $x\leadsto x$ in general.}\label{not: reachable}
We then say that $y$ is {\em reachable} from $x$. An irreducible class is a maximal subset $C$ of $\spaces$ such that $x\leadsto y$ for all $(x,y)\in C^2$. 
We denote by $B$ the set of all points $x\in \spaces$ that do not belong to any irreducible class; these are the points such that $x\not \leadsto x$, meaning that they can be visited by $X$ at most once. The state space $\spaces$ can be partitioned as
\label{not: Cj}\label{not: Jcal}
\begin{equation}
	\label{eq: decomposition of S in irreducible classes}
	\spaces=B\cup \Bigg(\bigcup_{j\in {\mathcal J}}C_j\Bigg),
\end{equation}
where the $C_j$ are the irreducible classes and ${\mathcal J}$ is countable.\footnote{Note that~\eqref{eq: decomposition of S in irreducible classes} is \emph{not} the Doeblin decomposition. The family $(C_j)_{j\in \mathcal J}$ contains both transient and essential classes. In particular, the sets $C_j$ are not the absorbing classes of the Markov chain, and may not be absorbing at all.}
We say that the class $C_{j'}$ is reachable from $C_{j}$ and write $C_j\leadsto C_{j'}$ if some $y\in C_{j'}$ is reachable from some $x\in C_j$.\footnote{Or equivalently, if all $y\in C_{j'}$ are reachable from all $x\in C_j$.} The relation $\leadsto$ is a partial order on $(C_j)_{j\in {\mathcal J}}$. 
We also define reachability from $\beta$ by writing $\beta\leadsto y$ if $y\in \supp \beta$ or if there exists $x\in \supp \beta$ such that $x\leadsto y$, and $\beta\leadsto C_j$ if some $y\in C_j$ is reachable from $\beta$.\footnote{Or equivalently, if all $y\in C_j$ are reachable from $\beta$.}
If there is a $j$ such that $x\leadsto C_j$ for all $x\in \spaces$, the Markov chain is $\varphi$-\emph{irreducible}.
If there is only one class and $B$ is empty, the Markov chain is \emph{matrix-irreducible}.

Using this decomposition into irreducible classes, we can now define admissible measures.
In this definition and in the following, if $\mu$ is a measure and $A$ is a set, $\mu|_A$ denotes the restriction of $\mu$ to $A$, which is the measure defined by $\mu|_A(\cdot)=\mu(A\cap \cdot )$.\label{not: muA}
\begin{definition}[Admissibility]
	\label{def: admissible measures}
	Let $k\in\N$. Let $\mu\in \mathcal P (\spaces^k)$. 
	We say that $\mu$ is \emph{pre-admissible} if there exists a set of indices $\mathcal J_\mu\subseteq \mathcal J$ such that 
	\begin{equation}
		\label{eq: definition pre admissible measure}
		\mu=\sum_{j\in \mathcal J_\mu}\mu|_{C_j^k},
	\end{equation} 
	in which case we impose that $\mathcal J_\mu$ is minimal.
	Moreover, if $\mu$ is pre-admissible, we say that $\mu$ is \emph{admissible} if the order $\leadsto$ is total on $(C_j)_{j\in {\mathcal J}_\mu}$ and $\beta\leadsto C_j$ for all $j\in {\mathcal J}_\mu$. 
	We denote by $\Ak$ the set of all admissible measures, and we set $\Adebal=\Ade\cap\balanced (\spaces^2)$.
\end{definition}
\label{not: Ak}
\label{not: Akbal}
\label{not: Jmu}
Those definitions are general, but we will mostly work with $k=1$ or $k=2$. When $k=1$, (pre-)admissible measures vanish on $B$. When $k=2$, (pre-)admissible measures are supported by pairs of points that belong to the same irreducible class.
Properties of $\Ak$ are further explored in appendix~\ref{section: admissible measures}.
In particular, setting
\label{not: abscont}
\begin{equation}
	\label{eq: definition absolutely continuous measures}
	\ith \abscont 2=\{\mu\in \mathcal M(\spaces^2)\ |\ \forall(x,y)\in \spaces^2,\,p(x,y)=0\Rightarrow\mu(x,y)=0\},
\end{equation}
Proposition~\ref{prop: characterization A2bal} shows that $\Adebal \cap\ith \abscont 2$ consists of precisely those measures which can be asymptotically approximated by $L_n^{(2)}$ with positive probability, and hence which are relevant for the corresponding LDP; see Remark~\ref{remark: other definitions of admissibility}.

From a technical point of view, the central result of this paper is the following theorem.
\begin{theorem}[Weak LDP for the pair empirical measures]
	\label{theorem: intro weak LDP for Ln2}
	The sequence $(\ith L2_n)$ satisfies the weak LDP in $\mathcal P(\spaces^2)$ with rate function $\Ide$ given by 
	\begin{equation}
		\label{eq: rate function I2}
		\Ide(\mu)=
		\begin{cases}
			(\ith \Lambda2)^*(\mu)=\ith \idv2(\mu)=\ith R2(\mu)\quad &\hbox{if $\mu\in \Adebal$},\\
			\infty&\hbox{otherwise},
		\end{cases}
	\end{equation}
	where the functions $\ith \Lambda2$, $\ith \idv2$, and $\ith R2$ are defined by equations~\eqref{eq: definition Lambda}, \eqref{eq: definition DV entropy} and~\eqref{eq: definition R} respectively and $\cdot^*$ denotes the convex conjugate as in Definition~\ref{def: convex conjugate}.
\end{theorem}
This theorem embodies two statements: first, $(\ith L2_n)$ satisfies the weak LDP, and second, the rate function of this weak LDP can be determined explicitly. 

The first point is developed in Section~\ref{section: weak LDP by subadditive method}, where a self-contained proof of the weak LDP for $(\ith L2_n)$ is provided (Theorem~\ref{theorem: Ln2 satisfies a weak LDP with rate function -s}) through the use of RL functions (Section \ref{section: ruelle lanford functions}). 
The second point, the identification of the rate function, is investigated in details in Section~\ref{section: identification rate function}. The expressions in the first line of~\eqref{eq: rate function I2} are standard in large deviations theory:
\begin{itemize}
	\item The function $(\ith \Lambda2)^*$ is the convex conjugate of the scaled cumulant generating function (SCGF) of the sequence $(\ith L2_n)$, defined in~\eqref{eq: definition Lambda}.\footnote{$\Lambda$ is also often called pressure, although this terminology is not physically accurate in general.}
	Since they provide information on the exponential scale, SCGFs are common objects in large deviations theory. Their convex conjugates are often identified as rate functions of LDPs, for instance via the Gärtner-Ellis Theorem (see~V.2 in~\cite{denhollander2008} or Chapter~2.3 of~\cite{DZ}) or Varadhan's Lemma (see III.3 in~\cite{denhollander2008}, or Chapters 4.3 and 4.5 of~\cite{DZ}).
	SCGFs are also relevant in the specific context of Markov chains, as in Theorem~6.3.8 of~\cite{DZ} and in Chapters 2 and 3 of~\cite{deacosta2022}.
	\item 
	The function $\ith \idv 2$ is the Donsker-Varadhan entropy (DV entropy), defined in~\eqref{eq: definition DV entropy} and first introduced by Donsker and Varadhan in~\cite{DV1}.\footnote{The function $J^{(2)}$ is actually the DV entropy for the pair empirical measures; the traditional DV entropy is the function  $\ith \idv 1$ defined in \eqref{eq: definition IDV1}.} Since then, it has been presented as the standard rate function of LDPs in Markovian setups (see for instance 
	Theorem~IV.7 of~\cite{denhollander2008}, or Chapter 6.5 of~\cite{DZ}). In the literature, it is common to find that the DV entropy satisfies the LDP lower bound, under very mild assumptions; see Chapter~3 of~\cite{deacosta2022}. See also Chapter~13.2 of~\cite{rassoul}.
	\item The function $\ith R 2$ is defined in~\eqref{eq: definition R}, and is expressed in terms of relative entropy. Going back to Sanov's theorem, relative entropy is always expected to play a role in level-2 large deviations. 
	This function appears for instance in the LDPs of Theorem~IV.3 of~\cite{denhollander2008}, of Chapter 6.5.2 of~\cite{DZ}, and of Chapter~13.2 of~\cite{rassoul}. {Note that the expression provided in \eqref{eq: definition R} is a completely explicit formula rather than the optimum of a function.}
\end{itemize}
We further notice that the functions $(\ith\Lambda2)^*$, $\ith\idv2$, and $\ith R2$ do not depend on $\beta$, so the rate function $\Ide$ only depends on $\beta$ through the set $\Adebal$. The second line of~\eqref{eq: rate function I2} is also interesting by itself, and is specific to the reducible case. It says that outside of $\Adebal$, the rate function is infinite. This is only a sufficient condition and $\ith I2$ can be infinite inside $\Ade$; see Remark~\ref{remark: I infinite on Adebal}. A brief analysis of the geometry of $\Adebal$ (see Appendix~\ref{section: admissible measures}) reveals that this property prevents the rate function from being convex in many cases, with the consequences discussed in the introduction.

The consequences of Theorem~\ref{theorem: intro weak LDP for Ln2} are the level-2 and level-3 weak LDPs:
\begin{theorem}[Level-2 weak LDP]
	\label{theorem: intro weak LDP for Ln1}
	The sequence
	$(\ith L1_n)_{n\geq 1}$ satisfies the weak LDP in $\mathcal P(\spaces)$ with rate function $\ith I 1$ given by
	\begin{equation}
		\label{eq: rate functions I1}
		\ith I 1(\mu)=
		\begin{cases}
			(\ith \Lambda1)^*(\mu)=\ith \idv 1(\mu)=\ith R1(\mu)\quad &\hbox{if $\mu\in \Aun$},\\
			\infty\quad &\hbox{otherwise,}
		\end{cases}
	\end{equation}
	where the functions $\ith \Lambda1$, $\ith \idv 1$, and $\ith R1$ are defined by equations~\eqref{eq: definition Lambda1}, \eqref{eq: definition IDV1}, and \eqref{eq: definition R1}, respectively and $\cdot^*$ denotes the convex conjugate as in Definition~\ref{def: convex conjugate}.
\end{theorem}
\begin{theorem}[Level-3 weak LDP]
	\label{theorem: intro weak LDP for Lninfty}
	Let $\mathcal P(\spaces^\N)$ be equipped with the weak topology; see Section~\ref{section: dawson gartner}.
	The sequence
	$(\ith L\infty_n)_{n\geq 1}$ satisfies the weak LDP in $\mathcal P(\spaces^\N)$, with rate function $\ith I \infty $ given by
	\begin{equation}
		\label{eq: rate functions Iinfty}
		\ith I \infty(\mu)=
		\begin{cases}
			(\ith \Lambda\infty )^*(\mu)=\ith \idv \infty(\mu)=\ith R\infty(\mu)\quad &\hbox{if $\mu\in \ith \Abal\infty$},\\
			\infty\quad &\hbox{otherwise,}
		\end{cases}
	\end{equation}
	where the functions $(\ith \Lambda\infty)^*$, $\ith \idv \infty$, and $\ith R\infty$ are defined by equations~\eqref{eq: definition Lambda infty *}, \eqref{eq: definition idv infty}, and \eqref{eq: definition R infty} respectively and the set $\ith \Abal \infty$ is introduced in Definition~\ref{def: admissibility on the process level}. Moreover, assuming that $\mu\in \ith \Abal\infty$ and $\ent (\ith \mu 1|\beta)<\infty$, we have the additional expression
	\begin{equation}
		\label{eq: level 3 relative entropy}
		\ith I\infty(\mu) =\lim_{k\to\infty}\frac 1k\ent  (\ith \mu k |\mathbb P_k),
	\end{equation}
	where $\ith \mu k$ is a marginal of $\mu$ that will be defined in Section~\ref{section: dawson gartner} and $\mathbb P_k$ is the marginal of $\mathbb P$ on the first $k$ coordinates; see~\eqref{eq: definition Pt}. The relative entropy $\ent (\cdot|\cdot)$ is defined in~\eqref{eq: definition relative entropy S(mu|nu)}.
\end{theorem}
The same comments as for Theorem~\ref{theorem: intro weak LDP for Ln2} can be made: the sets $\Aun$ and $\ith \Abal\infty$ which appear in \eqref{eq: rate functions I1} and \eqref{eq: rate functions Iinfty} may cause the rate functions to be nonconvex for reducible Markov chains.

The path we follow to derive Theorems~\ref{theorem: intro weak LDP for Ln1} and~\ref{theorem: intro weak LDP for Lninfty} from Theorem~\ref{theorem: intro weak LDP for Ln2} will come as no surprise to the experienced reader:
\begin{itemize}
	\item The weak LDP for $(\ith L1_n)$ is obtained using the contraction principle (see III.5 in~\cite{denhollander2008} or Theorem 4.2.1 of~\cite{DZ}), which has to be adapted, since we only have the weak LDP and the rate function may not be good.
	\item The weak LDP for $(\ith L\infty_n)$ is obtained using the Dawson-Gärtner Theorem (see for instance Theorem 4.6.1 of~\cite{DZ}). We follow the tracks of Sections 6.5.2, 6.5.3 of~\cite{DZ} or Section 4.4 of~\cite{deuschel1989}, where the level-3 LDP for Markov chains is proved under the uniformity assumption {\bf(U)}. Once again, in the absence of this assumption, some technical adaptations are required.
\end{itemize}

We briefly comment on the fact that if the state space $\spaces$ is finite, then all LDPs are full and all rate functions are good. In addition, in this case many technical aspects of the proofs can be considerably simplified; see Appendix~\ref{section: finite case} for a discussion. However, the main conceptual modifications of the subadditive method that we develop to handle reducible Markov chains, namely the {\em slicing} and {\em stitching} procedure below, remain necessary.

\begin{remark}
	\label{remark: intro defence of L2n}
	{
		Why do we need to prove Theorem \ref{theorem: intro weak LDP for Ln2} before proving Theorems~\ref{theorem: intro weak LDP for Ln1} and \ref{theorem: intro weak LDP for Lninfty}? 
		The first reason is technical and comes from the fact that the subadditive method only yields the weak LDP with an abstract rate function. 
		Indeed, the subadditive arguments of Section \ref{section: weak LDP by subadditive method} only conclude the existence of $\ith I2$. The identification of $\ith I2$ with expressions of \eqref{eq: rate function I2 for Ln2} is only done later in Section \ref{section: identification rate function}, and relies on the study of balanced measures (see Remark \ref{remark: about step 3 of IDV=Lambda*} and Lemma \ref{lemma: balanced measure decomposition}). The notion of balanced measure is fundamentally related to pair measures and has no equivalent in $\mathcal P(\spaces)$.
		If we used the same subadditive arguments as those of Section \ref{section: ruelle lanford functions} for the simple empirical measure $(\ith L1_n)$, we would be able to derive the existence of $\ith I1$, but not the expressions provided in \eqref{eq: rate functions I1}. With the tools at our disposal, it is not possible to identify $\ith I1$ without prior knowledge of $\ith I2$.
		
		Computing first $\ith I2$ and second $\ith I1$ was already the route of Chapter IV of \cite{denhollander2008}, even though this book only treats the irreducible finite case.
		In this reference, the simplicity of the expression $\ith I2=\ith R2$ allows the author to very quickly derive the LDP for $(L_n^{(2)})$ with rate function $\ith R2$ and to provide a connection with Sanov's Theorem. Note that the expression $\ith I2=\ith R2$ is not a variational expression but rather an explicit sum. The second reason for making the weak LDP for $(\ith L2_n)$ central is that the simplicity of this expression suggests think that the level of pair measures is the right level to consider.
		
		The third reason is that deriving the level-3 weak LDP of Theorem \ref{theorem: intro weak LDP for Lninfty} from the weak LDP for $(\ith L2_n)$ is easier in practice than deriving it from the weak LDP for $(\ith L1_n)$; see Section \ref{section: weak LPD for Lnk}.
	}
\end{remark}

\noindent
{\bf Note.} Since the first version of this preprint, the author has also studied the case where $S$ is the Euclidean space, using an approach similar to that of the present document. See \cite{daures2026}.

	\section{Weak LDP by subadditivity}
\label{section: weak LDP by subadditive method}
In this section, we prove the first part of Theorem~\ref{theorem: intro weak LDP for Ln2}, \emph{i.e.}~that the sequence $(L_n)$ defined in \eqref{eq: definition Ln2} satisfies the weak LDP. We use a subadditive method.
\subsection{Notations}
\label{section: notations}
Trajectories of $(X_n)_{n\geq 1}$ are sequences of elements of $\spaces$. 
In the following, we will often refer to the elements of $\spaces$ as {\em letters}, and to finite sequences of elements of $\spaces$, {\it i.e}.~finite pieces of trajectories, as {\em words}. The set of all words is denoted by $\words$, the set of words of length $t$ by $\spaces ^t$ and the length of a finite word $u$ by $|u|$. We will use the symbol $e$ for the empty word (which has length 0).\label{not: e}\label{not: words}\label{not: length}
We extend the definition of $p$ to finite words by setting 
\begin{equation}
	\label{eq: definition of p over words}
	p(u)=p(u_1,u_2)\times \ldots \times p(u_{|u|-1},u_{|u|}),
	\qquad u\in \words,\ |u|\geq 2,
\end{equation}
and $p(u)=1$ if $|u|\leq 1$.\label{not: p words}
The function $p$, seen as a function over $\words$, satisfies the following property: for $u,v\in \words$, $p(uv)=p(u)p(v)p(u_{|u|},v_1)$, where  the quantity $p(u_{|u|},v_1)$ depends only on the last letter of $u$ and the first letter of $v$. The law $\mathbb P$ induces a probability measure $\mathbb P_{t}$ on $\spaces ^t$ by
\label{not: Pt}
\begin{equation}
	\label{eq: definition Pt}
	\mathbb P_{t}(u)=\mathbb P(X_1=u_1,\ldots, X_t=u_t)=\beta(u_1)p(u).
\end{equation}
An empirical measure can be associated with each word. 
We define
\label{not: M(u)}
\label{not: L(u)}
\begin{equation}
	\label{eq: notation L[u] M[u]}
	M[u]=\sum_{i=1}^{t}\delta_{(u_i,u_{i+1})}\in \mathcal M(\spaces^2),\qquad L[u]=\frac 1{t} M[u]\in\proba,
	\qquad u\in \spaces^{t+1}.
\end{equation}
If $u$ has length 1 or 0, we set $L[u]$ to be the null measure.
We will extensively use the fact that $\tvnorm{M[u]}=|u|-1$ if $u\neq e$.\footnote{We also have $\tvnorm{M[e]}=0=|e|$.} 
Given $t\in \N$, $\Kdelta\subseteq \spaces$, $\mu\in \proba$ and $\rho >0$,
we define the following sets of words:
\begin{itemize}
	\item $\wordspos$ is the set of words $u$ such that $p(u)>0$.\label{not: wordspos}
	\item $\closeword_t(\mu,\rho)$ is the set of words $u\in \spaces^{t+1}$ such that $\hat L[u]\in \mathcal B(\mu,\rho)$ and $\mathbb P_{t+1}(u)>0$,\label{not: closeword1}
	\item $\closeword_{t,S_1}(\mu,\rho)$ is the intersection of $\closeword_t(\mu,\rho)$ and $\Kdelta^{t+1}$.\label{not: closeword2}
\end{itemize}
Observe also that $\Kdelta^t$ is exactly the set of words $u\in\spaces^t$ such that $L[u]\in \mathcal P(\Kdelta^2)$.

\begin{remark}
	\label{rem: defence of L2n}
{
	This section proves the abstract weak LDP for $(L_n)=(\ith L2_n)$ instead of that for $(\ith L1_n)$.
	We already commented on why it is preferable to study the weak LDP for $(\ith L2_n)$ rather than that for $(\ith L1_n)$ (see Remark \ref{remark: intro defence of L2n}). 
	However, one could contrast our strategy to obtain the weak LDP for $(\ith L2_n)$ with another valid one.  
	It would consist in applying the arguments of the current section to the empirical measures $\ith L{1,Z}_n$ of a generic Markov chain $(Z_n)$ in order to obtain the weak LDP for $(\ith L{1,Z}_n)$. Using the fact that $Y_n:=(X_n, X_{n+1})$ is a Markov chain as in Section \ref{section: weak LPD for Lnk}, this yields the weak LDP for $\ith L2_n=\ith L{1,Y}_n$. 
	Arguably, the efficiency of these two routes (ours and the latter one) is comparable.
	We choose to work consistently with $(L_n^{(2)})$ throughout. This approach reflects the central role of $(L_n^{(2)})$ in the theory of Markov chains, and prevents potentially confusing transitions between $(Y_n)$ and $(X_n)$.
}
\end{remark}
\subsection{Ruelle-Lanford functions and weak LDP for the pair empirical measures}
\label{section: ruelle lanford functions}
\begin{definition}[Ruelle-Lanford functions]
	\label{def: ruelle lanford functions}
	\label{not: sbar}\label{not: sund}
	For all Borel sets $A\subseteq \proba$, let
	\begin{equation*}
		\begin{split}
			\underline s(A)&=\liminf_{n\to \infty}\frac 1n\log \mathbb P(L_n\in A),\\
			\overline s(A)&=\limsup_{n\to \infty}\frac 1n\log \mathbb P(L_n\in A).
		\end{split}
	\end{equation*}
	For all $\nu\in\mathcal P(\spaces^2)$, we define $\underline s(\nu)$ and $\overline s(\nu)$ as the monotone limits
	${\underline s}(\nu)=\lim_{\delta\to0} {\underline s}(\mathcal B(\nu,\delta))$ and 
	${\overline s}(\nu)=\lim_{\delta\to0} {\overline s}(\mathcal B(\nu,\delta))$. We say that $(L_n)$ has a Ruelle-Lanford (RL) function if $\underline s=\overline s$ on $\mathcal P(\spaces^2)$, in which case the function $\underline s=\overline s$ is called the RL function of $(L_n)$.
\end{definition}

The notion of RL function was first introduced by Ruelle in~\cite{ruelle1965} as an expression for the entropy. 
Later, in~\cite{lanford1973}, Lanford revealed how efficient it can be in the theory of large deviations. This gave rise to the subadditive method, originally developed for independent and identically distributed random variables in~\cite{lanford1973,bahadurzabell1979,azencott}.
It has later been subjected to some adaptations to derive results for a larger class of random variables; among others, Markov chains in \cite{stroock1984,ellis1988}.
See the introduction for some examples of use of subadditivity for large deviations for Markov chains. See \cite{pfisterlewis1995} for a more detailed historical account. 

The interest in considering RL functions in our case resides in the following standard result; for a proof, see, for example, Theorem~4.1.1 in~\cite{DZ}, Proposition~3.5 of~\cite{pfister2000}, Theorem~3.1 of~\cite{pfisterlewis1995} and, more recently, Section~3.2 of~\cite{papierdecouple}.
\begin{lemma}
	\label{lemma: RL implies LDP}
	Suppose that $(L_n)$ has a RL function $s$. Then, $s$ is upper semicontinuous and $(L_n)$ satisfies the weak LDP with rate function $-s$.
\end{lemma}

Thanks to this lemma, the goal of this section is now simply to prove that $(L_n)$ has a RL function. 
The proof of the existence of the RL function and the derivation of its basic properties are achieved via the following technical proposition, whose proof will be covered by the next sections and concluded in Section~\ref{section: coupling map}. It states a key \emph{decoupled inequality}, which is the core of our subadditive method.

In the following proposition, for the purpose of the sole existence of the RL function, one can consider $\mu_1=\mu_2$\footnote{Taking two different measures in the construction is a common technique in subadditive methods, as it can often be used to later derive the convexity of the rate function; see~\cite{papierdecouple}. Even though our rate function is not convex, this technique still delivers useful properties; see Property~\ref{item: midpoint convexity of I2} of Proposition~\ref{prop: properties of rate function -s}.} and $\Kun=\spaces$ (and thus overlook the existence of $\Kdelta$).
\begin{proposition}
	\label{prop: P is decoupled}
	Let $0<\deltaa\leq 1/12$, and let, for any integers $n,t$,
	\begin{equation}
		\label{eq: definition Nnt}
		N=N(n,t)=\left \lceil \frac{t+1}{n(1-4\delta)}\right \rceil, 
		\qquad N_1=\Big\lfloor \frac{N}{2}\Big\rfloor,
		\qquad N_2=N-N_1.
	\end{equation}
	Let $\mu_1,\mu_2\in \Ade$ such that $\mu:=\frac12\mu_1+\frac12\mu_2\in \Ade$.
	Then, there exist an integer $n_0$, a sequence of integers $(t_n)$, and a finite set $\Kdelta\subseteq \spaces$ such that for all $n\geq n_0$, all $t\geq t_n$, and all set $\Kun\supseteq \Kdelta$,\footnote{The radius $20\delta$ in~\eqref{eq: decoupled inequality for P} is far from being optimal, but any multiple of $\delta$ acceptable for the purpose of the subadditive method; see how the factor $20$ eventually does not matter in the proof of Theorem~\ref{theorem: Ln2 satisfies a weak LDP with rate function -s}.}
	\begin{equation}
		\label{eq: decoupled inequality for P}
		\mathbb P_{t+1}(\closeword_{t,\Kun}(\mu,20\delta))\geq C_{n,t}\mathbb P_{n+1}(\closeword_{n,\Kun}(\mu_1,\delta))^{N_1}\mathbb P_{n+1}(\closeword_{n,\Kun}(\mu_2,\delta))^{N_2},
	\end{equation}
	for some constants $C_{n,t}$ satisfying
	\begin{equation}
		\label{eq: constant Cnt}
		\lim_{n\to \infty}\lim_{t\to\infty}\frac 1t\log C_{n,t}=0.
	\end{equation}
\end{proposition}
Taking Proposition~\ref{prop: P is decoupled} for granted, we can derive the existence of the RL function of $(L_n)$ and the weak LDP for $(L_n)$.
\begin{theorem}
	\label{theorem: Ln2 satisfies a weak LDP with rate function -s}
	\label{not: I2}
	\label{not: s}
	The sequence $(L_n)$ has a RL function $s:=\underline s=\overline s$ and satisfies the weak LDP with rate function $\Ide:=-s$.
\end{theorem}
\begin{proof}
	Let $\mu\in\Ade$. We set $\mu_1=\mu_2=\mu$ and $\Kun=\spaces$ in Proposition~\ref{prop: P is decoupled}. Then, \eqref{eq: decoupled inequality for P} yields
	\begin{equation*}
		\mathbb P_{}(L_t\in\mathcal B(\mu,20\delta))\geq C_{n,t}\mathbb P(L_n\in\mathcal B(\mu,\delta))^{N}.
	\end{equation*}
	By taking the logarithm on both sides, dividing by $t$ and letting $t\to \infty$, we obtain
	\begin{equation*}
		\underline s(\mathcal B(\mu,20\delta))\geq \lim_{t\to\infty}\frac{1}{t}\log C_{n,t}+\frac{1}{n(1-4\delta)}\log \mathbb P_{n+1}(L_n\in \mathcal B(\mu, \delta)).
	\end{equation*}
	By taking the limit superior as $n\to \infty$ and using \eqref{eq: constant Cnt}, we have
	\begin{equation*}
		\underline s(\mathcal B(\mu,20\delta))\geq \frac 1{(1-4\delta)}\overline s(\mathcal B(\mu,\delta)).
	\end{equation*}
	Further taking the limit as $\delta\to 0$ yields $\underline s(\mu)\geq\overline s(\mu)$. Since, obviously, we also have $\underline s(\mu)\leq\overline s(\mu)$, the sequence $(L_n)$ has a RL function $s$ and the weak LDP with the rate function $-s$ follows from Lemma \ref{lemma: RL implies LDP}.
\end{proof}
By Theorem~\ref{theorem: Ln2 satisfies a weak LDP with rate function -s}, the function $\ith I2$ is the unique rate function\footnote{In particular, $\ith I2$ is lower semicontinuous.} associated with the weak LDP for $(L_n)$.
Theorem~\ref{theorem: Ln2 satisfies a weak LDP with rate function -s} does not provide much information to compute $\Ide$ in practice. 
However, we can access a few useful properties of $\Ide$ as direct consequences of Proposition~\ref{prop: P is decoupled}, and of Proposition~\ref{prop: decoupling ineq for the decoupling map}, which we state here and will prove in Section~\ref{section: decoupling map}. 
Inequality~\eqref{eq: decoupling ineq for the decoupling map} can be seen as a reciprocal inequality of~\eqref{eq: decoupled inequality for P} in the case of disjoint supports and will indeed be obtained by reversing the construction used to derive Proposition~\ref{prop: P is decoupled}. 
\begin{proposition}
	\label{prop: decoupling ineq for the decoupling map} 
	Let $\lambda_1\in (0,1)$ and let $\lambda_2=1-\lambda_1$. 
	Let $0<\delta<\min(\lambda_1/12,\lambda_2/12)$ and let, for any integer $n$,
	\begin{equation*}
		t_i=t_i(n)=\lfloor n(\lambda_i-2\delta)\rfloor-1, \qquad i\in\{1,2\}.
	\end{equation*}
	Let $\mu_1,\mu_2\in \Ade$ be such that $\mu:=\lambda_1\mu_1+\lambda_2\mu_2\in \Ade$ and $\mathcal J_{\mu_1}\cap \mathcal J_{\mu_2}=\emptyset$.
	Then, there exists an integer $n_0$ such that for all $n\geq n_0$,\footnote{As for the factor $20$ in~\eqref{eq: decoupled inequality for P}, the factor $125/\lambda_i^2$ surely is not optimal, but this does not matter for the sequel; see the proof of Property~\ref{item: linearity of I2 between separated measures} of Proposition~\ref{prop: properties of rate function -s}.}
	\begin{equation}
		\label{eq: decoupling ineq for the decoupling map}
		\mathbb P_{t_1+1}\Big(\closeword_{t_1}\Big(\mu_1,\frac{125}{\lambda_1^2}\delta\Big)\Big)\mathbb P_{t_2+1}\Big(\closeword_{t_2}\Big(\mu_2,\frac{125}{\lambda_2^2}\delta\Big)\Big)\geq \widetilde C_n\mathbb P_{n+1}(\closeword_{n}(\mu,\delta)),
	\end{equation}
	for some constants $\widetilde{C}_n$ satisfying 
	\begin{equation}
		\label{eq: constant tilde Cn}
		\lim_{n\to\infty}\frac 1n\log \widetilde C_n=0.
	\end{equation}
\end{proposition}
Taking Proposition~\ref{prop: P is decoupled} and Prop~\ref{prop: decoupling ineq for the decoupling map} for granted, we deduce the following properties of the rate function $\ith I2$.
\begin{proposition}
	\label{prop: properties of rate function -s}
	The rate function $\Ide$ of Theorem~\ref{theorem: Ln2 satisfies a weak LDP with rate function -s} satisfies the following properties.
	\begin{enumerate}
		\item 		
		\label{item: midpoint convexity of I2}
		Let $\mu\in \Ade$ and let $\mathcal F=\{\nu\in\Ade\ |\ \mathcal J_\nu\subseteq \mathcal J_\mu\}$. Then, $\Ide|_\mathcal F$ is convex. In particular, we have the following. Let $\lambda_1\in[0,1]$ and let $\lambda_2=(1-\lambda_1)$.
		Let $\mu_1,\mu_2\in \Ade$ be such that $\mu:=\lambda_1\mu_1+\lambda_2\mu_2\in \Ade$. Then,
		\begin{equation}
			\label{eq: convexity of I}
			\Ide(\mu)\leq \lambda_1\Ide(\mu_1)+\lambda_2\Ide(\mu_2).
		\end{equation}
		\item 
		\label{item: linearity of I2 between separated measures}
		Let $\lambda_1\in[0,1]$ and let $\lambda_2=(1-\lambda_1)$.
		Let $\mu_1,\mu_2\in \Ade$ be such that $\mu:=\lambda_1\mu_1+\lambda_2\mu_2\in \Ade$ and $\mathcal J_{\mu_1}\cap\mathcal J_{\mu_1}=\emptyset$.
		Then,
		\begin{equation}
			\label{eq: linearity of I2 between separated measures}
			\Ide(\mu)=\lambda_1 \Ide(\mu_1)+\lambda_2\Ide(\mu_2).
		\end{equation} 
		\item 
		\label{item: I2 infinite outside Adebal}
		Let $\mu\in \mathcal P(\spaces^2)\setminus \Adebal$. Then,  
		\begin{equation}
			\label{eq: I2 infinite outside Adebal}
			\Ide(\mu)=\infty. 
		\end{equation}
		\item
		\label{item: I = - sinfty}
		For all Borel set $A\subseteq \proba$, let $\underline s_\infty(A)$ be the supremum of $\underline{s}(A\cap \mathcal P(K^2))$ over all finite $K\subseteq \spaces$.
		Let $\mu\in \mathcal P(\spaces^2)$. Then, the function $r\mapsto \underline s_\infty(\mathcal B(\mu,r))$ is nondecreasing and 
		\begin{equation*}
			\Ide(\mu)=-\lim_{r\to 0}\underline s_\infty(\mathcal B(\mu,r)).
		\end{equation*}
	\end{enumerate} 
\end{proposition}
One should read these properties with extra care. Property~\ref{item: midpoint convexity of I2} does not say that $\Ide$ is convex; see Example~\ref{ex: simple example three states}. The affinity stated in Property~\ref{item: linearity of I2 between separated measures} does not hold between every two measures; any two-states matrix-irreducible Markov chain is already a convincing counter-example. Also notice that Property~\ref{item: I2 infinite outside Adebal} only states a sufficient condition for $\ith I2$ to be infinite; see Remark~\ref{remark: I infinite on Adebal} and Example~\ref{ex: infinite entropy}.
\begin{proof}[Proof of Proposition~\ref{prop: properties of rate function -s}]
	\begin{enumerate}
		\item
		We begin with Property~\ref{item: midpoint convexity of I2}.
		Let $\nu_1,\nu_2\in \mathcal F$ and fix $\nu=\frac12(\nu_1+\nu_2)$. We first prove that 
		\begin{equation}
			\label{eqloc: midpoint convexity of I}
			\Ide(\nu)\leq\frac12(\Ide(\nu_1)+\Ide(\nu_2)).
		\end{equation}
		To do so, we use Proposition~\ref{prop: P is decoupled}.
		Considering~\eqref{eq: decoupled inequality for P} for $\nu, \nu_1,\nu_2$ with $\Kun=\spaces$, we have
		\begin{equation*}
			\mathbb P_{}(L_t\in\mathcal B(\nu,20\delta))\geq C_{n,t}\log\mathbb P(L_n\in\mathcal B(\nu_1,\delta))^{N_1}\mathbb P(L_n\in\mathcal B(\nu_2,\delta))^{N_2}.
		\end{equation*}
		By taking the logarithm of both sides, dividing by $t$ and letting $t\to \infty$, we get
		\begin{equation*}
			\underline s(\mathcal B(\nu,20\delta))\geq\lim_{t\to \infty}\frac 1t\log C_{n,t}+\frac{1}{2n(1-4\delta)}(\log\mathbb P(L_n\in\mathcal B(\nu_1,\delta))+\log\mathbb P(L_n\in\mathcal B(\nu_2,\delta))).
		\end{equation*}
		Taking the limit inferior of this expression as $n\to \infty$ yields
		\begin{equation*}
			\underline s(\mathcal B(\nu,20\delta))\geq \frac 1{2(1-4\delta)}(\underline s(\mathcal B(\nu_1,\delta))+\underline s(\mathcal B(\nu_2,\delta))).
		\end{equation*}
		Since $\Ide=-\underline s$, finally taking the limit as $\delta \to 0$ yields~\eqref{eqloc: midpoint convexity of I}.
		Let $\mu_1,\mu_2\in \mathcal F$. Then, ${\lambda\mu_1+(1-\lambda)\mu_2}\in \mathcal F$ for all $\lambda\in [0,1]$. Using a bisection argument based on~\eqref{eqloc: midpoint convexity of I}, together with the lower semicontinuity of $\irl$, we obtain, 
		for all $\lambda\in [0,1]$,
		\begin{equation*}
			\irl(\lambda\mu_1+(1-\lambda)\mu_2)\leq \lambda\irl(\mu_1)+(1-\lambda)\irl(\mu_2),
		\end{equation*} 
		which shows that $\Ide|_\mathcal F$ is convex. In particular, we have the following. For all $\lambda_1\in[0,1]$ and $\lambda_2=(1-\lambda_1)$,
		$\mu_1,\mu_2\in \Ade$ such that $\mu:=\lambda_1\mu_1+\lambda_2\mu_2\in \Ade$, both $\mu_1$ and $\mu_2$ are absolutely continuous with respect to $\mu$, hence they belong to $\mathcal F$ and satisfy~\eqref{eq: convexity of I}.
		\item
		We now prove Property~\ref{item: linearity of I2 between separated measures}. Let $\mu_1,\mu_2$ and $\lambda_1,\lambda_2$ be as in the statement of Property~\ref{item: linearity of I2 between separated measures}. By Property~\ref{item: midpoint convexity of I2}, we only have to prove that
		\begin{equation}
			\label{eqloc: concavity if I}
			\lambda_1\Ide(\mu_1)+\lambda_2\Ide(\mu_2)\leq \Ide(\mu).
		\end{equation}
		To do so, we use Proposition~\ref{prop: decoupling ineq for the decoupling map}. For convenience, we set $\delta_i=125\delta/\lambda_i^2$ for $i\in \{1,2\}$.
		Taking the logarithm on both sides in~\eqref{eq: decoupling ineq for the decoupling map}, dividing by $n$, and then letting $n\to\infty$ yields
		\begin{equation*}
			(\lambda_1-2\delta) \overline s(\mathcal B(\mu_1,\delta_1))+(\lambda_2-2\delta) \overline s(\mathcal B(\mu_2,\delta_2))
			\geq 0+\underline s(\mathcal B(\mu,\delta)).
		\end{equation*}
		Since $\Ide=-s$, taking the limit as $\delta\to 0$ yields~\eqref{eqloc: concavity if I}.
		\item
		Property~\ref{item: I2 infinite outside Adebal} is equivalent to $s(\mu)=-\infty$ for all $\mu \in\proba\setminus \Adebal$.
		Let $\mu \in\proba\setminus \Adebal$. by contrapositive of (\ref{item: sequence vn L[vn] to mu} $\Rightarrow$ \ref{item: mu admissible and absolutely continuous}) in Proposition~\ref{prop: characterization A2bal}, there exists $\ell\in \N$ and $\delta>0$ such that all words $w$ satisfying $|w|\geq \ell$ and $L[w]\in \mathcal B(\mu,\delta)$ must satisfy $\mathbb P_{|w|}(w)=0$. Hence, for all $n\geq \ell$, we have $\mathbb P(L_n\in \mathcal B(\mu,\delta))=0$.
		Therefore, $\overline s(\mathcal B(\mu,\delta))=-\infty$, implying $s(\mu)=-\infty$. 
		\item
		We now prove Property~\ref{item: I = - sinfty}.
		Let $\mu\in \proba$.
		The monotonicity of $r\mapsto \underline s _\infty(\mathcal B(\mu,r))$ is immediate by definition, hence Property~\ref{item: I = - sinfty} reformulates as 
		\begin{equation*}
			\lim_{r\to0}\underline s_\infty(\mathcal B(\mu,r))=\lim_{r\to0}\underline s(\mathcal B(\mu,r)).
		\end{equation*}
		The bound $\underline s_\infty(A)\leq \underline s(A)$ for all Borel sets $A\subseteq \proba$ is immediate, thus we have to show that
		\begin{equation}
			\label{eqloc: lim s infty geq lim s}
			\lim_{r\to0}\underline s_\infty(\mathcal B(\mu,r))\geq\lim_{r\to0}\underline s(\mathcal B(\mu,r)).
		\end{equation}
		When $\mu\notin \Adebal$, there is nothing to prove because the right-hand side limit is $s(\mu)=-\infty$ by Property~\ref{item: I2 infinite outside Adebal}. Assume that $\mu\in \Adebal$.
		Let $\delta>0$, and take $\Kdelta$ as given by Proposition~\ref{prop: P is decoupled}. Fix $n$ and $t$ as in Proposition~\ref{prop: P is decoupled}. There exists a finite set $\Kun\subseteq \spaces$, possibly depending on $n$, such that $\Kdelta \subseteq \Kun$ and 
		\begin{equation*}\label{eqloc: P(B inter H)> P(B)/2}
			\mathbb P_{n+1}( \closeword_{n,\Kun}(\mu,\delta))\geq \frac12\mathbb P_{n+1}( \closeword_n(\mu,\delta)).
		\end{equation*}
		Thus, by~\eqref{eq: decoupled inequality for P},
		\begin{equation*}
			\begin{split}
				\mathbb P_{t+1}(\closeword_{t,\Kun}(\mu,20\delta))
				&\geq C_{n,t}\mathbb P_{n+1}\left(\closeword_{n,\Kun}(\mu,\delta)\right)^{N}\\
				&\geq \frac{C_{n,t}}{2^{N}}\mathbb P_{n+1}(\closeword_n(\mu,\delta))^{N}.
			\end{split}
		\end{equation*}
		Taking the logarithm of both sides yields, dividing by $t$ and taking the limit inferior as $t\to\infty$ yields
		\begin{equation}\label{eq:chainsinf}
			\begin{split}
				\underline s_\infty(\mathcal B(\mu,20\delta))&\geq \underline s(\mathcal B(\mu,20\delta)\cap\mathcal P(\Kun^2))\\
				&\geq \lim_{t\to\infty}\frac 1t\log C_{n,t}
				-\frac{\log 2}{n(1-4\delta)}
				+\frac{1}{n(1-4\delta)}\log\mathbb P_n(\closeword_n(\mu,\delta)).
			\end{split}
		\end{equation}
		By further taking the limit inferior as $n\to\infty$ (both ends of the chain of inequalities are independent of $\Kun$, hence no problem ensues from the dependence of $\Kun$ on $n$), we get
		\begin{equation*}
			\underline s_\infty(\mathcal B(\mu,20\delta))\geq 
			\underline s(\mathcal B(\mu,\delta)).
		\end{equation*}
		We obtain~\eqref{eqloc: lim s infty geq lim s} by taking the limit as $\delta \to 0$.
	\end{enumerate}
\end{proof}
From now on, our goal is to prove Propositions~\ref{prop: P is decoupled} and~\ref{prop: decoupling ineq for the decoupling map}. This will be achieved in Sections~\ref{section: coupling map} and~\ref{section: decoupling map} respectively. 
\begin{remark} 
	Note that, while Proposition~\ref{prop: P is decoupled} is used in both the proofs the existence and properties of the RL function (Theorem~\ref{theorem: Ln2 satisfies a weak LDP with rate function -s} and Proposition~\ref{prop: properties of rate function -s}), Proposition~\ref{prop: decoupling ineq for the decoupling map} is only used to derive one property of the RL function (Property~\ref{item: linearity of I2 between separated measures} in Proposition~\ref{prop: properties of rate function -s}). The reader solely interested in the proof of the weak LDP can skip Section~\ref{section: decoupling map}.
\end{remark}

\subsection{Two preliminary examples}
\label{section: preliminary example}
Section~\ref{section: ruelle lanford functions} made clear that obtaining an inequality of the form~\eqref{eq: decoupled inequality for P} is the pivotal argument of our subadditive method. Before embarking in the proof of Proposition~\ref{prop: P is decoupled}, let us illustrate our strategy on simple examples.
\begin{example}
	\label{ex: aperiodic irreducible}
	Let $\spaces$ be finite, and assume that $(X_n)$ is matrix-irreducible and aperiodic. In other words, we assume that there exists $\tau>0$ such that, for all $x,y\in \spaces$, there exists $\xi_{x,y}\in \spaces^\tau$ satisfying $p(x\xi_{x,y} y)>0$. Let $\mu\in \proba$ and $\delta>0$. Then, for all $u,v\in \words$, there exists $\xi\in \spaces^\tau$ such that\footnote{Note that we used $\beta(v_1)\leq 1$.}
	\begin{equation}
		\label{eq: general decoupled inequality}
		\mathbb P_{|u|+|v|+\tau}(u\xi v)\geq c\mathbb P_{|u|}(u)\mathbb P_{|v|}(v),
		\qquad c:=\inf_{(x,y)\in\spaces^2}p(x\xi_{x,y}y)>0.
	\end{equation}
	Notice that if $u,v\in \closeword_n(\mu,\delta)$ for some $n\geq(2\tau+1)/2\delta$, then $u\xi v\in \closeword_{2n+\tau}(\mu,2\delta)$.
	Building on the above inequality, one can derive that for any family of words $(u^i)\in (\spaces ^n)^N$, there exists a family of words $(\xi^i)\in (\spaces^\tau)^{N-1}$ such that
	\begin{equation*}
		\mathbb P_{Nn+(N-1)\tau}(u^1\xi^1u^2\xi^2\ldots u^N)\geq c^{N-1}\prod_{i=1}^N\mathbb P_n(u^i).
	\end{equation*}
	By summing this inequality over $\closeword_{n}(\mu,\delta)^N$, one can find
	\begin{equation*}
		\mathbb P_{Nn+(N-1)\tau}(\closeword_{Nn+(N-1)\tau}(\mu,2\delta))\geq C_{n,N}\mathbb P_n(\closeword_{n}(\mu,\delta))^N,
	\end{equation*}
	where $C_{n,N}$ satisfies
	\begin{equation*}
		\lim_{n\to \infty}\lim_{N\to\infty}\frac 1{Nn+(N-1)\tau}\log C_{n,N}=0.
	\end{equation*}
	This inequality resembles~\eqref{eq: decoupled inequality for P} and is sufficient for using the machinery of Setion~\ref{section: ruelle lanford functions}. Therefore, $(L_n)$ satisfies the weak LDP.
\end{example}
This example is well-known in the literature.
A general construction for \emph{decoupled systems}, which include the case of Example~\ref{ex: aperiodic irreducible}, is provided in~\cite{papierdecouple}. The specific case of finite matrix-irreducible (but not necessarily aperiodic) Markov chains is mentioned in Example~2.20 as an application of the general construction. See also Examples~3 and~4 of~\cite{pfister2000}.

The central inequality in the construction of Example~\ref{ex: aperiodic irreducible} is~\eqref{eq: general decoupled inequality}. However, it is clear that the assumption of matrix-irreducibility was crucial in deriving such an inequality. In absence of irreducibility, how can we obtain an inequality that resembles~\eqref{eq: general decoupled inequality} enough to enable Proposition~\ref{prop: P is decoupled}?
The purpose of the rest of the section is to adapt this construction to reducible setups. As an introduction, let us illustrate the method with another simple, reducible, example.
\begin{example}
	\label{ex: preliminary example}
	Let $\spaces=\{1,2,3,4,5,6,7\}$ and consider a Markov chain represented on Figure~\ref{figure: preliminary example}, the initial measure $\beta$ being the uniform law on $\spaces$.
	\begin{figure}
		\centering
		\includegraphics{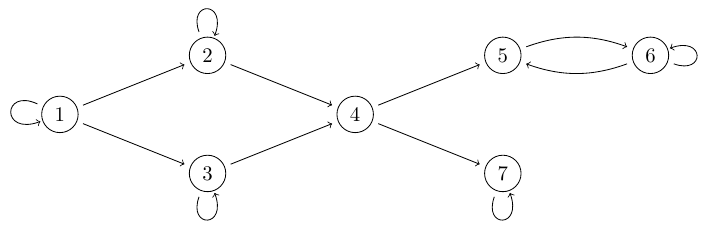}
		\caption{The 7-letter Markov chain of Example~\ref{ex: preliminary example}. The transition probabilities of each arrow do not matter for this example, as long as they are positive on each edge of the graph and zero otherwise.}
		\label{figure: preliminary example}
	\end{figure}
	The irreducible classes of the Markov chain are
	$(\{1\},\{2\},\{3\},\{5,6\},\{7\})$
	and $B=\{4\}$. The chain is neither matrix-irreducible nor $\varphi$-irreducible and thus, the existence of linking words $\xi$ as in Example~\eqref{ex: aperiodic irreducible} is not guaranteed. The two challenges brought by reducibility are the following. First, there is a notion of irreversibility. Since it is impossible to `come back' to $1$, the words $u=1112$ and $v=1245$ cannot possibly satisfy \eqref{eq: general decoupled inequality}.
	Second, some letters are incompatible with each other. Since neither $2\leadsto 3$ nor $3\leadsto 2$, no word of positive probability can contain simultaneously the letters $2$ and $3$. Thus we will not be able to derive an inequality of the form~\eqref{eq: general decoupled inequality} with the words $u=2222$ and $v=3333$.
	
	Let us consider a measure $\mu\in\Adebal$ and $\delta>0$. For instance, we take
	\begin{equation*}
		\mu=\frac14\delta_{1,1}+\frac18\delta_{2,2}+\frac14\delta_{5,6}+\frac14\delta_{6,5}+\frac18\delta_{6,6}\in\Adebal.
	\end{equation*}
	It is immediate that $\mu$ is balanced, preadmissible with $(C_j)_{j\in\mathcal J_{\mu}}=(\{1\},\{2\},\{5,6\})$, and admissible because $\beta\leadsto\{1\}\leadsto\{2\}\leadsto\{5,6\}$.
	
	Let $u=1112224565665665$ and $v=1111334565656566$. 
	Both words have their empirical measure close to $\mu$. Let us try to find a word $w$ such that $L[w]$ is almost as close to $\mu$, and $\mathbb P_{|w|}(w)$ is comparable to the product $\mathbb P_{|u|}(u)\mathbb P_{|v|}(v)$. The word $w$ will play the role of $u\xi v$ in~\eqref{eq: general decoupled inequality}. The key idea is to \emph{slice} words $u$ and $v$ into subwords that are easy to reassemble: we write
	\begin{equation*}
		u=(111)(222)(4)(565665665),\qquad v=(1111)(334)(565656566).
	\end{equation*}
	We can drop the subwords $4$ of $u$ and $334$ of $v$, as they consist of letters not appearing in the support of $\mu$ (in addition, the letter $3$ in $v$ is incompatible with the letter $2$ of $u$. Also notice that neither $u$ nor $v$ has the letter $7$, as it is incompatible with their letters $5$ and $6$). We then \emph{stitch} together the remaining subwords $111$, $222$, $565665665$, $1111$ and $565656566$ while respecting the total order on $(\{1\},\{2\},\{5,6\})$, and adding transition letters when needed. We set\footnote{One can notice that $w$ is shorter than $|u|+|v|$, in contrast with more standard subadditive arguments.}
	\begin{equation*}
		w=(111)(1111)(222)(4)(565665665)(6)(565656566).
	\end{equation*}
	The transition letters $4$ and $6$ were added in the word $w$ because $p(2,5)=p(5,5)=0$.
	The measure $L[w]$ is reasonably close to $\mu$. 
	We also have 
	\begin{align*}
		\mathbb P_{|w|}(w)
		&= kp(111)p(1111)p(222)p(565665665)p(565656566),
	\end{align*}
	where $k$ is a product of a few missing factors of the form $p(x,y)$ and $\beta(1)$.  
	Since
	\begin{equation*}
		\mathbb P_{|u|}(u)\leq p(111)p(222)p(565665665),\qquad \mathbb P_{|v|}(v)\leq p(1111)p(565656566),
	\end{equation*}
	we find $\mathbb P_{|w|}(w)\geq k\mathbb P_{|u|}(u)\mathbb P_{|v|}(v)$. This maneuver can be reproduced for any $u,v\in \words$ whose empirical measures are close to $\mu$. By bounding $k$ away from 0, we obtain an inequality that is similar to~\eqref{eq: general decoupled inequality}.
\end{example}
In Example~\ref{ex: preliminary example}, the construction of a longer word $w$ satisfying the desired properties resulted from two operations. First, the operation of \emph{slicing} the words $u$ and $v$ into smaller subwords, and second the operation of rearranging and \emph{stitching} those smaller subwords together. These operations were carried out in such a way that $p(w)$ is close to $p(u)p(v)$ and $L[w]$ is close to $\frac12L[u]+\frac 12L[v]$. The next section provides a general construction for these slicing and stitching operations.

\subsection{The slicing map and the stitching map}
\label{section: slicing and stitching}

We begin by introducing a \emph{slicing map} and a \emph{stitching map}, whose properties will be used to prove Proposition~\ref{prop: P is decoupled} in Section~\ref{section: coupling and decoupling maps}. The slicing map is used to turn a word $u$ into a collection of subwords $u^{j}$, which contains most of the letters of $u$ belonging to the class $C_j$. The stitching map is used to turn a collection of words into a longer word containing all of them as subwords.

In this section, $K$ is a finite subset of $\spaces$ and we denote by $\mathcal J_K$ the set $\{j\in\mathcal J\ |\ C_j\cap K\neq \emptyset\}$. Also in this section, we let $\mathcal J_0\subseteq \mathcal J_K$ be such that the order $\leadsto$ is total on $(C_j)_{j\in \mathcal J_0}$. Notice that $\mathcal J_K$ and $\mathcal J_0$ are necessarily finite. We also assume that they are not empty. For convenience, we rename the elements of $\mathcal J_0$ so that 
\begin{equation}
	\label{eq: class naming convention}
	\begin{cases}
		\mathcal J_0=\{1,\ldots,r\},\\C_j\leadsto C_{j+1}\ \forall j\in\{1,\ldots ,r-1\}.
	\end{cases}
\end{equation}
We denote $C_j\cap K$ by $K_j$ for all $j\in\mathcal J_0$. 
For later purposes, if $\beta\leadsto C_1$, we also denote by $K_0$ the singleton $\{z_0\}$ where $z_0$ is an arbitrary reference letter in $\supp\beta$ such that $z_0\leadsto C_1$.
\subsubsection{The slicing map}
\begin{definition}[The slicing map]
	\label{def: the slicing map}
	The slicing map is the map $F_{\mathcal J_0}:\wordspos\to \words^{\mathcal J_0}$ defined as follows. Let $u\in \wordspos$. 
	For all $j\in\mathcal J_0$, if $u$ does not have any letter in $ K_j$, we define $u^j=e$. Otherwise, we define $u^j$ as the largest subword of $u$ whose first and last letters are in $K_j$. We set $F_{\mathcal J_0}(u)=(u^j)_{j\in \mathcal J_0}$.
\end{definition}
In the following, except in Section~\ref{section: decoupling map}, the set $\mathcal J_0$ will be fixed, so we will simply write $F$ instead of $F_{\mathcal J_0}$.
\begin{remark}
	\label{remark: decomposition of u}
	Let $u\in\wordspos$ and $(u^j)_{j\in \mathcal J_0}=F(u)$. It follows from the definition that when $u^{j}\neq e$, the first and last letters of $u^j$ belong to $K_j\subseteq C_j$, but the letters in between may not belong to $K_j$. However, since $u\in\wordspos$ implies that $u_i\leadsto u_{i+1}$ for all $1\leq i\leq |u|-1$, we observe that all letters of $u^j$ belong to $C_j$, and that the only letters of $u$ that belong to $K_j$ are those contained in $u^j$. In particular, the subwords $u^j$ are non-overlapping.
	The word $u$ can be decomposed as follows:
	\begin{equation}
		\label{eq: decomposition of u}
		u=\zeta^{1}u^{1}\zeta^{2}u^{2}\ldots\zeta^{r}u^{r}\zeta^{r+1},
	\end{equation}
	where the $\zeta^j$ are the remaining subwords of $u$, in between the subwords $u^j$.
	In order to prevent any ambiguity in \eqref{eq: decomposition of u}, we set $\zeta^{j+1}=e$ when $u^{j}=e$.
\end{remark}
A visual representation of the application of $F$ to two words $u^{1}$ and $u^2$, decomposed respectively into $F(u^1)=(u^{1,1},u^{1,2})$ and $F(u^2)=(u^{2,1},u^{2,2})$, is given in Figure~\ref{figure: slicing map (subfigure)}.
Two useful properties of $F_{}$ are provided in Lemmas~\ref{lemma: geographic ineq for the slicing map} and \ref{lemma: proba ineq for the slicing map} below. In the following, we let
\begin{equation}
	\label{eq: definition Gamma}
	\Gamma=\bigcup_{j\in\mathcal J_0}K^2_j.
\end{equation}
\begin{lemma}
	\label{lemma: geographic ineq for the slicing map}
	Let $u\in \wordspos$ and let $(u^j)_{j\in \mathcal J_0}=F(u)$. Then, for all $j\in\mathcal J_0$, 
	\begin{equation}\label{eq: detailed geographic ineq for the slicing map}
		\tvnorm{M[u]|_{C^2_j}-M[u^j]}\leq M[u](C_j^2\backslash K_j^2),
	\end{equation}
	and 
	\begin{equation}
		\label{eq: geographic ineq for the slicing map}
		\tvnorml{M[u]-\sum_{j\in\mathcal J_0}M[u^j]}\leq M[u](\spaces^2\backslash \Gamma).
	\end{equation}	
\end{lemma}
\begin{proof}
	Let $j\in \mathcal J_0$. By Remark~\ref{remark: decomposition of u}, we have $M[u^j]\leq M[u]|_{C_j^2}$ in $\mathcal M(\spaces ^2)$, because every term $\delta_{(u_k,u_{k+1})}$ of the sum $M[u^j]$ also appears in the sum $M[u]|_{C_j^2}$. Thus, the left-hand side of \eqref{eq: detailed geographic ineq for the slicing map} is equal to $M[u](C_j^2)-M[u^j](C_j^2)$. Since we also have 
	\begin{equation*}
		M[u^j](C_j^2)\geq M[u^j](K_j^2)=M[u](K_j^2),
	\end{equation*}
	we obtain~\eqref{eq: detailed geographic ineq for the slicing map}.
	
	In the same way, $\sum_{j\in\mathcal J_0}M[u^j] \leq M[u]$, so the left-hand side of \eqref{eq: geographic ineq for the slicing map} is equal to $M[u](\spaces^2)-\sum_{j\in\mathcal J_0}M[u^j](\spaces^2)$. Since $$M[u^j](\spaces^2) \geq M[u^j](K_j^2) = M[u](K_j^2),$$ we obtain \eqref{eq: geographic ineq for the slicing map}.
\end{proof}
In the following, if $\underline u=(u^1,\ldots, u^k)\in\words^k$ is a list of words, we let $\kv u=k$ be the number of entries in $\underline u$ and\label{not: |u|}\label{not: ku}
\begin{equation}
	\label{eq: notation length list of words}
	|\underline u|=\sum_{i=1}^k|u^i|
\end{equation}
be the total length of words of $\underline u$.
\begin{lemma}
	\label{lemma: proba ineq for the slicing map}
	Let $F_n$ be the restriction of the map $F_{}$ to $\spaces^{n+1}\cap\wordspos$. Then, for all $\underline v=(v^j)_{j\in \mathcal J_0}$, 
	\begin{equation}
		\label{eq: proba ineq for the slicing map}
		\mathbb P_{n+1}(F_n^{-1}(\underline v))\leq \csli\prod_{j=1}^rp(v^j),
	\end{equation}
	where $\csli=(n+1)^{r+1}$.
\end{lemma}
\begin{proof}
	Inequality~\eqref{eq: proba ineq for the slicing map} is trivial if $\underline v$ is not in the range of $F_n$.
	Assume $\underline v$ is in the range of $F_n$ and let $u\in\wordspos$ be such that $F_n( u)=\underline v$, {\it i.e.} such that  $u^j=v^j$ for all $1\leq j\leq r$, where we let $u^j$ be the subword $F(u)^j$. We begin by comparing $\mathbb P_{n+1}(u)$ and the product of every $p(u^j)$ by identifying their common factors, and bounding the remaining ones. Consider the decomposition \eqref{eq: decomposition of u} of $u$. The quantity $p(u)$ is a product of factors $p(u^j)$, $p(\zeta^j)$, and $p(x,y)$, $x$ and $y$ being the first or last letters of some $u^j$ or $\zeta^j$. In this product, we bound by $1$ each factor $p(x,y)$ except the ones for which $y$ is the first letter of some $\zeta^j$. We have
	\begin{equation}
		\label{eqloc: bound on each P(u)}
		\mathbb P_{n+1}(u)\leq\Bigg(\prod_{j=1}^rp(u^{j})\Bigg)Q(u),\qquad 
		Q(u)=\prod_{\substack{1\leq j\leq r\\
		}}q_u^j(\zeta^j),
	\end{equation}
	where, for any $1\leq j\leq r$, the function $q_u^j$ is defined on $\words$ as follow:
	\begin{itemize}
		\item if $\zeta^j\neq e$ and $j>1$, then $q_u^j(\xi)=p(x\xi)$, where $x$ is the letter preceding $\zeta^j$ in $u$ (such a letter necessarily exists, since by convention $\zeta^{j}=e$ when $u^{j-1}=e$);
		\item if $\zeta^j\neq e$ and $j=1$, then $q_u^j(\xi) =\mathbb P_{|\xi|}(\xi)$;\footnote{Notice that the initial distribution $\beta$ only appears in $Q(u)$ if $\zeta^1 \neq e$. If $\zeta^1=e$, the inequality \eqref{eqloc: bound on each P(u)} bounds $\beta(u_1)$ by 1, while if $\zeta^1 \neq e$ we include $\beta(u_1)$ in $Q(u)$ in anticipation of some summation below.}
		\item if $\zeta^j=e$, then $q_u^j(\xi)=1$ for all $\xi\in \words$.
	\end{itemize}
	Notice that, for all $1\leq j\leq r$,\footnote{In other words, for any $j$, the function $q_u^j$ defines a probability measure on $\spaces^{|\zeta^j|}$. Also notice that $\zeta^j$ is involved in~\eqref{eqloc: sum quj(xi)} only through its length.}
	\begin{equation}
		\label{eqloc: sum quj(xi)}
		\sum_{\xi\in\spaces^{|\zeta^j|}}q_u^j(\xi)=1.
	\end{equation}
	Inequality~\eqref{eqloc: bound on each P(u)} already resembles the conclusion~\eqref{eq: proba ineq for the slicing map}. However, the map $F_n$ is far from injective, because of the complete loss of data carried by the $\zeta^i$, which prevents us from deriving \eqref{eq: proba ineq for the slicing map} from \eqref{eqloc: bound on each P(u)} by simply bounding $Q(u)$ by $1$. We will prove that 
	\begin{equation}
		\label{eqloc: bound on the sum of Q(u)}
		\sum_{u\in F_n^{-1}(\underline v)}Q(u)\leq (n+1)^{r+1}.
	\end{equation}
	Taking \eqref{eqloc: bound on the sum of Q(u)} for granted, we obtain  \eqref{eq: proba ineq for the slicing map} by summing \eqref{eqloc: bound on each P(u)} over all $u\in F_n^{-1}(\underline v)$.
	
	Thus, it only remains to prove \eqref{eqloc: bound on the sum of Q(u)}.
	From now on, since $u\in F_n^{-1}(\underline v)$ may vary, we will write $\zeta^j(u)$ instead of simply $\zeta^j$ to prevent any confusion. 
	To compute the sum on the left-hand side of \eqref{eqloc: bound on the sum of Q(u)}, we partition $F_n^{-1}(\underline v)$ according to the length of the $\zeta^j(u)$. Let us define integers $\ell_1, \dots, \ell_{r+1} \geq 0$ such that
	\begin{equation}
		\label{eqloc: constraint lk}
		\begin{split}
			\ell_1+\ldots+\ell_{r+1}+|\underline v|=n+1,
		\end{split}
	\end{equation}
	and consider all preimages $u$ of $\underline v$ such that the length of each $\zeta^{j}(u)$ is exactly $\ell_{j}$.
	Notice that all $u$ satisfying this constraint have the same functions $q_u^j(\cdot)$, which we simply call $q^j(\cdot)$. Under this constraint, identifying a specific datum $u$ amounts simply to identifying its subwords $\zeta^{j}(u)$. The sum of $Q(u)$ over this specific set of preimages is, by \eqref{eqloc: sum quj(xi)},
	\begin{align*}
		\sum_{\zeta^1 \in \spaces^{\ell_1}}\sum_{\zeta^2 \in \spaces^{\ell_2}}\cdots \sum_{\zeta^{r+1} \in \spaces^{\ell_{r+1}}}
		\prod_{j=1}^{r+1} q^j(\zeta^{j})
		=\prod_{j=1}^{r+1} \sum_{\zeta^{j}\in\spaces^{\ell_j}}q^j(\zeta^{j})
		=\prod_{j=1}^{r+1}1
		=1.
	\end{align*}
	This means that the left-hand side of \eqref{eqloc: bound on the sum of Q(u)} is bounded above by the number of possible ways to choose $(\ell_1, \dots, \ell_{r+1})$. A crude upper bound on this number is obtained as follows.
	There are at most $n+1$ possibilities for each $\ell_j$, thus there are at most $(n+1)^{r+1}$ possibilities for $(\ell_1, \dots, \ell_{r+1})$. The bound~\eqref{eqloc: bound on the sum of Q(u)} is proved, and the proof is complete.
\end{proof}
\subsubsection{The stitching map}
Recall that $z_0$ was defined at the beginning of the section to be a letter in $\supp \beta$ such that $z_0\leadsto C_1$ and $K_0=\{z_0\}$. Also recall that if $\underline v$ is a list of words, $\kv v$ denotes the number of entries in $\underline v$ and $|\underline v|$ denotes the total length of $\underline v$, as in~\eqref{eq: notation length list of words}.
\begin{lemma}[Transition words]
	\label{lemma: stitching words}
	Assume that $\beta\leadsto C_1$.
	Let $\Delta$ be the set of pairs $(x,y)$ where $x\in K_j$, $y\in K_{j'}$, for some $0\leq j\leq r$ and $1\vee j\leq j'\leq r$. Then, for all $(x,y)\in\Delta$, there exists a finite word $\xi_{x,y}$ such that 
	\begin{equation}
		\label{eq: definition eta tau}
		\eta:=\inf_{(x,y)\in \Delta} p(x\xi_{x,y}y)>0, 
		\qquad 
		\tau:= \sup_{(x,y)\in \Delta} |\xi_{x,y}|+1<\infty.
	\end{equation}
\end{lemma}
\begin{proof}
	For all $(x,y)\in\Delta$, we have $x\leadsto y$, so the definition of the relation $\leadsto$ yields a word $\xi_{x,y}$ satisfying $p(x\xi_{x,y}y)>0$. Since $\Delta$ is finite, the conclusion is immediate.
\end{proof}
\begin{definition}[The stitching map]
	\label{def: the stitching map}
	Let $t\in\N$.
	We say that a finite sequence of non-empty words $\underline v=(v^1,\ldots,v^k)\in \words^k$ is {\em stitchable} if there exists $1\leq j(1)\leq j(2)\leq \dots \leq j(k)$ such that for each $1\leq i \leq k$, the first and last letter of $v^i$ belongs to $K_{j(i)}$.	
	We let 
	\begin{equation*}
		\stitchable =\bigg\{\underline v\in \bigcup_{k\in \N}\words ^k \ \Big|\ \underline v\ \hbox{is stitchable},\ |\underline v|\geq t+1\bigg \}
	\end{equation*}
	denote the set of stitchable sequences of total length greater than $t+1$.
	We define the stitching map $G_t:\stitchable \to \spaces^{t+1}$ in the following way.
	When $\underline v\in\stitchable$, for $2\leq i\leq \kv v$, let $\xi^i$ be the word $\xi_ {x,y}$ from Lemma~\ref{lemma: stitching words} where $x$ is the last letter of $v^{i-1}$ and $y$ is the first letter of $v^{i}$. Denote by $\xi^1$ the word $\xi_{z_0,y}$, where $y$ is the first letter of $v^1$. Let
	\begin{equation}
		\label{eq: definition stitching map}
		G_{}(\underline v)=z_0\xi^1v^1\xi^2v^2\ldots \xi^{\kv v}v^{\kv v},
	\end{equation}
	and let $G_t(\underline v)$ be the prefix of the word $G(\underline v)$ of length $t+1$.
\end{definition}
Note that $G_t(\underline v)$ is well defined when $\underline v\in \stitchable$ because in that case, $|G(\underline v)|\geq t+1$.
A visual representation of the stitching map is provided in Figure \ref{figure: stitching map (subfigure)}.
\begin{lemma}
	\label{lemma: geographic ineq for the stitching map}
	Let $k\in\N$ and $t\in \N$. For all $\underline v\in\stitchable$,
	\begin{equation}
		\label{eq: geographic ineq for the stitching map}
		\tvnormxl{M[G_{t}(\underline v)]-\sum_{i=1}^{\kv v}M[v^i]}\leq |\underline v|-(t+1)+2\kv v\tau,
	\end{equation}
	where $\tau$ is as in \eqref{eq: definition eta tau}.
\end{lemma}
\begin{proof}
	We first compute the distance between $M[G(\underline v)]$ and $\sum_{i}M[v^{i}]$. As they are both sums of Dirac measures, we only have to count the number of terms that differ between the two. They share every Dirac term $\delta_{(v^{i}_j,v^{i}_{j+1})}$, thus, their difference consists of the sum of terms involving at least one letter of some $\xi^{i}$ or $z_0$.
	Since the words $z_0 \xi^1, \xi^2, \ldots, \xi^{\kv v}$ each contribute to at most $\tau$ terms in the sum defining $M[G_{t}(\underline v)]$, we find
	\begin{equation}
		\label{eqloc: tvnormmg1}
		\tvnormxl{M[G_{}(\underline v)]-\sum_{i=1}^{\kv v}M[v^{i}]}\leq \kv v\tau.
	\end{equation}
	Since $G_{t}(\underline v)$ is a prefix of $G_{}(\underline v)$, the TV distance between $M[G_{t}(\underline v)]$ and $M[G_{}(\underline v)]$ is the difference between the length of $G_{}(\underline v)$ and $t+1$. By definition of the stitching map, we have $|G_{}(\underline v)|		\leq |\underline v|+\kv v\tau$,	hence 
	\begin{align}
		\label{eqloc: tvnormmg2}
		\tvnormxl{M[G_{t}(\underline v)]-M[G_{}(\underline v)]}\leq |\underline v|-(t+1)+\kv v\tau.
	\end{align}
	Combining \eqref{eqloc: tvnormmg2} with \eqref{eqloc: tvnormmg1} yields \eqref{eq: geographic ineq for the stitching map}, as claimed.
\end{proof}
\begin{lemma}
	\label{lemma: proba ineq for the stitching map}
	Let $t,k_*, l_{*}\in\N$.
	Let 
	\begin{equation}
		\label{eq: definition Btkl}
		B^{(t)}_{k_*,l_*}=\{\underline v\in \stitchable\ |\ \kv v\leq k_*,\ |\underline v|\leq l_*\}.
	\end{equation}
	Then, for all $w\in\spaces^{t+1}$,
	\begin{equation}
		\label{eq: proba ineq for the stitching map}
		\mathbb P_{t+1}(w)\geq \csti\sum_{\substack{\underline v\in B^{(t)}_{k_*,l_*} \\ G_t(\underline v)=w }}\prod_{i=1}^{\kv v}p(v^i),
	\end{equation}
	where
	\begin{equation}
		\label{eq: constant csti}
		\csti={\beta(z_0)\eta^{k _*}}\frac1{(l_*+\tau k_*)^2}\bigg(\frac {2k_*}{e(l_*+(\tau+2) k_*)} \bigg)^{2k_*}.
	\end{equation}
	
\end{lemma}
\begin{proof}
	First, notice that for all $\underline v\in B^{(t)}_{k_*,l_*}$, we have $\mathbb P_{|G(\underline v)|}(G(\underline v))={\beta(z_0)}p(G(\underline v))$ and
	\begin{equation}
		\label{eqloc: bound sti for one word}
		p(G(\underline v))
		\geq \eta^{\kv v}\prod_{i=1}^{\kv v}p(v^i)
		\geq \eta^{k_*}\prod_{i=1}^{\kv v}p(v^i).
	\end{equation}
	Let $w\in\spaces^{t+1}$. The remainder of the proof consists in summing \eqref{eqloc: bound sti for one word} over all $\underline v\in B^{(t)}_{k_*,l_*}$ such that $G_t(\underline v)=w$, in order to obtain \eqref{eq: proba ineq for the stitching map}. 
	Assume $w\in G_t(B^{(t)}_{k_*,l_*})$, and consider a word $\kappa\in\words$ satisfying $|\kappa|+t+1\leq l_*+\tau k_*$. Let us show that 
	\begin{equation}
		\label{eqloc: bound sti for G before prefix}
		c\mathbb P_{|w\kappa|}(w\kappa)
		\geq
		\sum_{\substack{
				\underline v\in B^{(t)}_{k_*,l_*}\\
				G(\underline v)=w\kappa
		}}
		\prod_{i=1}^{\kv v}p(v^i),
		\qquad 
		c=k_*\beta(z_0)^{-1}\eta^{-k_*}\Big(\frac {e(l_*+(\tau+2)k_*)} {2k_*}\Big)^{2k_*}.
	\end{equation}
	Indeed, by summing~\eqref{eqloc: bound sti for one word} over all $\underline v\in G^{-1}(w\kappa)\cap B^{(t)}_{k_*,l_*}$, we get
	\begin{equation}
		\label{eqloc: bound P(wkappa)}
		\#\big(G^{-1}(w\kappa)\cap B^{(t)}_{k_*,l_*}\big)
		p(w\kappa)
		\geq \eta^{k_*}
		\sum_{\substack{
				\underline v\in B^{(t)}_{k_*,l_*}\\
				G(\underline v )=w\kappa
		}}
		\prod_{i=1}^{\kv v}p(v^i).
	\end{equation}
	To prove~\eqref{eqloc: bound sti for G before prefix}, it suffices to provide an estimate of the cardinality of $G^{-1}(w\kappa)\cap B^{(t)}_{k_*,l_*}$.
	Let $k\leq k_*$. To determine a preimage $\underline v $ of $w\kappa$ by $G$ such that $\kv v=k$, it suffices to specify the indices where the words $v^i$ and $\xi^i$ start in the word $G(\underline v)=w\kappa=z_0\xi^1v^1\ldots \xi^kv ^k$. This piece of information is given by the datum of a nondecreasing sequence of $2k$ elements in $\{0,1,\ldots ,|w\kappa|-1\}$. Thus,
	\begin{equation*}
		\#\big\{\underline v\in G^ {-1}(w\kappa)\ |\ \kv v=k\big \}\leq \binom{|w\kappa|+2k-1}{2k}\leq \Big(\frac {e(|w\kappa|+2k_*)} {2k_*}\Big)^{2k_*}.
	\end{equation*}
	In the bound above, we have used that $\binom {x+y}y\leq (e(x+y)/y)^y$ for any integers $y\leq x$, and that the function defined in the right-hand side of the bound is nondecreasing in $x$ and $y$.
	Hence, we have
	\begin{equation*}
		\#\big\{\underline v\in G^{-1}(w\kappa)\ |\ \kv v\leq k_*\big \}\leq k_*\Big(\frac {e(|w\kappa|+2k_*)} {2k_*}\Big)^{2k_*}.
	\end{equation*}
	By inclusion, this is also an upper bound on the cardinality of $G^{-1}(w\kappa)\cap B^{(t)}_{k_*,l_*}$. Hence, since $|w\kappa| =|\kappa|+t+1\leq l_*+\tau k_*$, we have
	\begin{equation*}
		\#\big(G^{-1}(w\kappa)\cap B^{(t)}_{k_*,l_*}\big)\leq k_*\Big(\frac {e(l_*+(\tau +2)k_*)} {2k_*}\Big)^{2k_*}.
	\end{equation*}
	Applying this bound to~\eqref{eqloc: bound P(wkappa)} and multiplying the inequality by $\beta(z_0)$, we obtain~\eqref{eqloc: bound sti for G before prefix} with the announced constant $c$.
	Now, we can use the estimations on $\mathbb P_{|w\kappa|}(w\kappa)$ provided by~\eqref{eqloc: bound sti for G before prefix} to bound the value of $\mathbb P_{t+1}(w)$.
	Observe that for all $m\geq 0$,
	\begin{equation*}
		\sum_{\kappa\in\spaces^m}\mathbb P_{t+1+m}(w\kappa)=\mathbb P_{t+1}(w)\sum_{\kappa\in\spaces^m}p(w_{t+1}\kappa)=\mathbb P_{t+1}(w).
	\end{equation*}
	Hence, by summing both sides in~\eqref{eqloc: bound sti for G before prefix} over all possible suffixes $\kappa$, we have:
	\begin{align*}
		(l_*+\tau k_*-(t+1))c\mathbb P_{t+1}(w)
		&=\sum_{\substack{\kappa\in \words\\ |w\kappa|\leq l_*+\tau k_*}}c\mathbb P_{|w\kappa|}(w\kappa)\\
		&\geq \sum_{\substack{\kappa\in \words\\ |w\kappa|\leq l_*+\tau k_*}}\sum_{\substack{
				\underline v\in B^{(t)}_{k_*,l_*}\\
				G(\underline v)=w\kappa
		}}
		\prod_{i=1}^{\kv v}p(v^i)\\
		&=\sum_{\substack{
				\underline v\in B^{(t)}_{k_*,l_*}\\
				G_t(\underline v)=w
		}}
		\prod_{i=1}^{\kv v}p(v^i).
	\end{align*}
	The last equality holds because all $\underline v\in B^{(t)}_{k_*,l_*}$ satisfy $|G(\underline v)|\leq l_*+\tau k_*$. A rough upper bound on the obtained constant yields~\eqref{eq: proba ineq for the stitching map} with constant $\csti$. 
\end{proof}

\subsection{Coupling trajectories}
\label{section: coupling and decoupling maps}
Using the slicing and stitching maps of the previous section, we build here the \emph{coupling} and \emph{decoupling maps}. In Section~\ref{section: coupling map}, the properties of the coupling map are used to finally prove Proposition~\ref{prop: P is decoupled}, thus allowing the machinery of Section~\ref{section: ruelle lanford functions} to provide the existence of the RL function and the weak LDP for $(L_n)$. In Section~\ref{section: decoupling map}, the properties of the {decoupling map} yield Proposition~\ref{prop: decoupling ineq for the decoupling map}, which is used in the derivation of the additional properties of the rate function, in Proposition~\ref{prop: properties of rate function -s}.
\subsubsection{The coupling map and proof of Proposition~\ref{prop: P is decoupled}}
\label{section: coupling map}

The goal of this section is to prove Proposition~\ref{prop: P is decoupled}.
Let us fix $\delta$, $\mu_1$, $\mu_2$, $\mu$, $N$, $N_1$, and $N_2$ as in Proposition~\ref{prop: P is decoupled}.
Note that this choice of parameters yields the two convenient bounds\footnote{The first one is the reason why we chose $\delta<1/12$.}
\begin{equation}
	\label{eq: bounds on delta an Nn/t}
	\frac1{1-4\delta}\leq 1+6\delta,
	\qquad
	t\geq \frac n\delta\ \Rightarrow\ 	1\leq \frac{Nn}{t}\leq 1+8\delta.
\end{equation}
We also fix a finite set $K\subseteq \spaces$ such that 
\begin{equation}\label{eq:defKgdelta}
	\mu_1(K^2)\geq 1-\delta, \qquad \mu_2(K^2)\geq 1-\delta,
\end{equation}
and we set $\mathcal J_0=\mathcal J_\mu\cap \mathcal J_K$, where $\mathcal J_K=\{j\in j\ |\ K\cap C_j\neq \emptyset\}$. The measure $\mu$ being admissible, $\mathcal J_0$ satisfies the same assumptions as in Section \ref{section: slicing and stitching}, and we keep the notations of \eqref{eq: class naming convention}.
\begin{lemma}
	\label{lemma: reoredring the list}
	Let $k\in \N$. There exists a map $\sigma: F_{}(\wordspos)^k\to \bigcup_{i=0}^{rk}\wordspos^i$ such that, for all $\underline w\in F_{}(\wordspos)^k$, the list $\sigma(\underline w)$ is stitchable and the elements of $\sigma(\underline w)$ are exactly the non-empty elements of $\underline w$.
\end{lemma}
\begin{proof}
	Let $\underline u=(u^1,\ldots ,u^k)\in \wordspos^k$.
	Applying $F$ to each word $u^i$ yields a $k\times r$ matrix $(u^{i,j})$ whose lines are $(u^{i,j})_{1\leq j\leq r}=F(u^i)$. If $u^{i,j}\neq e$, then the first and last letters of $u^{i,j}$ belong to $K_j$. We set $\sigma(F(u^1),\ldots, F(u^k))$ to the be the list obtained by reading the entries of the matrix $(u^{i,j})$ column by column, while skipping empty words.
\end{proof}
We introduce the \emph{coupling map} in Definition \ref{def: the coupling map} below. The coupling map takes $N$ words of length $n+1$ and turns them into one word of length $t+1$, using the slicing and stitching maps.
\begin{definition}[The coupling map]
	\label{def: the coupling map}
	Let $t, n\in \N$. 
	Let $\underline u=(u^1,\ldots, u^N)\in (\wordspos)^N$ satisfy $|u^i|=n+1$ for all $1\leq i\leq N$. 
	Let $\underline v=\sigma(F(u^1),\ldots, F(u^N))$, where $\sigma$ is the reordering map of Lemma~\ref{lemma: reoredring the list}.
	If $|\underline v|\geq t+1$, we set $\Psi_{n,t}(\underline u)=G_t(\underline v)$, where $G_t$ is the stitching map from Definition \ref{def: the stitching map}. We call 
	\begin{equation*}
		\Psi_{n,t}:(\wordspos \cap \spaces^{n+1})^N\to \words
	\end{equation*}
	the coupling map.
\end{definition}
\begin{figure}[thb]
	\centering 
	\includegraphics{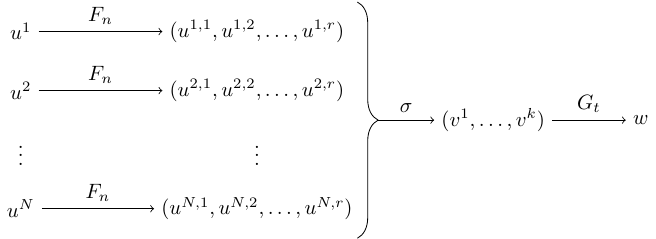}
	\caption{Definition of the coupling map.}
	\label{figure: sketch of coupling map}
\end{figure}
The definition of the coupling map as the composition of several maps is sketched in Figure~\ref{figure: sketch of coupling map}. Another useful illustration of the way the coupling map acts on trajectories is provided in Figure~\ref{figure: the coupling map}.
\begin{figure}[thb]
	\centering
	\subfigure[The action of the slicing map on two words $u^1$ and $u^2$.]{
		\includegraphics[width=0.8\textwidth]{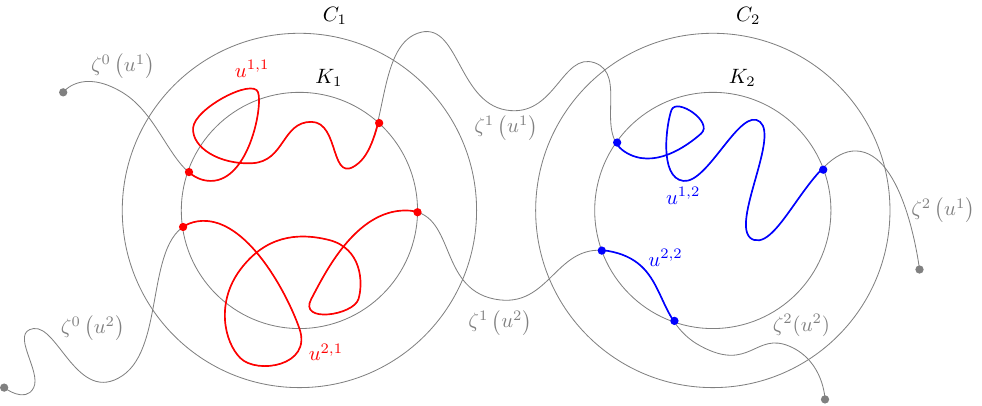}
		\label{figure: slicing map (subfigure)}
	}
	\vspace{0.5cm} 
	\subfigure[The action of the stitching map on a stitchable list $(v^1,v^2,v^3,v^4)$.]{
		\includegraphics[width=0.8\textwidth]{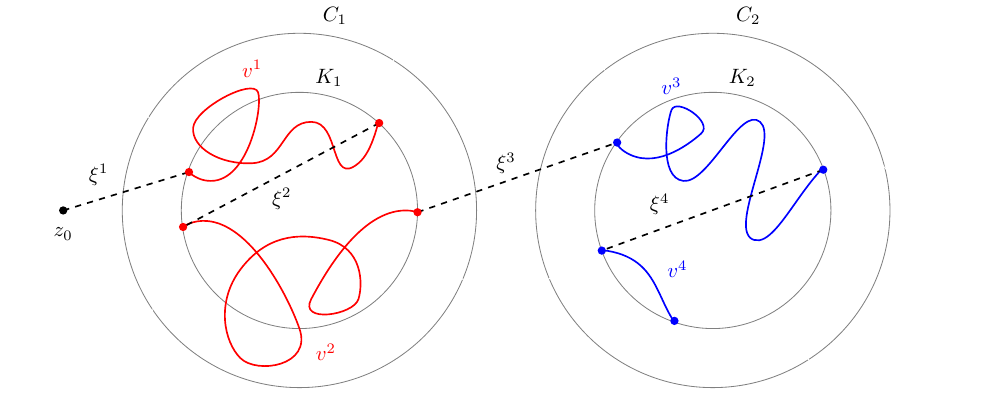}
		\label{figure: stitching map (subfigure)}
	}
	\caption{Example of use of the coupling map, with two words $u^1$ and $u^2$. The slicing map produces two subwords $u^{1,1}$, $u^{1,2}$ of $u^1$ and two subwords  $u^{2,1}$, $u^{2,2}$ of $u^2$ in~\eqref{figure: slicing map (subfigure)}. These are then reordered as the stitchable sequence $(v^1, \ldots, v^4)$ and reassembled by the stitching map in~\eqref{figure: stitching map (subfigure)}. Here, we assume that the length of the obtained word is $t+1$ so the operation of taking the prefix of length $t+1$ is not represented. Trajectories are depicted as smooth curves for easier readability, even though they are discrete.}
	\label{figure: the coupling map}
\end{figure}

\begin{lemma}
	\label{lemma: alphabet of the coupling map}
	There exists a finite set $\Kdelta\subseteq \spaces$ such that for all integers $n,t$ and $\Kun\supseteq \Kdelta$, for all $\underline u=(u^1,\ldots,u^N)\in(\Kun^{n+1})^N$, the word
	$\Psi_{n,t}(u)$ belongs to $\Kun^{t+1}$ if it exists. 
\end{lemma}
\begin{proof}
	We fix $\Kdelta$ to be the union of $K$, $\{z_0\}$, and the set of all letters involved in the transition words $\xi_{x,y}$ with $x,y\in K$. Then, $S_1$ is a finite set. Let $\Kun\supseteq \Kdelta$ and let $\underline u=(u^1,\ldots, u^N)\in (\Kun ^{n+1})^N$. By construction, any letter in the word $\Psi_{n,t}(u)$ belongs either to a transition word $\xi_{x,y}$ or to some $u^i$. In both cases, it is an element of $\Kun$.
\end{proof}
\begin{lemma}
	\label{lemma: geographic ineq for the coupling map}
	Let $n,t$ be integers. Let $\underline u=(u^1,\ldots,u^N)\in (\spaces^{n+1})^N$ 
	satisfy $\sum_{i=1}^NM[u^i](\Gamma)\geq t+1$, where $\Gamma$ was defined in \eqref{eq: definition Gamma}. Then,
	$\Psi_{n,t}(\underline u)$ is well defined and 
	\begin{equation}
		\label{eq: geographic ineq for the coupling map}
		\tvnorml{M[\Psi_{n,t}(\underline u)]-\sum_{i=1}^N{M[u^i]}}\leq 2\big(Nn-(t+1)\big)+2N(r\tau+1).
	\end{equation} 
\end{lemma}
\begin{proof}
	Let $\underline v=(v^1,\ldots, v^k)$ be the stitchable list $\sigma(F(u^1),\ldots, F(u^N))$ as in Definition~\ref{def: the coupling map}.\footnote{Since the number of empty words in the list may vary, $k$ is not fixed here. The number $k$ depends on the data $(u^1,\ldots, u^N)$.} 
	We first argue that $\underline v \in \stitchable$, so that $\Psi_{n,t}(\underline u)$ is well defined. By its definition, $\underline v$ is stitchable, hence it suffices to show that $|\underline v|\geq t+1$, {\it i.e.}~that
	\begin{equation}
		\label{eqloc: condition for |Psi(U)| > t}
		\sum_{i=1}^k|v^i|= \sum_{i=1}^N\sum_{j=1}^r \tvnorm{M[F_n(u^i)^j]}+k\geq t+1.
	\end{equation} 
	By Lemma~\ref{lemma: geographic ineq for the slicing map} and the assumption on $\underline u$,
	\begin{align*}
		\sum_{i=1}^N\sum_{j=1}^r \tvnorm{M[F_n(u^i)^j]}
		&\geq \sum_{i=1}^N\bigg(\tvnormm{M[u^i]}-\tvnorml{M[u^i]-\sum_{j=1}^rM[F_n(u^i) ^j]}\bigg)
		\\
		&\geq \sum_ {i=1}^N\left(n-(n-M[u^i](\Gamma))\right)\geq t+1.
	\end{align*}
	This shows \eqref{eqloc: condition for |Psi(U)| > t}, thus $\Psi_{n,t}(\underline u)$ is well defined.
	We now turn to the proof of \eqref{eq: geographic ineq for the coupling map}. By definition,
	\begin{equation*}
		\sum_{i=1}^N\sum_{j=1}^rM[F(u^i)^j]=\sum_{i=1}^k M[v^i],
		\qquad
		\Psi_{n,t}(\underline u)=G_t(\underline v).
	\end{equation*}
	On the one hand, by Lemma~\ref{lemma: geographic ineq for the stitching map},
	\begin{equation*}
		\tvnorml{M[\Psi_{n,t}(\underline u)]-\sum_{i=1}^kM[v^i]}\leq |\underline v|-(t+1)+2k\tau\leq N(n+1)-(t+1)+2Nr\tau.
	\end{equation*}
	On the other hand, by Lemma~\ref{lemma: geographic ineq for the slicing map}, 
	\begin{align*}
		\tvnorml{\sum_{i=1}^N\sum_{j=1}^rM[F_n(u^i)^j]-\sum_{i=1}^N{M[u^i]}}
		&\leq \sum_{i=1}^N\tvnorml{\sum_{j=1}^r M[F(u^i)^j]-M[u^i]}\\
		&\leq \sum_{i=1}^NM[u^i](\spaces^2\setminus \Gamma)\\
		&\leq Nn-(t+1).
	\end{align*}
	Combining these two bounds yields~\eqref{eq: geographic ineq for the coupling map} and completes the proof. 
\end{proof}
\begin{lemma}
	\label{lemma: proba ineq for the coupling map}
	For all $w\in\spaces^{t+1}$ in the range of $\Psi_{n,t}$, 
	\begin{equation}
		\label{eq: proba ineq for the coupling map}
		\mathbb P_{t+1}(w)\geq C_{n,t}\mathbb P_{n+1}^{\otimes N}\left(\Psi_{n,t}^{-1}(w)\right),
	\end{equation}
	for some constants $C_{n,t}$ satisfying \eqref{eq: constant Cnt}.
\end{lemma}
\begin{proof}
	To prove~\eqref{eq: proba ineq for the coupling map}, we compute the probability of the preimage $\Psi_{n,t}^{-1}(w)$ of a fixed word $w$ in the range of $\Psi_{n,t}$. 
	Let $\underline u=(u^1,\ldots, u^N)$ be a preimage of $w$ by $\Psi_{n,t}$, that is to say $\underline u$ satisfies $G(\sigma(F_n(u^1),\ldots, F_n(u^N)))=w$. 
	Since we only consider words $u^i$ of length $n+1$, that are each sliced by $F_n$ into at most $r$ subwords, the stitchable sequence $\sigma(F_n(u^1),\ldots,F_n(u^N))$ is of total length at most $N(n+1)$ and has less than $Nr$ entries. In other words, every preimage $\underline u$ of $w$ by $\Psi_{n,t}$ satisfies
	\begin{equation}
		\label{eqloc: preimages of w by psi}
		\sigma(F_n(u^1),\ldots,F_n(u^N))\in B^{(t)}_{Nr,N(n+1)},
	\end{equation} 
	where $B ^{(t)}_{k_*,l_*}$ is defined by~\eqref{eq: definition Btkl}.  Then,  by~\eqref{eqloc: preimages of w by psi}, 
	\begin{align}
		\mathbb P_n^{\otimes N}(\Psi^{-1}_{n,t}(w))
		&=\sum_{\substack{
				\underline v\in \stitchable\\
				G_t(\underline v)=w
		}}
		\sum_{u\in \sigma^{-1}(\underline v)}\prod_{i=1}^N\mathbb P_n(F_n^{-1}(u^i))
		\nonumber\\
		&=\sum_{\substack{
				\underline v\in B^{(t)}_{Nr,N(n+1)}\\
				G_t(\underline v)=w
		}}
		\sum_{u\in \sigma^{-1}(\underline v)}\prod_{i=1}^N\mathbb P_n(F_n^{-1}(u^i)).
		\label{eqloc: probability of preimage of w}
	\end{align}
	Let $\underline v\in B^{(t)}_{Nr,N(n+1)}$ be such that $G_t(\underline v)=w$ and let $u=(u^{i,j})\in \sigma^{-1}(\underline v)$. 
	By Lemma \ref{lemma: proba ineq for the slicing map},
	\begin{equation*}
		\mathbb P_n(F^{-1}_n(u^i))\leq \csli\prod_{j=1}^rp(u^{i,j})
	\end{equation*}
	for all $1\leq i\leq N$.
	There is a one-to-one correspondence between the family of non-empty subwords $u^{i,j}$ and $(v^1,\ldots, v^{\kv v})$, hence,
	\begin{equation*}
		\prod_{i=1}^N\mathbb P_n(F^{-1}(u^i))\leq\csli^N\prod_{i=1}^N\prod_{j=1}^rp(u^{i,j})= \csli^N\prod_{j=1}^{\kv v}p(v^{j}).
	\end{equation*}
	Therefore, by injecting this bound into~\eqref{eqloc: probability of preimage of w}, we obtain
	\begin{equation}
		\label{eqloc: before bounding sigma {-1}(v)}
		\mathbb P_n^{\otimes N}(\Psi^{-1}_{n,t}(w))\leq 
		\sum_{\substack{
				\underline v\in B^{(t)}_{Nr,N(n+1)}\\
				G_t(\underline v)=w
		}}
		\#\sigma^{-1}(\underline v)\csli^N\Bigg(\prod_{j=1}^{\kv v}p(v^{j})\Bigg).
	\end{equation}
	We need a combinatorial argument to bound the factor $\#\sigma^{-1}(\underline v)$, {\it i.e.} the number of $N\times r$ matrices of words $U= (u^{i,j})$ such that $\sigma (U)=\underline v$.
	If we know which of the words $u^{i,j}$ are empty, there is only one possibility, that is filling the matrix of words with entries $v^l$ in the right order while leaving blank the specified entries. Hence $\#\sigma ^{-1}(\underline v)$ is simply the number of choices of which of the $Nr$ words are empty, knowing that there must be $Nr-\kv v$ of them. Using the crude bound
	\begin{equation*}
		\#\sigma ^{-1}(\underline v)=\binom{Nr}{\kv v}\leq 2^{Nr}
	\end{equation*}
	in~\eqref{eqloc: before bounding sigma {-1}(v)}, we obtain
	\begin{equation*}
		\mathbb P_n^{\otimes N}(\Psi^{-1}_{n,t}(w))\leq 2^{Nr}\csli^N
		\sum_{\substack{
				\underline v\in B^{(t)}_{Nr,N(n+1)}\\
				G_t(\underline v)=w
		}}
		\prod_{j=1}^{\kv v}p(v^{j}).
	\end{equation*}
	Hence, applying Lemma \ref{lemma: proba ineq for the stitching map} with $k_*=Nr$ and $l_*=N(n+1)$, we obtain
	\begin{equation*}
		\mathbb P_n^{\otimes N}(\Psi_{n,t}^{-1}(w))\leq \frac{2^{Nr}\csli^N}{\csti}\mathbb P_{t+1}(w).
	\end{equation*}
	The bound \eqref{eq: proba ineq for the coupling map} is proved with 
	\begin{equation}
		\label{eqloc: CnN'} 
		C_{n,t}=\frac {\csti}{2^{Nr}\csli^N}
		=2^{-Nr}(n+1)^{-N(r+1)}{(N(n+1+\tau r))^{-2}}\beta(z_0)\eta^{Nr}\bigg(\frac {2r}{e(n+1+(\tau+2) r)}\bigg)^{2Nr}.
	\end{equation}
	It remains to show that $C_{n,t}$ satisfies condition~\eqref{eq: constant Cnt}. 
	Notice that $C_{n,t}$ is actually a function of variables $(n, N)$ that only depends on $t$ through the definition of $N$. 
	By~\eqref{eq: definition Nnt}, at fixed $n$, we have
	\begin{equation*}
		\begin{split}
			\lim_{t\to\infty}&\frac1t\log C_{n,t}
			=\frac{1}{n(1-4\delta)}\lim_{N\to\infty}\frac 1N\log C_{n,t}\\
			&=\frac{1}{n(1-4\delta)}\bigg(
			-r\log2-(r+1)\log (n+1)+r\log \eta +2r\log \frac {2r}{e(n+1+(\tau+2) r)} 
			\bigg).
		\end{split}
	\end{equation*}
	Further taking the limit as $n\to\infty$, we get~\eqref{eq: constant Cnt}.
\end{proof}
We can finally prove Proposition~\ref{prop: P is decoupled}, using the properties of the coupling map established above.
\begin{proof}[Proof of Proposition \ref{prop: P is decoupled}]
	Set $n_0$ to be an integer such that $4(r\tau+1)/n_0\leq \delta$. For all $n\in \N$, set $t_n$ to be an integer such that $n/t_n\leq \delta$. Let $\Kdelta$ be given by Lemma~\ref{lemma: alphabet of the coupling map}. Let $n\geq n_0$, $t\geq t_n$, and $\Kun\supseteq \Kdelta$.
	Denote by $\mathcal W$ the set $\closeword_{n,\Kun}(\mu_1,\delta)^{N_1}\times \closeword_{n,\Kun}(\mu_2,\delta)^{N_2}$. 
	Proposition~\ref{prop: P is decoupled} relies on the following statement: 
	\begin{equation}
		\label{eqloc: psi(W) subset W}
		\Psi_{n,t}(\mathcal W)\subseteq \closeword _{\Kun,t}(\mu, 20\delta),
	\end{equation}
	which we will prove below. Once \eqref{eqloc: psi(W) subset W} is proved, we have 
	\begin{equation*}
		\mathbb P_{t+1}\left(\closeword_{t,\Kun}(\mu,20\delta)\right)
		\geq \mathbb P_{t+1}\left(\Psi_{n,t}(\mathcal W)\right),
	\end{equation*}
	hence by Lemma \ref{lemma: proba ineq for the coupling map},
	\begin{equation*}
		\begin{split} 
			\mathbb P_{t+1}\left(\closeword_{t,\Kun}(\mu,20\delta)\right)
			&\geq C_{n,t}\mathbb P_{n+1}^{\otimes N}\left( \Psi_{n,t}^{-1} \left( \Psi_{n,t}(\mathcal W) \right) \right) \\
			&\geq C_{n,t}\mathbb P_{n+1}^{\otimes N}\left(\mathcal W\right).
		\end{split} 
	\end{equation*}
	This yields \eqref{eq: decoupled inequality for P}, with the constant $C_{n,t}$ from Lemma~\ref{lemma: proba ineq for the coupling map}, which satisfies \eqref{eq: constant Cnt}. It only remains to prove~\eqref{eqloc: psi(W) subset W}, which we do now.
	Let $\underline u=(u^1,\ldots ,u^N)\in \mathcal W$. Equation \eqref{eqloc: psi(W) subset W} encapsulates three properties:
	\begin{enumerate}
		\item The word $\Psi_{n,t}(\underline u)$ is well defined. 
		\label{itemloc: psint exists}
		\item $L[\Psi_{n,t}(\underline u)]$ is at most $20\delta$-close to $\mu$.
		\label{itemloc: geographic for psint}
		\item Every letter of $\Psi_{n,t}(\underline u)$ belongs to $\Kun$.
		\label{itemloc: letters of psint}
	\end{enumerate}
	Once the existence of $\Psi_{n,t}(\underline u)$ is proved, Property~\ref{itemloc: letters of psint} will be an immediate consequence of Lemma \ref{lemma: alphabet of the coupling map}. We now prove Properties~\ref{itemloc: psint exists} and~\ref{itemloc: geographic for psint} by showing that $\underline u$ satisfies the assumptions of Lemma \ref{lemma: geographic ineq for the coupling map}, namely that $\sum_{i=1}^NM[u^i](\Gamma)\geq t+1$. 
	Let $i\leq N$ and let $\nu_i$ denote $\mu_1$ or $\mu_2$ depending on whether $i\leq N_1$ or $i>N_1$.
	Since $\nu_i$ is admissible, and by \eqref{eq:defKgdelta}, we have $\nu_i(\Gamma)=\nu_i(K ^2)\geq 1-\delta$. Thus,
	\begin{equation*}
		L[u^i]\left(\Gamma\right)\geq \nu_i\left(\Gamma\right)-|L[u^i]-\nu_i|_\mathrm{TV}\geq 1-\delta -\delta.
	\end{equation*}
	Hence $M[u^i](\Gamma)\geq n(1-2\delta)$, and
	\begin{equation*}
		\sum_{i=1}^NM[u^i](\Gamma)\geq Nn(1-2\delta)\geq t+1. 
	\end{equation*}
	By Lemma \ref{lemma: geographic ineq for the coupling map}, Property~\ref{itemloc: psint exists} holds. Another consequence of Lemma~\ref{lemma: geographic ineq for the coupling map} is that
	\begin{align}
		\tvnorm{L[\Psi_{n,t}(\underline u)]-\mu}
		&\leq \frac1{t}({Nn-(t+1)+2N(r\tau+1)})+\tvnorml{\frac 1t\sum_{i=1}^NM[u^i]-\mu}.\label{eqloc: triangle ineq for the coupling map}
	\end{align}
	Using~\eqref{eq: bounds on delta an Nn/t} and $n\geq n_0$, the first term on the right-hand side of \eqref{eqloc: triangle ineq for the coupling map} above is no larger than $1+8\delta-1+\delta$. Since $\underline u\in \mathcal W$, the second term satisfies 
	\begin{align*}
		\tvnorml{\frac 1t\sum_{i=1}^NM[u^i]-\mu}
		&\leq
		\frac 1t\sum_{i=1}^{N_1}\tvnorm{M[u^i]-n\mu_1}
		+\frac 1t\sum_{i=N_1+1}^{N_2}\tvnorm{M[u^i]-n\mu_2}\\
		&\phantom{=\frac 1t\sum\tvnorm{M[ui]-n\mu1}}
		+\tvnorml {\frac {N_1}tn\mu_1+\frac {N_2}tn\mu_2-\mu}
		\\
		&\leq \frac{N_1}tn\delta+\frac{N_2}tn\delta+\bigg|\frac {N_1}tn-\frac 12\bigg|+\bigg|\frac {N_2}tn-\frac 12\bigg|\\
		&\leq (1+8\delta)\delta+ \frac 128\delta+\frac128\delta+\delta.
	\end{align*}
	Combining the bounds on the two terms the right-hand side of \eqref{eqloc: triangle ineq for the coupling map}, we obtain
	\begin{equation*}
		\tvnorm{L[\Psi_{n,t}(\underline u)]-\mu}\leq 20\delta,
	\end{equation*}
	and we have finally shown Property~\ref{itemloc: geographic for psint}. This completes the proof of~\eqref{eqloc: psi(W) subset W} and concludes the proof of the proposition. 
\end{proof}

\begin{remark}
	In irreducible cases, the construction of the coupling map is considerably simplified, yet not trivial. 
	Assume here that there is only one irreducible class $C_1$ in~\eqref{eq: decomposition of S in irreducible classes} (in particular, this assumption is satisfied when the Markov chain is $\varphi$-irreducible).
	Then, the slicing map is reduced to simply removing a prefix and a suffix from the argument. The use of the slicing map on $\underline u:=(u^1,\ldots, u^N)$ results in a stitchable sequence of at most $N$ words and the coupled word $\Psi_{n,t}(\underline u)$ is simply the prefix of length $t$ of the word
	\begin{equation*}
		z_0\xi^{1}v^{1}\xi^{2} \ldots \xi^{v} v^{k},
	\end{equation*}
	where $(v_1,\ldots,v^k)$ is the list $(F(u^1),\ldots, F(u^N))$ with empty words removed. 
	
	If $\spaces$ is finite and the Markov chain is matrix-irreducible, as in Example 2.20 of~\cite{papierdecouple}, we recover the map $\psi_{n,t}(u)$ from~\cite{papierdecouple}. Indeed, one can take $C_1=K_1=K=\spaces$. Then, the slicing map is the identity, and we have
	\begin{equation*}
		\Psi_{n,t}(\underline u)=z_0\xi(z_0,u^1_1)\psi_{n,t}(\underline u),\qquad\underline u\in (\spaces^n)^N.
	\end{equation*}
\end{remark}

\subsubsection{The decoupling map and proof of Proposition~\ref{prop: decoupling ineq for the decoupling map}}
\label{section: decoupling map}
The construction of the previous section achieved to prove Proposition~\ref{prop: P is decoupled}, which is already enough to derive Theorem~\ref{theorem: Ln2 satisfies a weak LDP with rate function -s} in Section~\ref{section: ruelle lanford functions}. 
The goal of the present section is to prove Proposition~\ref{prop: decoupling ineq for the decoupling map}, which is used in the proof of Property~\ref{item: linearity of I2 between separated measures} of Proposition~\ref{prop: properties of rate function -s}. 
To do so, we present another construction that uses slicing and stitching maps.

Let $\lambda_1$, $\lambda_2$, $\delta$, $t_1$, $t_2$, $\mu_1$, $\mu_2$ and $\mu$ be as in Proposition~\ref{prop: decoupling ineq for the decoupling map}.
For convenience, we also set
\begin{equation*}
	\delta_i=\frac{125}{\lambda_i^2}\delta,\qquad i\in \{1,2\}.
\end{equation*}
Note that, by the choice of $\delta$ and the definition of $t_i$, if $n\geq 1/\delta$, we have the following two bounds for $i\in \{1,2\}$:\footnote{
	The first bound is the reason why we required $\delta$ to be no larger than $\lambda_i/12$.}
\begin{equation}
	\label{eq: remark bounds on n/ti and n-t1-t2}
	\frac1{\lambda_i}\leq \frac{n}{t_i}
	\leq \frac{1}{\lambda_i}\Big(1+\frac4{\lambda_i}\delta\Big),
	\qquad
	0\leq n- t_1-t_2\leq 8n\delta.
\end{equation}
Let $K\subseteq \spaces$ be a finite set such that $\mu(K^2)>1-\delta$.
We set
\begin{equation*}
	\mathcal J_i=\mathcal J_K\cap {\mathcal J_{\mu_i}},
	\qquad \Gamma_i=\bigcup_{j\in \mathcal J_i}K_j ^2,
	\qquad i\in\{1,2\},
\end{equation*}
where $\mathcal J_K=\{j\in j\ |\ K\cap C_j\neq \emptyset\}$.
We set $\mathcal J_0:=\mathcal J_1\cup \mathcal J_2$ and $r:=|\mathcal J_1\cup\mathcal J_2|<\infty$. 
For all $j\in \mathcal J_0$, let $K_j=K\cap C_j$.
In the following, when $\mathcal J'\subseteq \mathcal J$, and $\nu$ is a measure, we denote\label{not: muJ}
\begin{equation*}
	\nu|_{\mathcal J'}=\nu|_{\bigcup_{j\in{\mathcal J'}}C^2_j}.
\end{equation*}

In order to prove Proposition~\ref{prop: decoupling ineq for the decoupling map}, we introduce the \emph{decoupling map}. 
It acts in the opposite way of the coupling map: while the coupling map takes words approximating $\mu_1$ and $\mu_2$ and creates a word approximating $\mu$, the decoupling map takes a word approximating $\mu$ and creates two words approximating $\mu_1$ and $\mu_2$ respectively (given that $\mu = \lambda_1\mu_1 + \lambda_2\mu_2$ and $\mu_1, \mu_2$ are supported on disjoint classes).
\begin{definition}[The decoupling map] 
	\label{def: the decoupling map}
	Let $u\in\spaces^{n+1}\cap\wordspos$. Applying the slicing maps $F_{\mathcal J_1}$ and $F_{\mathcal J_2}$ from Definition~\ref{def: the slicing map} to $u$ defines two sequences $\underline v_1$ and $\underline v_2$ of subwords of $u$ indexed by $\mathcal J_1$ and $\mathcal J_2$, respectively. Both lists $\sigma(\underline v_1)$ and $\sigma(\underline v_2)$ are stitchable, where $\sigma$ is as in Lemma~\ref{lemma: reoredring the list}.\footnote{Here the action of $\sigma$ is simply to remove empty words, since the lists are already in the right order.} If $|\underline v_1|\geq t_1+1$ and $|\underline v_2|\geq t_2+1$, we define the words
	\begin{equation*}
		\widetilde\Psi_{1,n}(u)=G_{t_1}(\sigma(\underline v_1)),\qquad \widetilde\Psi_{2,n}(u)=G_{t_2}(\sigma(\underline v_2)),
	\end{equation*}
	where $G_{t_i}$ is the stitching map from Definition~\ref{def: the stitching map},
	and we call $\widetilde\Psi_n: u\mapsto \left(\widetilde \Psi_{1,n}(u),\widetilde \Psi_{2,n}(u)\right)$ the decoupling map.
\end{definition}
Beware that the decoupling map is by no means the inverse of the coupling map.
Figure~\ref{figure: decoupling map sketch} provides a sketch of the definition of the decoupling map.
\begin{figure}[thb]
	\centering
	\includegraphics{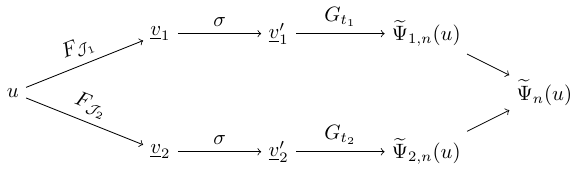}
	\caption{The definition of the decoupling map as composition of slicing and stitching maps.}
	\label{figure: decoupling map sketch}
\end{figure}
\begin{lemma}
	\label{lemma: geographic ineq for the decoupling map}
	Let $u\in\words$ be such that $M[u](\Gamma_1)\geq t_1+1$ and $M[u](\Gamma_2)\geq t_2+1$. Then, for $i\in \{1,2\}$, $\widetilde \Psi_{i,n}(u)$ is well defined and
	\begin{equation}
		\label{eq: geographic ineq for the decoupling map}
		\begin{split}
			\tvnorm{M\left[\widetilde\Psi_{i,n}(u)\right]-M[u]|_{\mathcal J_i}}&\leq 2\big(|u|-1-t_1-t_2\big)+2r\tau.
		\end{split}
	\end{equation}	
\end{lemma}
\begin{proof}
	Let $i\in\{1,2\}$. 
	Let $\underline v'_i=\sigma(F_{\mathcal J_i}(u))$. Since $\underline v'_i$ is stitchable, we only have to verify that $|\underline v'_i|\geq t_i+1$ in order for $\widetilde \Psi_{i,n}(u)$ to be well defined. Using the reverse triangular inequality, the bound~\eqref{eq: geographic ineq for the slicing map} in Lemma~\ref{lemma: geographic ineq for the slicing map} applied to $F_{\mathcal J_i}$ yields
	\begin{equation*}
		\tvnorm{M[u]}-\sum_{j\in\mathcal J_i}\tvnorm{M[u^j]}\leq M[u](\spaces^2\backslash\Gamma_i)\leq n-(t_i+1),
	\end{equation*}	
	so that indeed, we have in particular $|\underline v'_i|\geq \sum_{j\in\mathcal J_i}\tvnorm{M[u^j]}\geq t_i+1$.
	We now prove~\eqref{eq: geographic ineq for the decoupling map}.
	Since $M[u]|_{\mathcal J_i}$ is the sum of $M[u]|_{C_j^2}$ over $j\in \mathcal J_i$,
	the bound~\eqref{eq: detailed geographic ineq for the slicing map} of Lemma~\ref{lemma: geographic ineq for the slicing map} yields
	\begin{equation}
		\label{eqloc: MuJi -sum Muj}
		\tvnorml{M[u]|_{\mathcal J_i}-\sum_{j\in \mathcal J_i}M[u^j]}
		=\sum_{j\in\mathcal J_i}\tvnormm{M[u]|_{C_j ^2}-M[u^j]}
		\leq \sum_{j\in \mathcal J_i}M[u](C_j^2)-M[u](\Gamma_i).
	\end{equation} 
	The elements of $\spaces^2$ that are in a $C_j^2$ for some $j\in\mathcal J_i$ are not in any $K_{j'}^2$ with $j'\in \mathcal J_{3-i}$. Hence we have a bound on the first term of the right-hand side of~\eqref{eqloc: MuJi -sum Muj}:
	\begin{align*}
		\sum_{j\in \mathcal J_i}M[u](C_j^2)
		&\leq (|u|-1)-\sum_{j'\in \mathcal J_i}M[u](K_{j'}^2)
		\leq |u|-t_{3-i}.
	\end{align*}
	Also bounding the second term of the right-hand side of~\eqref{eqloc: MuJi -sum Muj} by $-(t_i+1)$, we have 
	\begin{equation}
		\label{eqloc: geographic decoupling ineq part 1}
		\tvnorml{M[u]|_{\mathcal J_i}-\sum_{j\in \mathcal J_i}M[u^j]}\leq |u|-t_{1}-t_2-1.
	\end{equation}
	Moreover, by Lemma \ref{lemma: geographic ineq for the stitching map}, 
	\begin{align}
		\tvnorml{M[\widetilde{\Psi}_{i,n}(u)]-\sum_{j\in\mathcal J_i}M[u^j]}&\leq |\underline v'_i|-(t_i+1)+2r\tau
		\nonumber\\
		&\leq |u|-(t_{1}+1)-(t_2+1)+2r\tau.
		\label{eqloc: geographic decoupling ineq part 2}
	\end{align}
	The last inequality holds because $|\underline v'_1|+|\underline v'_{2}|\leq |u|$. Combining~\eqref{eqloc: geographic decoupling ineq part 1} and~\eqref{eqloc: geographic decoupling ineq part 2}, we obtain~\eqref{eq: geographic ineq for the decoupling map}.
\end{proof}
\begin{lemma}
	\label{lemma: proba ineq for the decoupling map}
	For all $(w^1,w^2)\in \spaces^{t_1+1}\times \spaces^{t_2+1}$, 
	\begin{equation}
		\label{eq: proba ineq for the decoupling map}
		\mathbb P_{t_1+1}(w^1)\mathbb P_{t_2+1}(w^2)\geq \widetilde C_n\mathbb P_{n+1}(\widetilde \Psi_{n}^{-1}(w^1,w^2)),
	\end{equation}
	for some constants $\widetilde C_n$ satisfying \eqref{eq: constant tilde Cn}.
\end{lemma}
\begin{proof}If $(w^1,w^2)$ is not in the range of $\widetilde \Psi_{n}$, \eqref{eq: proba ineq for the decoupling map} is trivial, so we now assume that $(w^1,w^2)$ is in the range of $\widetilde \Psi_{n}$. For convenience, we define, for $i\in\{1,2\}$,
	\begin{align*}
		\mathcal V_i&=\{\underline v\in \wordspos^{ \mathcal J_i}\ |\ G_{t_i}(\sigma(\underline v))=w^i, \ \forall j\in \mathcal J_i,\,v^j\in C_j^{|v^j|}\},\\
		\mathcal V'_i&= \{\underline v'\in \wordspos^{ \mathcal J'}\ |\ \mathcal J'\subseteq \mathcal J_i,\ G_{t_i}(\underline v')=w^i, \ \forall j\in \mathcal J',\,e\neq v^j\in C_j^{|v^j|}\}.
	\end{align*}
	Notice that $\sigma $ is bijective between $\mathcal V_i$ and $\mathcal V'_i$ since its action only consists in removing empty words from the list.
	Let $u\in \widetilde{\Psi}_n^{-1}(w^1,w^2)$.
	There are $\underline v_1 \in \mathcal V_1$ and $\underline v_2 \in \mathcal V_2$ such that $F_{\mathcal J_1}(u)=\underline v_1$ and $F_{\mathcal J_2}(u)=\underline v_2$. Notice that then, $F_{\mathcal J_0}(u)$ is the list $\underline v = (v^j)_{j\in \mathcal J_0}$ obtained by combining and reordering the (possibly empty) words of $\underline v_1$ and $\underline v_2$. Since $\mathcal J_1 \cap \mathcal J_2=\emptyset$, the map $h:(\underline v_1, \underline v_2)\mapsto \underline v$ is a bijection.
	It follows that 
	\begin{equation*}
		\mathbb P_{n+1}(\Psi_n^{-1}(w^1,w^2))=\sum_{{\substack{\underline v_1\in \mathcal V_1\\\underline v_2\in \mathcal V_2}}}
		\mathbb P_{n+1}(F_{n}^{-1}(h(\underline v_1,\underline v_2))),
	\end{equation*}
	where $F_n$ is the restriction of $F_{\mathcal J_0}$ to $\wordspos\cap \spaces^{n+1}$ as in Lemma~\ref{lemma: proba ineq for the slicing map}. By Lemma \ref{lemma: proba ineq for the slicing map} applied to each term of the sum, with $(v^j)_{j\in \mathcal J_0}:=h(\underline v_1,\underline v_2)$, we have
	\begin{align*}
		\mathbb P_n(\widetilde \Psi_n^{-1}(w^1,w^2))
		&\leq \csli \sum_{{\substack{\underline v_1\in \mathcal V_1\\\underline v_2\in \mathcal V_2}}}\prod _{j\in \mathcal J_0}p(v^j)
		\\&=\csli\Bigg(\sum_{\underline v_1\in \mathcal V_1}\prod _{j\in\mathcal J_1}p((\underline v_1)^j)\Bigg)\Bigg( \sum_{\underline v_2\in\mathcal V_2}\prod _{j\in\mathcal J_2}p((\underline v_2)^j)\Bigg)
		\\&=\csli\Bigg(\sum_{\underline v_1'\in\mathcal V'_1}\prod _{j=1}^{\kv{v_1'}}p((\underline v_1')^j)\Bigg)\Bigg( \sum_{\underline v_2'\in\mathcal V'_2}\prod _{j=1}^{\kv {v_2'}}p((\underline v_2')^j)\Bigg).
	\end{align*}
	We have $\mathcal V'_i\subseteq G_{t_i}^{-1}(w^i)$ for $i\in \{1,2\}$, so we can crudely bound the two last sums using Lemma~\ref{lemma: proba ineq for the stitching map}.
	By Lemma~\ref{lemma: proba ineq for the stitching map} with $k_*=r$ and $l_*=n+1$. We get
	\begin{equation*}
		\mathbb P_n(\Psi^{-1}_n(w_1,w_2))\leq \frac{\csli}{\csti^2}\mathbb P_{t_1}(w^1)\mathbb P_{t_2}(w^2).
	\end{equation*}
	The bound~\eqref{eq: proba ineq for the decoupling map} is proved with $\widetilde C_n=\csti^2/\csli$. It remains to show~\eqref{eq: constant tilde Cn}. The constant $\widetilde C_n$
	satisfies
	\begin{align*}
		\frac 1n\log\widetilde C_n
		=& 
		\frac {r+1}n\log (n+1)-\frac 4n\log (n+1+\tau r)\\&+\frac 2n\log\beta(z_0)+\frac {2r}n\log (\eta) +\frac {4r}n\log \frac {2r}{e(n+1+(\tau+2) r)}.
	\end{align*}
	This quantity vanishes in the limit as $n\to \infty$.
\end{proof}
\begin{proof}[Proof of Proposition~\ref{prop: decoupling ineq for the decoupling map}]
	Let $n_0$ be an integer such that $(2r\tau/t_1)\vee(2r\tau/t_2)\leq \delta$, with $\tau$ as in Lemma~\ref{lemma: stitching words}, and let $n\geq n_0$.
	We will show that 
	\begin{equation}
		\label{eqloc: inclusion for the decoupling map}
		\widetilde \Psi_n(\closeword_n(\mu,\delta))\subseteq \closeword_{t_1}(\mu_1,\delta_1)\times\closeword_{t_2}(\mu_2,\delta_2).
	\end{equation}
	Once \eqref{eqloc: inclusion for the decoupling map} has been proved, we can apply Lemma~\ref{lemma: proba ineq for the decoupling map} to obtain
	\begin{align*}
		\mathbb P_{t_1+1}(\closeword_{t_1}(\mu_1,\delta_1))\mathbb P_{t_2+1}(\closeword_{t_2}(\mu_2,\delta_2))
		&\geq\mathbb P_{t_1+1}\otimes \mathbb P_{t_2+1}({\widetilde \Psi}_n(\closeword_n(\mu, \delta)))\\
		&\geq \widetilde C_{n}\mathbb P_{n+1}(\closeword_n(\mu,\delta)).
	\end{align*}
	This gives us~\eqref{eq: decoupling ineq for the decoupling map} where the constant $\widetilde C_n$ satisfies \eqref{eq: constant tilde Cn}. 
	To check that~\eqref{eqloc: inclusion for the decoupling map} holds, we first need to make sure that $\widetilde \Psi_n(u)$ is well defined for all $u\in \closeword_n(\mu,\delta)$: since
	\begin{equation}
		\label{eqloc: L[u](Gammai)}
		L[u](\Gamma_i)\geq \mu(\Gamma_i)-\tvnorm{L[u]-\mu}\geq \lambda_i-\delta -\delta,
	\end{equation}	
	we have $M[u](\Gamma_i)\geq n(\lambda_i-2\delta)\geq t_i+1$, thus by Lemma~\ref{lemma: geographic ineq for the decoupling map}, $\widetilde{\Psi}_{n}(u)$ is well defined.
	We now prove~\eqref{eqloc: inclusion for the decoupling map}.
	Let $u\in \closeword_n(\mu,\delta)$ and $i\in \{1,2\}$. We have to compare $L\left[\widetilde\Psi_{i,n}(u)\right]$ and $\mu_i$. We shall first establish a few preliminary bounds.
	By~\eqref{eqloc: L[u](Gammai)}, we have 
	\begin{equation*}
		L[u]\bigg(\bigcup_{j\in \mathcal J_{\mu}}C_j^2\bigg)\geq 	L[u](\Gamma_1\cup\Gamma_2)\geq 1-4\delta.
	\end{equation*}
	Hence, since $\mathcal J_{\mu_1}$ and $\mathcal J_{\mu_2}$ are disjoint,
	\begin{equation*}
		\begin{split}
			\tvnorm{\lambda_1\mu_1-L[u]|_{\mathcal J_{\mu_1}}}
			+\tvnorm{\lambda_2\mu_2-L[u]|_{\mathcal J_{\mu_2}}}
			&=\tvnorm{\mu-L[u]|_{\mathcal J_{\mu_1}\cup \mathcal J_{\mu_2}}}
			\\&\leq \delta + 4\delta,
		\end{split}
	\end{equation*}
	thus both terms are no larger than $5\delta$. In particular, we have $L[u]|_{\mathcal J_{\mu_i}}(A)\leq \lambda_i\mu_i(A)+5\delta$ for all $A\subseteq \spaces^2$. Together with~\eqref{eqloc: L[u](Gammai)}, this implies 
	\begin{align*}
		\tvnorm{L[u]|_{\mathcal J_{\mu_i}}-L[u]|_{\mathcal J_{i}}}
		&=L[u]\bigg(\bigcup_{j\in\mathcal J_{\mu_i}}C_j^2\bigg)-L[u]\bigg(\bigcup_{j\in\mathcal J_{i}}C_j^2\bigg)
		\\&\leq \lambda_i+5\delta-(\lambda_i-2\delta)
		\\&=7\delta.
	\end{align*}
	We can now compare $L\left[\widetilde\Psi_{i,n}(u)\right]$ and $\mu_i$. 
	On the one hand, by Lemma~\ref{lemma: geographic ineq for the decoupling map},
	\begin{align*}
		\tvnorm{L\left[\widetilde\Psi_{i,n}(u)\right]-\frac n{t_i}L[u]|_{\mathcal J_i}}
		&\leq \frac 2{t_i}\big(n-t_1-t_2\big)+\frac{2r\tau}{t_i}
		\\&\leq \frac {n}{t_i}16\delta+\delta.
	\end{align*}
	The bound on $n-t_1-t_2$ is obtained thanks to the choice of $\delta$ and $n$ that enabled~\eqref{eq: remark bounds on n/ti and n-t1-t2}. The choice of $n\geq n_0$ also enabled the bound on $2r\tau/t_i$.
	On the other hand, 
	\begin{align*}
		\tvnorm{\frac n{t_i}L[u]|_{\mathcal J_i}-\mu_i}
		&\leq \frac n{t_i}\tvnorm{L[u]|_{\mathcal J_i}-L[u]|_{\mathcal J_{\mu_i}}}+\tvnorml{\Big(\frac1{\lambda_i}-\frac n{t_i}\Big)L[u]|_{\mathcal J_{\mu_i}}}
		\\&\qquad \qquad \qquad \qquad\qquad \qquad \qquad +\frac 1{\lambda_i}\tvnorm{L[u]|_{\mathcal J_{\mu_i}}-\lambda_i\mu_i}
		\\&\leq\frac n{t_i}7\delta +\frac {4}{\lambda^2_i}\delta +\frac{1}{\lambda_i}5\delta.
	\end{align*}
	The term $4\delta/\lambda_i ^2$ in the middle is obtained thanks to~\eqref{eq: remark bounds on n/ti and n-t1-t2}. The two other therms are obtained thanks to the preliminary bounds stated above in the proof.
	Combining these two bounds yields
	\begin{align*}
		\tvnorml{L\left[\widetilde\Psi_{i,n}(u)\right]-\mu_i}
		&\leq \frac{n}{t_i}23\delta+\frac {4}{\lambda^2_i}\delta +\frac{1}{\lambda_i}5\delta+\delta
		\\&\leq \frac {96}{\lambda^2_i}\delta +\frac{1}{\lambda_i}28\delta+\delta.
	\end{align*}
	The bound on $n/t_i$ is obtained thanks to~\eqref{eq: remark bounds on n/ti and n-t1-t2}.
	Roughly bounding both $1/\lambda_i$ and $1$ by $1/\lambda_i^2$, we get the bound $125\delta/\lambda_i^2$, which proves~\eqref{eqloc: inclusion for the decoupling map} and completes the proof.
\end{proof}
	\section{Identification of the rate function}
\label{section: identification rate function}
Theorem~\ref{theorem: intro weak LDP for Ln2} was only partially proven in the previous section; Theorem~\ref{theorem: Ln2 satisfies a weak LDP with rate function -s} states that $(\ith L2_n)$ satisfies the weak LDP with rate function $\Ide$, but it remains to prove~\eqref{eq: rate function I2}.
This is the goal of the current section.
For convenience, we set $ I:=\Ide$ in this section.
\subsection{Usual rate functions} 
\label{section: usual rate functions}
Let us properly define the functions involved in~\eqref{eq: rate function I2}.
In the following, $\mathbb E[\cdot]$ denotes the expectation associated to the law of $(X_n)$.\label{not: E} Fix $\mu\in\proba$. Recall that $\mathcal M(\spaces^2)$ denotes the space of finite signed measures on $\spaces^2$ and that $\bounded(\spaces^2)$ denotes the set of bounded functions on $\spaces ^2$. In addition, $\expbounded(\spaces^2):=\exp(\bounded(\spaces^2))$ denotes the set of positive functions on $\spaces^2$ that are bounded away from $0$ and $\infty$. The supremum norm on $\bounded(\spaces^2)$ is denoted $\|\cdot\|$.\footnote{We do not equip $\bounded(\spaces^2)$ with the topology induced by $\|\cdot\|$. See Appendix~\ref{section: duality}.}\label{not: uniform norm} The SCGF $\Lambda=\ith \Lambda 2$ of $(\ith L2_n)$ is defined by\label{not: Lambda}\label{not: expbounded}
\begin{equation}
	\label{eq: definition Lambda}
	\Lambda (V)=\limsup_{n\to \infty}\frac 1n\log\mathbb E\left[e^{n\langle L_n,V\rangle}\right],\qquad V\in \bounded(\spaces^2).
\end{equation}
We also define the truncated SCGF $\Lambda_\infty=\ith \Lambda2_\infty$ of $(\ith L2_n )$ by
\begin{equation}
	\Lambda_K (V)=\limsup_{n\to \infty}\frac 1n\log\mathbb E\left[e^{n\langle L_n,V\rangle}\indic_{\mathcal P(K ^2)}(L_n)\right],\qquad V\in \bounded(\spaces^2),\ K\subseteq \spaces\hbox{ finite},
\end{equation}
and\label{not: Lambda infty}
\begin{equation}
	\Lambda_\infty(V)=\sup\{\Lambda_K(V)\ |\ K\subseteq \spaces,\ |K|<\infty\}.
\end{equation}
Clearly, $\Lambda_K\leq\Lambda_\infty\leq \Lambda$. See Example~\ref{ex: right only random walk} for a situation in which $\Lambda_\infty$ and $\Lambda$ are not equal.
Their convex conjugates $\Lambda^*$ and $\Lambda_\infty^*$ are defined by Definition~\ref{def: convex conjugate} and are lower semicontinuous.\label{not: Lambda *}\label{not: Lambda infty *}

As defined in~\cite{DV1}, the DV entropy $\idv=\ith \idv 2$ is given by\footnote{ 
	The supremum in~\eqref{eq: definition DV entropy} is often taken over $\{\uf\in \expbounded(\spaces^2)\ |\ \uf\geq1\}$ instead of $\expbounded(\spaces^2)$. In fact, when $\uf\in \expbounded(\spaces^2)$, replacing $f$ by $f/\inf f$ does not change the value of the bracket in~\eqref{eq: definition DV entropy}, hence the two definitions agree.
}\label{not: IDV}
\begin{equation}
	\label{eq: definition DV entropy}
	\idv(\mu)=\sup_{\uf\in\expbounded(\spaces^2)}\Big\langle \mu, \log\frac{\uf}{P\uf}\Big\rangle, \qquad \mu\in \balanced (\spaces^2),
\end{equation} 
where\footnote{Any stochastic kernel $r$ on a space $\mathcal X$ defines an operator on $\mathcal E(\mathcal X)$ by $rf(x)=\sum_{y\in \mathcal X}r(x,y)f(y)$. The quantity $Pf(x,y)$ in~\eqref{eq: definition operator p} has the same definition with $\mathcal X=\spaces^2$ and $r$ the stochastic kernel of the Markov chain $((X_n,X_{n+1}))_{n\geq 1}$, defined by $r((x_1,y_1),(x_2,y_2))=\indic_{\{y_1=x_2\}}p(x_2,y_2)$. Notice that $Pf(x,y)$ does not depend on $x$.}
\begin{equation}
	\label{eq: definition operator p}
	Pf(x,y):=\ith P2f(x,y)=\sum_{z\in \spaces}p(y,z)f(y,z),\qquad (x,y)\in \spaces^2,\ \uf \in \expbounded(\spaces^2).
\end{equation}
The function $\idv$ is the standard (see for instance \cite{DZ,rassoul,deacosta2022}) DV entropy for the Markov chain $((X_n,X_{n+1}))_{n\geq 1}$. The function $J$ is lower semicontinuous as a pointwise supremum of a family of continuous functions.

If $\mu,\nu\in \mathcal P(\spaces^k)$ for some $k\in \N$, we denote by $\ent (\mu|\nu)$ the relative entropy of $\mu$ with respect to $\nu$, which is defined by\label{not: relative entropy}
\begin{equation}
	\label{eq: definition relative entropy S(mu|nu)}
	\ent(\mu|\nu)=
	\begin{cases}
		\sum_{u\in \spaces^k}\mu(u)\log\frac{\mu(u)}{\nu(u)},\quad&\hbox{if $\mu\ll \nu$},
		\\\infty,\quad &\hbox{otherwise}.
	\end{cases} 
\end{equation}
In this definition and throughout the rest of this paper, we adopt the conventions that $0\times\log 0=0$.\
The fact that $\ent (\mu|\nu)$ is well defined in $[0,+\infty]$ comes from properties of the real function $s\mapsto s\log s $, namely that it is bounded below and convex. Additional properties of the relative entropy can be found in Chapter~5 of~\cite{rassoul}.
We use $\ent (\cdot|\cdot)$ to define a function $R=\ith R2$ in the following way. Let $\mu\in\balanced (\spaces^2)$. When $r$ is a probability kernel on $\spaces$, 
we define the probability measure $\ith \mu 1 \otimes r$ on $\alphabet$ by $\ith \mu 1\otimes r(x,y)=\ith \mu 1(x)r(x,y)$.\label{not: muotimeq}
We define\label{not: R}
\begin{equation}
	\label{eq: definition R}
	R(\mu)= \ent (\mu|\ith \mu1\otimes p), \quad \mu\in \balanced (\spaces^2).
\end{equation}
By the variational formula for the relative entropy (see for instance Theorem~5.4 of~\cite{rassoul}), we also have, for all $\mu \in \mathcal P(\spaces^2)$,
\begin{equation*}
	R(\mu)=\sup_{V\in \bounded(\spaces^2)}\big(\langle \mu,V\rangle -\log \langle\ith \mu 1\otimes p, e^ V\rangle \big).
\end{equation*}
This implies in particular that $R$ is lower semicontinuous, since both the functions $\mu \mapsto\langle \mu,V\rangle$ and $\mu\mapsto  \langle \ith \mu 1\otimes p,e^ V\rangle$ are continuous for a fixed $V$.
Let $q$ be the stochastic kernel defined by $q(x,y)={\mu(x,y)}/{\ith \mu 1(x)}$ if $\ith \mu 1(x)>0$, and arbitrarily otherwise. It follows that $\ith \mu 1 \otimes q=\mu$.
If $\mu\in \ith \abscont 2$, then $\mu \ll \ith \mu 1\otimes p$ and we have the convenient formula
\begin{equation} 
	\label{eq: expression R2}
	R(\mu)=\sum_{x,y\in \spaces}\mu(x,y)\log \frac{q(x,y)}{p(x,y)}=\sum_{x\in\spaces}\ith \mu 1(x)\ent (q(x,\cdot)|p(x,\cdot)).
\end{equation}

As mentioned earlier, the convex conjugate of the SCGF, the DV entropy, and the function $\ith R2$ are standard in large deviations theory of Markov chains. For convenience, until the end of Section~\ref{section: identification rate function}, we will omit the exponent and simply write $\Lambda$, $\Lambda_\infty$, $\idv$ and $R$ instead of $\ith \Lambda2$, $\ith \Lambda2_\infty$, $\ith \idv2$ and $\ith R2$ respectively.
We propose a complete identification of $I$.
\begin{proposition}
	\label{prop: identification I2}
	The function $I$ satisfies
	\begin{equation}
		\label{eq: rate function I2 for Ln2}
		I(\mu)=
		\begin{cases}
			\Lambda^*_\infty(\mu)=\Lambda^*(\mu)=\idv(\mu)=R(\mu),\quad &\hbox{if $\mu\in \Adebal$},\\
			\infty,&\hbox{otherwise}.
		\end{cases}
	\end{equation}
\end{proposition}
\begin{proof}
	Property~\ref{item: I2 infinite outside Adebal} of Lemma~\ref{prop: properties of rate function -s} already showed that $\irl$ is infinite on the complement of $\Adebal$. The description of $I$ over $\Adebal$ requires a more involved proof. The conclusion comes as consequence of several propositions.
	\begin{itemize}
		\item By Proposition~\ref{prop: I=I**}, the function $I$ coincides with its own convex biconjugate $I^{**}$ on $\Ade$.
		\item In Proposition~\ref{prop: I = Lambda*}, we show that $I^*=\Lambda_\infty$ on $\bounded(\spaces^2)$.
		Hence, by the previous point, we have $I=I^{**}=\Lambda_\infty^*$ on $\Adebal$.
		\item Finally, Proposition~\ref{prop: equality of all rate functions on Adebal}, establishes that
		\begin{equation}
			\label{eq: equality of all rate functions on Adebal}
			\Lambda_\infty^*(\mu)=\Lambda^*(\mu)=\idv(\mu)=R(\mu),
			\qquad \mu\in \Adebal.
		\end{equation}
	\end{itemize} 
	Put together, these equalities yield the conclusion. We prove Propositions~\ref{prop: I=I**}, \ref{prop: I = Lambda*}, and~\ref{prop: equality of all rate functions on Adebal} in Sections~\ref{section: convexity}, \ref{section: varadhan}, and~\ref{section: alternative expressions of the rate function} below, respectively.
\end{proof}
\begin{remark}
	\label{remark: I infinite on Adebal}
	Note that Proposition~\ref{prop: identification I2} does not say that $I$ is finite on $\Adebal$. For instance, if a measure $\mu\in \Adebal$ does not belong to $\Adebal \cap\ith \abscont 2$, one can easily check that Equation~\eqref{eq: rate function I2 for Ln2} holds with 
	\begin{equation*}
		I(\mu)=\Lambda^*_\infty(\mu)=\Lambda^*(\mu)=\idv(\mu)=R(\mu)=\infty.
	\end{equation*}
	See Example~\ref{ex: infinite entropy} for a non-trivial example of a measure that belongs to $\Adebal\cap \ith \abscont2$ while having infinite rate function.
\end{remark}
\begin{corollary}
	\label{coro: I affine between classes}
	Let $\mu\in \Adebal$, and let $\tilde \mu_j=\mu(C_j^2)^{-1}\mu|_{C_j^2}$ for all $j\in \mathcal J_\mu$. We have
	\begin{equation}
		\label{eq: I affine between classes}
		I(\mu)=\sum_{j\in \mathcal J_\mu}\mu(C_j^2)I(\tilde\mu_j).
	\end{equation}
\end{corollary}
\begin{proof}
	Since the function $R$ satisfies~\eqref{eq: I affine between classes}, this is a direct consequence of Proposition~\ref{prop: identification I2}.
\end{proof}
\subsection{The convex biconjugate of $I$}
\label{section: convexity}
Let $I^{**}$ denote the convex biconjugate of $I$, defined on $\mathcal P(\spaces^2)$ as in Appendix~\ref{section: duality}.
In this section, we prove the following property of the rate function $I$.
\begin{proposition}
	\label{prop: I=I**}
	Let $\mu\in \Ade$. Then,
	\begin{equation}
		\label{eq: I=I**}
		I(\mu)=I^{**}(\mu).
	\end{equation}
\end{proposition}
The identification $I=I^{**}$ does not hold on the full space $\mathcal P(\spaces^2)$, but rather on the smaller set of admissible measures. In fact, $I^{**}$ is convex, whereas we do not expect $I$ to be, by Property~\ref{item: I2 infinite outside Adebal} of Proposition~\ref{prop: properties of rate function -s}. 
Let $\conv I$ denote the convex envelope of $I$, which is the largest convex function under $I$ and satisfies
\begin{equation}
	\label{eq: def conv I}
	\begin{split}
		\conv I(\mu)=\inf\bigg\{\sum_{i=1}^k \lambda_iI(\mu_i)\ \Big|\ k\in \N,\ (\mu_i)\in& \proba^k, \ (\lambda_i)\in (0,1)^k,\\& \sum_{i=1}^k\lambda_i=1,\ 
		\sum_{i=1}^k \lambda_i\mu_i=\mu\bigg\},
		\qquad \mu\in \mathcal P(\spaces^2).
	\end{split}
\end{equation}
In the proof of Proposition~\ref{prop: I=I**}, we need properties of $I$ and $\conv I$, stated in Lemmas~\ref{lemma: I = sigmaco I} and~\ref{lemma: I** = sigmaco I} below.
\begin{proof}
	Let $\mu \in \Ade$.
	By Corollary~\ref{coro: fenchel moreau}, $I^{**}$ is the lower semicontinuous envelope of $\conv I$.
	By Lemma~\ref{lemma: I** = sigmaco I}, $\conv I$ agrees with its lower semicontinuous envelope on $\Ade$, thus $I^{**}(\mu)=\conv I(\mu)$. 
	Lemma~\ref{lemma: I = sigmaco I} yields the conclusion.
\end{proof}
\begin{lemma}
	\label{lemma: I = sigmaco I}
	Let $\mu\in \Ade$. Then, $I(\mu)=\conv I(\mu)$.
\end{lemma}
\begin{lemma}
	\label{lemma: I** = sigmaco I}
	Let $\mu\in \Ade$.
	Then, the function $\conv I$ is lower semicontinuous at $\mu$.
\end{lemma}
\begin{proof}[Proof of Lemma~\ref{lemma: I = sigmaco I}]
	By definition, $\conv I(\mu)\leq I(\mu)$. Let us prove the reverse bound.
	Let $\delta>0$. By~\eqref{eq: def conv I}, there exist finite sequences $(\lambda_i)\in (0,1)^k$ and $(\mu_i)\in\proba^k$ such that 
	\begin{equation*}
		\sum_{i=1}^k \lambda_i=1,
		\quad \sum_{i=1}^k\lambda_i\mu_i=\mu,
		\quad \conv I(\mu)+\delta \geq \sum_{i=1}^k\lambda_i I(\mu_i).
	\end{equation*}
	Each $\mu_i$ must be absolutely continuous with respect to $\mu$. Hence, each $\mu_i$ is admissible and satisfies $\mathcal J_{\mu_i}\subseteq \mathcal J_\mu$. By Property~\ref{item: midpoint convexity of I2} of Proposition~\ref{prop: properties of rate function -s} and Jensen's inequality, we have 
	\begin{equation*}
		\conv I(\mu)+\delta\geq I\bigg(\sum_{i=1}^k \lambda_i\mu_i\bigg)=I(\mu).
	\end{equation*}
	This holds for every $\delta>0$, thus $\conv I(\mu)\geq I(\mu)$.
\end{proof}
\begin{proof}[Proof of Lemma~\ref{lemma: I** = sigmaco I}]
	Let $(\mu_n)$ be a sequence of probability measures such that $\mu_n\to \mu$. We have to prove that
	\begin{equation}
		\label{eq: semicontinuity of conv I}
		\liminf_{n\to\infty} \conv I(\mu_n)\geq \conv I(\mu).
	\end{equation}
	This is immediate when the left-hand side is infinite. Otherwise, by taking a subsequence if necessary, we can assume that $\conv I(\mu_n)<\infty$ for all $n\in \N$ without loss of generality. By~\eqref{eq: def conv I}, for all $n\in \N$, there exist finite sequences $(\lambda_{n,i})\in (0,1)^{k_n}$ and $(\mu_{n,i})\in \proba^{k_n}$ such that 
	\begin{equation}
		\label{eqloc: lower bound for conv I}
		\sum_{i=1}^{k_n}\lambda_{n,i}=1,
		\qquad \sum_{i=1}^{k_n}\lambda_{n,i}\mu_{n,i}=\mu_n,
		\qquad \conv I(\mu_n)+\frac 1n\geq\sum_{i=1}^{k_n}\lambda_{n,i}I(\mu_{n,i}). 
	\end{equation} 
	Since $\conv I(\mu_n)<\infty$, we must have $I(\mu_{n,i})<\infty$ for all $i\leq k_n$. Thus, by Property~\ref{item: I2 infinite outside Adebal} of Proposition~\ref{prop: properties of rate function -s}, the measure $\mu_{n,i}$ is admissible for all $i\leq k_n$. Let $D$ denote the support of $\mu$. For all $i\leq k_n$, let $q_{n,i}=\mu_{n,i}(D)$ and 
	\begin{equation*}
		\mu_{n,i}'=
		\begin{cases}
			\frac1{q_{n,i}}\mu_{n,i}|_D,\quad &\hbox{if $q_{n,i}\neq 0$},\\
			\mu,\quad &\hbox{otherwise},
		\end{cases}
		\qquad
		\mu_{n,i}''=
		\begin{cases}
			\frac1{1-q_{n,i}}\mu_{n,i}|_{\spaces^2\setminus D},\quad &\hbox{if $q_{n,i}\neq 1$},\\
			\mu,\quad &\hbox{otherwise}.
		\end{cases}
	\end{equation*}
	The measures $\mu_{n,i}'$ and $\mu_{n,i}''$ are admissible.
	We have $\mu_{n,i}=q_{n,i}\mu_{n,i}'+(1-q_{n,i})\mu_{n,i}''$ and 
	\begin{equation}
		\label{eqloc: I affine with qi}
		\begin{split}
			I(\mu_{n,i})= q_{n,i}I(\mu_{n,i}')+(1-q_{n,i})I(\mu_{n,i}'').
		\end{split}
	\end{equation}
	Indeed, this is immediate if $q_{n,i}\in \{0,1\}$. Otherwise, $q_{n,i}\in (0,1)$ and this is a consequence of Property~\ref{item: linearity of I2 between separated measures} of Proposition~\ref{prop: properties of rate function -s} with the partition $\mathcal J_{\mu_{n,i}}=\big(\mathcal J_{\mu_{n,i}}\cap \mathcal J_\mu\big)\cup \big(\mathcal J_{\mu_{n,i}}\cap \mathcal J_\mu^c\big)$, where $\mathcal J_\mu^c$ denotes the complement of $\mathcal J_\mu$ in $\mathcal J$. 
	Let
	\begin{equation*}
		p_n=\sum_{i=1}^{k_n}\lambda_{n,i}q_{n,i}=\mu_{n}(D).
	\end{equation*}
	Since $p_n\to 1$ as $n\to \infty$, we can set 
	\begin{equation*}
		\lambda'_{n,i}=\frac{\lambda_{n,i}q_{n,i}}{p_n},
	\end{equation*}
	for $n$ large enough.
	Then, we have 
	\begin{equation*}
		\begin{split}
			\mu_{n}&=p_n\sum_{i=1}^{k_n}\lambda'_{n,i}\mu_{n,i}'+\sum_{i=1}^{k_n}\lambda_{n,i}(1-q_{n,i})\mu_{n,i}'',\\
			\conv I(\mu_n)+\frac 1n
			&\geq p_n\sum_{i=1}^{k_n}\lambda'_{n,i}I(\mu_{n,i}')+ \sum_{i=1}^{k_n}\lambda_{n,i}(1-q_{n,i})I(\mu_{n,i}'')
			\geq p_n\sum_{i=1}^{k_n}\lambda'_{n,i} I(\mu_{n,i}').
		\end{split}
	\end{equation*}
	For all $i\leq k_n$, the measure $\mu_{n,i}'$ is an element of $\{\nu \in \Ade\ |\ \mathcal J_\nu\subseteq \mathcal J_ \mu\}$. 
	By Property~\ref{item: midpoint convexity of I2} of Proposition~\ref{prop: properties of rate function -s}, since $\sum_{i=1}^{k_n}\lambda'_{n,i}=1$, we have
	\begin{equation*}
		\conv I(\mu_n)+\frac 1n\geq p_n I\bigg( \sum_{i=1}^{k_n}\lambda'_{n,i}\mu_{n,i}'\bigg).
	\end{equation*}
	Moreover,
	\begin{align*}
		\tvnorml{\sum_{i=1}^{k_n}\lambda'_{n,i}\mu_{n,i}'-\mu}
		&\leq \tvnorml{\sum_{i=1}^{k_n}\lambda'_{n,i}\mu_{n,i}'-\mu_n}+\tvnorm{\mu-\mu_n}
		\\&= \tvnorml{\sum_{i=1}^{k_n}(\lambda'_{n,i}-\lambda_{n,i}q_{n,i})\mu_{n,i}'-\sum_{i=1}^{k_n}\lambda_{n,i}(1-q_{n,i})\mu_{n,i}''}+\tvnorm{\mu-\mu_n}
		\\&\leq  \sum_{i=1}^{k_n}(\lambda_{n,i}'-\lambda_{n,i}q_{n,i})+\sum_{i=1}^{k_n}\lambda_{n,i}(1-q_{n,i})+\tvnorm{\mu-\mu_n}
		\\&=2(1-p_n)+\tvnorm{\mu-\mu_n}.
	\end{align*}
	Since $p_n\to 1$ as $n\to \infty$, this proves that $\sum_{i=1}^{k_n}\lambda'_{n,i}\mu_{n,i}'\to\mu$ as $n\to \infty$. By the lower semicontinuity of $I$ and Lemma~\ref{lemma: I = sigmaco I}, we have
	\begin{equation*}
		\liminf_{n\to \infty} \conv I(\mu_n)\geq 	\liminf _{n\to \infty} p_nI\bigg( \sum_{i=1}^{k_n}\lambda'_{n,i}\mu_{n,i}'\bigg)\geq I(\mu)=\conv I(\mu).
	\end{equation*}
\end{proof}
\begin{remark}
	We have described $I^{**}$ over the set $\Ade$, which is enough for the purpose of the rest of this paper. However, it is fair to wonder whether the identification $I^{**}=\conv I$ also holds on the whole space $\mathcal P(\spaces^2)$. In full generality, it does not, as illustrated by Example~\ref{ex: right only random walk}. It turns out that $\conv I$ is not always lower semicontinuous on $\mathcal P(\spaces^2)$, preventing it to agree with $I^{**}$. One can actually show that $I^{**}=\sigmaco I$ on $\mathcal P(\spaces^2)$, where $\sigmaco I$ is the $\sigma$-convex envelope of $I$. In other words,
	\begin{equation*}
		\begin{split}
			I^{**}(\mu)=\sigmaco I(\mu)=\inf \bigg\{\sum_{m=1}^\infty \lambda_mI(\mu_m)\ \Big|\ (\mu_m)\in \proba^\N, &\ (\lambda_m)\in [0,1]^\N,\\& \sum_{m=1}^\infty\lambda_m=1,
			\ \sum_{m=1}^\infty \lambda_m\mu_m=\mu\bigg\}.
		\end{split}
	\end{equation*}
	For all $\mu\in \Ade$, we have $\sigmaco I(\mu)=\conv I(\mu)$.
\end{remark}

\subsection{A variation of Varadhan's lemma}
\label{section: varadhan}
The functional $\Lambda$ plays an important role in large deviations theory. In fact, one may expect that $I^*=\Lambda$ by Varadhan's Lemma.
However, standard versions of Varadhan's Lemma require the LDP to be full with good rate function, whereas Theorem~\ref{theorem: Ln2 satisfies a weak LDP with rate function -s} only provides the weak LDP.
The purpose of this section is to derive a version of Varadhan's lemma that fits the context of our weak LDP. This requires using the function $\Lambda_\infty$ instead of $\Lambda$. 

\begin{proposition}
	\label{prop: I = Lambda*}
	For all $V\in \bounded(\spaces^2)$, 
	$\irl^*(V)=\Lambda_\infty(V)$.
\end{proposition}
\begin{proof}
	This is a direct consequence of Lemmas~\ref{lemma: varadhan upper bound} and \ref{lemma: varadhan lower bound} below.
\end{proof}
\begin{lemma}
	\label{lemma: varadhan upper bound}
	For all $V\in\bounded(\spaces^2)$, $\irl^*(V)\geq\Lambda_\infty(V)$. 
\end{lemma}
The proof of Lemma~\ref{lemma: varadhan upper bound} is a simple adaptation of Varadhan's upper bound to the weak LDP; see for example Lemma 4.3.6 in \cite{DZ}. The compactness properties usually obtained from the goodness of the rate function in Varadhan's upper bound are here obtained through the definition of $\Lambda_\infty$.
\begin{proof}
	Let $K\subseteq \spaces$ be a finite set and $V\in\bounded(\spaces^2)$. Let $\delta>0$.
	As $\irl$ is lower semicontinuous and $\langle \cdot, V\rangle$ is continuous, for every $\mu\in \proba$, there exists an open neighborhood $A_\mu$ of $\mu$ such that 
	\begin{equation}
		\inf_{\nu\in \overline A_\mu}\irl(\nu)\geq \irl(\mu)-\delta ,\quad \sup_{\nu\in\overline A_\mu}\langle \nu,V\rangle\leq \langle \mu, V\rangle+\delta.
	\end{equation}
	Let $\alpha\in\mathbb R$.
	The set $\mathcal L:=\{\mu\in \mathcal P(K ^2)\ |\ I(\mu)\leq \alpha\}$ is a compact subset of $\mathcal P(\spaces^2)$, hence there exists a finite sequence $\mu_1,\ldots,\mu_k\in\mathcal L$ such that $\mathcal L\subseteq \bigcup_{i=1}^k A_{\mu_i}$. 
	Denote by $C$ the complement in $\mathcal P(K^2)$ of this union, which is compact.
	We have
	\begin{equation*}
		\mathcal P(K^2)=\Bigg(\bigcup_{i=1}^k\mathcal P(K^2)\cap\overline A_{\mu_i}\Bigg)\cup C.
	\end{equation*}
	Therefore we have
	\begin{align}
		\mathbb E\left[e^{n\langle L_n,V\rangle }\indic_{\mathcal P(K^2)}(L_n)\right]
		&\leq \sum_{i=1}^k\mathbb E\left[e^{n\langle L_n,V\rangle }\indic_{\mathcal P(K^2)\cap \overline A_{\mu_i}}(L_n)\right]+\mathbb E\left[e^{n\langle L_n,V\rangle }\indic_{C}(L_n)\right]\nonumber\\
		&
		\leq \sum_{i=1}^k e^{n(\langle \mu_i, V\rangle +\delta)}\mathbb P (L_n\in\mathcal P(K^2)\cap \overline A_{\mu_i})
		\\&
		\qquad \qquad\qquad \qquad \qquad \qquad +e^{n\|V\|}\mathbb P(L_n\in C).
		\label{eqloc: expectation for Lambda_infty in Varadhan upper bound}
	\end{align}
	We use Theorem~\ref{theorem: Ln2 satisfies a weak LDP with rate function -s}. For all $i\in \{1,\ldots, k\}$, by the weak LDP upper bound applied on the compact set $\mathcal P (K^2)\cap \overline A_{\mu_i}$,
	\begin{equation*}
		\limsup_{n\to \infty}\frac 1n\log \mathbb P (L_n\in\mathcal P(K^2)\cap \overline A_{\mu_i})=-\inf_{\mu\in \mathcal P(K^2)\cap \overline A_{\mu_i}}\irl(\mu)\leq -\irl(\mu_i)+\delta.
	\end{equation*}
	Therefore,
	\begin{align*}
		\limsup_{n\to\infty}\frac1n\log \bigg(\sum_{i=1}^k e^{n(\langle \mu_i, V\rangle +\delta)}\mathbb P (L_n\in\mathcal P(K^2)\cap A_{\mu_i})\bigg)
		&\leq \max_{1\leq i\leq k}(\langle \mu_i, V\rangle +2\delta-\irl(\mu_i))
		\\&\leq \sup_{\mu\in \mathcal P(K^2)}(\langle \mu, V\rangle -\irl (\mu))+2\delta\\
		&\leq \sup_{\mu\in \mathcal P(\spaces^2)}(\langle \mu, V\rangle -\irl (\mu))+2\delta.
	\end{align*}
	We also handle the last term on the right-hand side of~\eqref{eqloc: expectation for Lambda_infty in Varadhan upper bound} with the weak LDP of Theorem~\ref{theorem: Ln2 satisfies a weak LDP with rate function -s}. By definition, $\irl (\mu)>\alpha$ for all $\mu\in C$, thus the weak LDP upper bound applied on the compact set $C$ yields
	\begin{equation*}
		\limsup_{n\to \infty}\frac 1n\log e^{n\|V\|}\mathbb P \left(L_n\in C\right)\leq \|V\|-\inf_{\mu\in C}\irl(\mu)\leq \|V\|-\alpha.
	\end{equation*}
	Combining these bounds and~\eqref{eqloc: expectation for Lambda_infty in Varadhan upper bound}, we get
	\begin{equation*}
		\begin{split}
			\Lambda_K(V)
			&=\limsup_{n\to \infty}\frac 1n\log \mathbb E\left[e^{n\langle L_n,V\rangle }\indic_{\mathcal P(K^2)}(L_n)\right]\\
			&\leq \left(\|V\|-\alpha\right)\vee\left(\sup_{\mu\in \proba}(\langle \mu, V\rangle-\irl(\mu)) +2\delta\right)\\
			&=\sup_{\mu\in \proba}(\langle \mu, V\rangle-\irl(\mu)) +2\delta,
		\end{split}
	\end{equation*}
	where the last inequality holds for a choice of $\alpha<\infty$ large enough.
	We get $\Lambda_K(V)\leq \irl^*(V)$ by taking the limit as $\delta\to0$. Since this holds for any finite set $K$, we have $\Lambda_\infty(V)\leq \irl^*(V)$.
\end{proof}
\begin{lemma}
	\label{lemma: varadhan lower bound}
	For all $V\in\bounded(\spaces^2)$, $\irl^*(V)\leq\Lambda_\infty(V)$.
\end{lemma}
The lower bound $\irl ^*\leq \Lambda$ of Varadhan's lemma holds even in the context of the weak LDP because it only relies on the LDP lower bound (see Lemma 4.3.4 of~\cite{DZ}). In order to obtain the stronger bound $I^*\leq \Lambda_\infty$, we follow the same approach. For this bound, we shall use the formula of Property~\ref{item: I = - sinfty} in Proposition~\ref{prop: properties of rate function -s}.
\begin{proof}
	Let $V\in\bounded(\spaces^2)$, and let $\mu\in \proba$. Fix $\delta>0$ and let $K$ be a finite subset of $\spaces$. There exists $r>0$, independent of $K$, such that the open ball $\mathcal B(\mu,r)$ satisfies 
	\begin{equation*}
		\inf_{\mathcal B(\mu,r)\cap \mathcal P (K^2)}\langle\cdot, V\rangle\geq \inf_{\mathcal B(\mu,r)}\langle\cdot, V\rangle \geq \langle \mu, V\rangle-\delta.
	\end{equation*}
	Therefore, by the definition of $\underline s$ (see Definition~\ref{def: ruelle lanford functions}),
	\begin{equation}
		\label{eqloc: lower bound for Lambda K}
		\begin{split}
			\Lambda_K(V)&\geq \liminf_{n\to\infty}\frac 1n\log \mathbb E\left[e^{n\langle L_n,V\rangle }\indic_{\mathcal B(\mu,r)\cap \mathcal P (K^2)}(L_n)\right]\\
			&\geq \inf_{\mathcal B(\mu,r)\cap \mathcal P (K^2)}\langle \cdot, V\rangle+\underline s(\mathcal B(\mu,r)\cap \mathcal P(K^2))\\
			&\geq \langle \mu, V\rangle-\delta+\underline s(\mathcal B(\mu,r)\cap\mathcal P(K^2)).
		\end{split}
	\end{equation}
	Notice that, since $K$ is independent of $r$, the supremum of the right-hand side of~\eqref{eqloc: lower bound for Lambda K} over all finite $K\subseteq \spaces$ is $\langle \mu, V\rangle-\delta+\underline s_\infty(\mathcal B(\mu,r))$, where $\underline s_\infty$ is as in Property~\ref{item: I = - sinfty} of Proposition~\ref{prop: properties of rate function -s}.
	Therefore, by Property~\ref{item: I = - sinfty} of Proposition~\ref{prop: properties of rate function -s},
	\begin{equation*}
		\Lambda_\infty(V)\geq \langle \mu, V\rangle-\delta+\underline s_\infty(\mathcal B(\mu,r))\geq \langle \mu, V\rangle-\delta-I(\mu).
	\end{equation*}
	The proof is completed by taking the limit as $\delta\to 0$ and the supremum over all $\mu\in \proba$.
\end{proof}
\subsection{Alternative expressions of the rate function}
\label{section: alternative expressions of the rate function}
The goal of this section is to prove the chain of identities~\eqref{eq: equality of all rate functions on Adebal}.
We will need some properties of balanced measures stated in Appendix~\ref{section: balanced measures}.
We will also need the following expressions for $\Lambda$ and $\Lambda_\infty$. Fix a bounded function $V\in\bounded(\spaces^2)$ and a finite set $K\subseteq\spaces$. Define two operators $P^V$ and $P_K^V$ by 
\begin{equation*}
	P^V\uf(x,y)=e^{V(x,y)}P\uf(x,y),\qquad P_K^V(x,y)=\indic_{K^2}(x,y)e^{V(x,y)}P\uf(x,y),
	\qquad f\in \expbounded(\spaces^2),
\end{equation*}
where $Pf$ is as in~\eqref{eq: definition operator p}.
These operators satisfy
\begin{equation*}
	\langle \ith \beta 2,(P^V)^n1\rangle=\mathbb E\left[e^{n\langle L_n,V\rangle }\right],\qquad \langle \ith \beta 2 ,(P^V_K)^n1\rangle=\mathbb E\left[e^{n\langle L_n,V\rangle }\indic_{\mathcal P(K^2)}(L_n)\right],
\end{equation*}
where $\ith \beta 2:=\beta\otimes p$. It follows that
\begin{equation}
	\label{eq: lambda expressed as a limit with a operator PV}
	\Lambda(V)=\limsup_{n\to \infty}\frac 1n\log \langle \ith \beta 2,(P^V)^n1\rangle,\qquad \Lambda_K(V)=\limsup_{n\to \infty}\frac 1n\log \langle \ith \beta 2 ,(P^V_K)^n1\rangle.
\end{equation}

\begin{proposition}
	\label{prop: equality of all rate functions on Adebal}
	Equality~\eqref{eq: equality of all rate functions on Adebal} holds. In other words, 
	$\Lambda_\infty^*(\mu)=\Lambda^*(\mu)=\idv(\mu)=R(\mu)$ 
	for all $\mu\in \Adebal$.
\end{proposition}
\begin{proof}
	The identity $\idv(\mu)=R(\mu)$ holds in all generality and is proved in Theorem~13.1 of~\cite{rassoul}, where $I$ is our $\idv$, $\nu_\mathcal Y$ denotes the first marginal of a probability measure $\nu$, and $\kappa$ is the second marginal of $\nu$. In this case, the condition $\nu_\mathcal Y=\kappa$ corresponds to $\nu\in\balanced(\spaces^2)$, and $H(\nu|Q)$ becomes our $R$. We do not repeat the proof here.

	Another inequality that does not need further development is $\Lambda_\infty\geq \Lambda$, which is immediate since $\Lambda\geq \Lambda_\infty$ by definition. As a consequence, in order to complete the proof, it suffices to prove that, for all $\mu\in \Adebal$,
	\begin{equation}
		\label{eqloc: bound DV Lambda Lambda infty}
		\Lambda^*(\mu)\geq \idv(\mu)\geq\Lambda^*_\infty(\mu).
	\end{equation}
	We prove this chain of inequalities in three steps. The first step is to show $\Lambda^*(\mu)\geq \idv(\mu)$, the second step is to show $\idv(\mu)\geq\Lambda_\infty^*(\mu)$ under assumption that $\mu$ has a finite support, and the third step consists in approximating $\mu$ by finitely supported measures in the general case to obtain $\idv(\mu)\geq\Lambda_\infty^*(\mu)$ for arbitrary $\mu\in \Adebal$.
	
	{\it Step 1.} We first show that $\Lambda^*(\mu)\geq \idv(\mu)$.\footnote{The argument of Step 1 is standard. See for instance Lemma 4.1.36 of~\cite{deuschel1989} or Theorem~4.1 of~\cite{deacosta2022}}
	Let $f\in \expbounded(\spaces^2)$.
	The function $V=\log \frac \uf{P\uf}$ is bounded and satisfies $(P^V)^n\uf=\uf$. Hence, as $\uf$ is bounded above and below, Equation \eqref{eq: lambda expressed as a limit with a operator PV} yields
	\begin{align*}
		\Lambda(V)&=\limsup_{n\to \infty}\frac1n\log \langle\ith \beta2, (P^V)^n1\rangle\\
		&= \limsup_{n\to \infty}\frac1n\log\langle\ith \beta2,(P^V)^n \uf\rangle\\
		& = \limsup_{n\to \infty}\frac1n\log\langle\ith \beta2, \uf\rangle=0.
	\end{align*}
	Therefore,
	\begin{equation*}
		\Lambda^*(\mu) \geq \langle\mu,V\rangle=\Big\langle \mu, \log\frac{\uf}{P\uf}\Big\rangle.
	\end{equation*}
	This holds true for any $\uf\in\expbounded(\spaces^2)$, hence we get $\Lambda^*(\mu)\geq \idv(\mu)$.
	
	{\it Step 2}. We prove $\idv(\mu)\geq \Lambda_\infty^*(\mu)$ under the assumption that $\mu$ has finite support.\footnote{The method used in Step 2 is a variation of a common argument. See for instance Lemma 4.1.36 of~\cite{deuschel1989}, Lemma~3.3 of~\cite{wu2005}, or Theorem~4.1 of~\cite{deacosta2022}. The big difference is the definition of the SCGF involved in the inequality with $\idv$, since $\Lambda_\infty$ is not mentioned by these references. See also our proof of Proposition~\ref{prop: rate functions for occupation time}.} In this case, there exists a finite $K\subseteq \spaces$ satisfying $\mu(K^2)=1$. Without loss of generality, since $\mu\in \Ade$, we can assume that $K^2$ is contained in the union of $C_j^2$ with $j\in {\mathcal J}_\mu$, hence any state of $ K$ is reachable form $\beta$. 
	For any $g\in \expbounded(\spaces^2)$, we have $
	\langle \ith \beta 2,P^n g\rangle =\langle\ith \beta 2P^n,g\rangle$, where $\nu P$ is defined on $\proba$ by $\nu P(x,y)=\sum_{x'\in \spaces}\nu(x',x)p(x,y)$.
	The measure $\ith \beta 2P^n$ is supported by pairs of states $(x,y)$ such that $x$ (resp. $y$) is reachable from $\beta$ in $n$ steps (resp. $n+1$ steps).
	Fix $V\in\bounded(\spaces^2)$, and let $\lambda>\Lambda_\infty(V)\geq \Lambda_K(V)$. We define the function
	\begin{equation*}
		\uf=\sum_{n=0}^\infty e^{-\lambda n}(P^V_K)^n1,
	\end{equation*}
	and we show that $f\in \expbounded(\spaces^2)$. Since $\uf$ is bounded below by $1$ and coincides with 1 outside $K^2$, it suffices to show that $f$ is finite on $K^2$.
	By~\eqref{eq: lambda expressed as a limit with a operator PV}, we have 
	\begin{equation*}
		\limsup_{n\to\infty}\frac 1n\log \langle \ith \beta 2,e^{-n\lambda}(P_K^V)^n1\rangle =\Lambda_K(V)-\lambda<0,
	\end{equation*}
	thus $\langle \ith \beta2, \uf\rangle<\infty$. It follows that $\uf$ is finite on the support of $\ith \beta2$. Moreover, for all $(x,y)\in K^2$, 
	\begin{align}
		\label{eqloc: Pu=A(u-1)}
		P\uf(x,y)=e^{-V(x,y)}P_K^V\uf(x,y)=e^{\lambda-V(x,y)}(\uf(x,y)-1).
	\end{align}
	It follows that $\uf$ is also finite on the support of $\ith \beta 2P$ because
	\begin{align*}
		\langle \ith \beta 2P,\uf\rangle
		&=\ith \beta 2((K^2)^c)+\langle \ith \beta 2P,\uf\indic _{K^2}\rangle
		\\&=\ith \beta 2((K^2)^c)+\langle \ith \beta 2,P\uf\indic_{K^2}\rangle
		\\&\leq 1+e^{\lambda+\|V\|}\langle\ith \beta 2,\uf\rangle<\infty.
	\end{align*}
	By induction, $\uf$ is finite on the support of $\ith \beta 2 P^n$ for all $n$, thus $f$ is finite on the set of all pairs reachable from $\beta$. Therefore $\uf$ is finite on $K^2$ and we can conclude that $\uf\in\expbounded(\spaces^2)$. Therefore, by definition of $\idv(\mu)$ and~\eqref{eqloc: Pu=A(u-1)},
	\begin{align*}
		\idv(\mu)\geq \Big\langle \mu,\log\frac{\uf}{P\uf}\Big\rangle&=\sum_{(x,y)\in K^2}\mu(x,y)\log \frac{\uf(x,y)}{P\uf(x,y)}\\
		&=\sum_{(x,y)\in K^2}\mu(x,y)(V(x,y)-\lambda) +\sum_{(x,y)\in K^2}\mu(x,y)\log \frac{\uf(x,y)}{\uf(x,y)-1}\\
		&\geq \langle\mu,V\rangle -\lambda+0.
	\end{align*}
	The inequality $\idv(\mu)\geq \langle\mu,V\rangle -\lambda$ holds true for all $\lambda >\Lambda_{\infty}(V)$, hence $\idv(\mu)\geq \langle\mu,V\rangle -\ \Lambda_{\infty}(V)$.
	Therefore, $\idv(\mu)\geq\Lambda_{\infty}^*(\mu)$.
	
	At this point, we have proven~\eqref{eqloc: bound DV Lambda Lambda infty}~---~and thus~\eqref{eq: equality of all rate functions on Adebal}~---~on the set of measures of $\Adebal$ that have finite support.
	
	{\it Step 3.} We turn to the general case, where the support of $\mu$ is no longer assumed to be finite. Consider the sequence $(\mu_n)$ given by Lemma~\ref{lemma: balanced measure decomposition} in Appendix \ref{section: balanced measures}. By Property~\ref{item: support of mun} of Lemma~\ref{lemma: balanced measure decomposition}, each $\mu_n$ has finite support and belongs to $\Adebal$ (the admissibility of $\mu_n$ follows from the admissibility of $\mu$, since $\mu_n$ is absolutely continuous with respect to $\mu$). Hence, by the previous step, 
	we have $R(\mu_n)=\idv(\mu_n)=\Lambda^*(\mu_n)$. 
	By Property~\ref{item: mun to mu and R(mun) to R(mu} of Lemma~\ref{lemma: balanced measure decomposition} and the lower semicontinuity of $\Lambda^*$, we have 
	\begin{equation*}
		\idv(\mu)=R(\mu)=\lim_{n\to\infty }R(\mu_n)=\lim_{n\to\infty}\Lambda^*(\mu_n)\geq \Lambda^*(\mu).
	\end{equation*}
	This completes the proof.
\end{proof}
\begin{remark}
	\label{remark: about step 3 of IDV=Lambda*}
	Step 3 of the proof of Proposition~\ref{prop: equality of all rate functions on Adebal} is one reason why we had to work with pair measures from the beginning. Indeed, Lemma~\ref{lemma: balanced measure decomposition} only works with pair measures, and has no counterpart with the standard occupation measures $(\ith L1_n)$. 
\end{remark}
	\section{Level-2 weak LDP}
\label{section: contraction}
\subsection{A weak contraction principle}
In order to deduce the weak LDP for $(\ith L1_n)$ from that of $(\ith L2_n)$, we will use a variation of the contraction principle, adapted from the usual one and stated in Lemma~\ref{lemma: contraction principle} below. It is comparable to Theorem~3.3 of~\cite{pfisterlewis1995}.\footnote{Theorem~3.3 of~\cite{pfisterlewis1995}, uses the assumption that $\mathcal X$ and $\mathcal Y$ are locally compact to prove that the contracted rate function is lower semicontinuous. In our Lemma~\ref{lemma: contraction principle}, we drop this assumption and replace it with the hypothesis that they are metric.}
\begin{lemma}[Weak contraction principle]
	\label{lemma: contraction principle}
	Let $\mathcal X$ and $\mathcal Y$ be metric spaces, and $(A_n)$ a sequence of random variables taking values in $\mathcal X$ and satisfying the weak LDP with rate function $I_\mathcal{X}$. Let $F:\mathcal X\to\mathcal Y$ be a continuous function such that $F^{-1}(K)$ is compact for every compact set $K$ of $\mathcal Y$. Then, $(F(A_n))$ satisfies the weak LDP in $\mathcal Y$ with rate function $I_{\mathcal Y}$ defined by
	\begin{equation}
		I_{\mathcal Y}(y)=\inf_{x\in F^{-1}(y)}I_{\mathcal X}(x).
	\end{equation}
\end{lemma}
\begin{proof}
	We only have to ensure that $I_\mathcal Y$ is lower semicontinuous. Once this is proved,
	it suffices to reproduce any proof of a full contraction principle while replacing the word ``closed'' with ``compact''. See for instance Theorem~III.20 of~\cite{denhollander2008} or Theorem~4.1.2 of~\cite{DZ}.
	
	Let $(y_n)$ be a sequence of elements of $\mathcal Y$ converging to some $y\in \mathcal Y$. We must show that 
	\begin{equation}
		\label{eqloc: contraction lsc goal}
		I_\mathcal Y(y)\leq \liminf_{n\to \infty}I_\mathcal Y(y_n).
	\end{equation}
	If the right-hand side of~\eqref{eqloc: contraction lsc goal} is infinite, this is clear. Otherwise, 
	consider a subsequence $(y_{\varphi(n)})$ satisfying $I_\mathcal Y(y_{\varphi(n)})<\infty$ for all $n$ and such that $ \liminf_{n\to \infty}I_\mathcal Y(y_n)= \lim_{n\to \infty}I_\mathcal Y(y_{\varphi(n)})$. Then $y_{\varphi(n)}\to y$ as $n\to \infty$. By definition of $I_\mathcal Y$, there exists a sequence $(x_n)$ of elements of $\mathcal X$ such that, for all $n$, $F(x_n)=y_{\varphi(n)}$ and 
	\begin{equation*}
		I_\mathcal X(x_n)\leq I_\mathcal Y(y_{\varphi(n)})+\frac 1n.
	\end{equation*} 
	The sequence $(x_n)$ takes values in the compact set $F^{-1}\left(\{y_{\varphi(n)}\ |\ n\in \N\}\cup\{y\}\right)$. Hence it has a subsequence $(x_{\psi(n)})$ that converges to a certain $x\in\mathcal X$ as $n\to\infty$. Since $F$ is continuous, $F(x)=y$. The function $I_\mathcal X$ being lower semicontinuous, we get
	\begin{align*}
		I_\mathcal Y(y)&\leq I_\mathcal X(x)\\
		&\leq\liminf_{n\to \infty}I_{\mathcal X}(x_{\psi(n)})\\
		&\leq \liminf_{n\to \infty}\Big(I_{\mathcal Y}(y_{\psi\circ \varphi(n)})+\frac 1n \Big)\\
		&=\lim_{n\to \infty}I_\mathcal Y(y_{\varphi(n)})\\
		&=\liminf_{n\to \infty}I_\mathcal Y(y_n).
	\end{align*}
\end{proof}
\begin{remark}
	\label{remark: no level 1 LDP} 
	It is common to derive the full level-1 LDP as a corollary of the full level-2 LDP. For this, the proof would consist in applying Lemma~\ref{lemma: contraction principle}, or another variant of the contraction principle to the continuous function $F: \mu\mapsto\langle \ith \mu1, f\rangle$, where $f:\spaces\to R^d$ is any function.
	However, this is not possible in general. For instance, in Example~\ref{ex: bryc and dembo's counterexample for level 1}, the level-1 weak LDP fails even though the level-2 weak LDP holds.
	
	In fact, the route based on Lemma~\ref{lemma: contraction principle} is successful only in the case of a finite $\spaces$,\footnote{One can show that the critical assumption that $F^{-1}(K)$ is compact for every compact set of $\mathcal Y$ in Lemma~\ref{lemma: contraction principle} is satisfied if and only if $\spaces$ is finite.}
	which is detailed in Appendix~\ref{section: finite case}.
\end{remark}
\subsection{Weak LDP for the occupation times}
\label{section: weak LDP for Ln1}
In this section, we aim to prove Theorem~\ref{theorem: intro weak LDP for Ln1}, namely, to show that $(\ith L1_n)$ satisfies the weak LDP with rate function $\ith I1$ satisfying~\eqref{eq: rate functions I1}. Let us define the functions involved in~\eqref{eq: rate functions I1}. The function $\ith \Lambda1$ is the SCGF of $(\ith L1_n)$, defined as\label{not: Lambda 1}
\begin{equation}
	\label{eq: definition Lambda1}
	\ith \Lambda 1(V)=\limsup_{n\to \infty}\frac 1n\log \mathbb E\left[e^{n\langle \ith L1_n,V\rangle}\right], \qquad V\in \bounded(\spaces).
\end{equation} 
Its convex conjugate is defined by Definition~\ref{def: convex conjugate}.
The kernel $p$ defines an operator on $\bounded(\spaces)$ by $pf(x)=\sum_{y\in \spaces}p(x,y)f(y)$.
For any function $V\in\bounded(\spaces)$, let $p^V$ denote the operator defined by 
$p^Vf(x)=e^{V(x)}pf(x)$. Like in~\eqref{eq: lambda expressed as a limit with a operator PV}, 
\begin{equation*}
	\langle\beta,(p^V)^n1\rangle=\mathbb E_\beta\left[e^{n\langle \ith L1_n,V\rangle}\right],
\end{equation*}
and $\ith \Lambda1$ satisfies the equation
\begin{equation}
	\label{eq: expression of Lambda1 with operator pV}
	\ith\Lambda1(V)=\limsup_{n\to \infty}\frac 1n\log\langle \beta ,(p^V)^n1\rangle. 
\end{equation}
Let $\expbounded(\spaces)$ denote the set of positive functions on $\spaces$ that are bounded away from $0$ and $\infty$.\label{not: expbounded 1}
The DV entropy $\ith J1$ is defined by\label{not: IDV 1}
\begin{equation}
	\label{eq: definition IDV1}
	\ith \idv 1(\mu)=\sup _{\uf \in \expbounded(\spaces)}\Big\langle \mu,\log \frac{f}{pf}\Big \rangle  =\sup _{\uf \in \expbounded(\spaces)} \sum_{x\in\spaces}\mu(x)\log \frac{\uf(x)}{p\uf(x)} 
	,\qquad \mu \in \mathcal P(\spaces).
\end{equation}
The functional $\ith R1$ is defined on $\mathcal P(\spaces)$ as the contraction of the function $\ith R2$ defined in~\eqref{eq: definition R}:\label{not: R1}
\begin{equation}
	\label{eq: definition R1}
	\ith R1(\mu)=\inf\{\ith R2(\nu)\ |\ \nu\in\balanced(\spaces^2),\ \ith \nu 1=\mu\}
	,\qquad \mu \in \mathcal P(\spaces).
\end{equation}
All these functions coincide on $\Aun$, as proved in Proposition~\ref{prop: rate functions for occupation time} below.

Let $\pi:\balanced(\spaces^2)\to \mathcal P(\spaces)$ be the map defined by $\pi(\mu)=\ith\mu 1$. It is continuous with respect to the TV metric. Our strategy to prove Theorem~\ref{theorem: intro weak LDP for Ln1} consists in applying the contraction principle of Lemma~\ref{lemma: contraction principle} to $\pi$, and use Proposition~\ref{prop: rate functions for occupation time} to make the rate function explicit.\label{not: pi1}
Lemma~\ref{lemma: Ln1 satisfies a weak LDP} and Proposition~\ref{prop: rate functions for occupation time} are stated and proved below.
\begin{proof}[Proof of Theorem~\ref{theorem: intro weak LDP for Ln1}]
	By Lemma~\ref{lemma: Ln1 satisfies a weak LDP}, $(\ith L1_n)$ satsifies a weak LDP with rate function given by~\eqref{eq: I1 = inf I2}.
	It remains to show~\eqref{eq: rate functions I1}. Let $\mu\in \mathcal P(\spaces)$. If $\mu\notin \Aun$, by Proposition~\ref{prop: projection of A2bal}, there are no $\nu\in \Adebal$ such that $\mu=\ith \nu 1$, therefore by~\eqref{eq: I1 = inf I2}, we have $\ith I1(\mu)=\infty$. If $\mu\in \Aun$, by~\eqref{eq: I1 = inf I2} and the equality $\ith I 2=\ith R2$ of Theorem~\ref{theorem: intro weak LDP for Ln2}, we have $\ith I1(\mu)=\ith R1(\mu)$. The remaining equalities follow from Proposition~\ref{prop: rate functions for occupation time}.
\end{proof}
\begin{lemma}
	\label{lemma: Ln1 satisfies a weak LDP}
	The sequence $(\ith L1_n)$ satisfies a weak LDP in $\mathcal P(\spaces)$ with rate function $\ith I1$ given by
	\begin{equation}
		\label{eq: I1 = inf I2}
		\ith I 1(\mu)=\inf\{\Ide(\nu)\ |\ \nu\in\balanced(\spaces^2),\ \ith \nu 1=\mu\}.
	\end{equation}
\end{lemma}
\begin{proof}
	We only have to apply Lemma~\ref{lemma: contraction principle} to $\pi$.
	Since $\pi$ is defined on $\balanced(\spaces^2)$, we shall work only with balanced measures.
	In order to do so, we set
	\begin{equation*}
		\ith {\tilde L}2_n=\ith L 2_n-\frac 1n\delta_{(X_n,X_{n+1})}+\frac 1n\delta_{(X_n,X_1)}\in\balanced(\spaces^2).
	\end{equation*} 
	We have $\ith L1_n=\pi(\ith{\tilde L}2_n)$. 
	Moreover, for all $\mu\in \proba$, all $\delta>0$ and all $n\geq 2/\delta$, since the TV distance between $\ith L 2_n$ and $\ith{\tilde L}2_n$ is at most $\frac 2n$, we have
	\begin{equation*}
		\mathbb P(\ith { L}2_n\in \mathcal B(\mu,\delta))\leq \mathbb P(\ith {\tilde L}2_n\in \mathcal B(\mu,2\delta))\leq\mathbb P(\ith { L}2_n\in \mathcal B(\mu,3\delta)).
	\end{equation*}
	Thus,
	\begin{equation*}
		\underline s(\mu)
		\leq \lim_{\delta\to 0}\liminf_{n\to \infty} \frac 1n\log \mathbb P(\ith {\tilde L}2_n\in \mathcal B(\mu,2\delta))
		\leq \lim_{\delta\to 0}\limsup_{n\to \infty} \frac 1n\log \mathbb P(\ith {\tilde L}2_n\in \mathcal B(\mu,2\delta))
		\leq \overline s(\mu),
	\end{equation*}
	where $\underline s$ and $\overline s$ are as in Definition~\ref{def: ruelle lanford functions}. Since $\underline s=\overline s$ by Theorem~\ref{theorem: Ln2 satisfies a weak LDP with rate function -s}, $(\ith {\tilde L}2_n)$ has a RL function that coincides with $s=\underline s=\overline s$. Therefore, the sequences $(\ith L 2_n)$ and $(\ith {\tilde L}2_n)$ satisfy the same weak LDP by Lemma~\ref{lemma: RL implies LDP}.
	It only remains to check  that $\pi$ satisfies the assumptions of Lemma~\ref{lemma: contraction principle}.
	By Prokhorov's theorem any set $K\subseteq \mathcal P(\spaces)$ is compact if and only if it is tight.
	Let $\epsilon >0$ and let $K$ be compact. Then, by tightness, there exists a finite set $\Kdelta\subseteq \spaces$ such that every $\mu\in K$ satisfies $\mu(\Kdelta^c)\leq \epsilon$. Let $ \nu \in \pi^{-1}(K)\subseteq\balanced(\spaces^2)$. Then,
	\begin{equation*}
		\nu((\Kdelta\times \Kdelta)^c)=\nu((\Kdelta^c\times \spaces)\cup (\spaces\times \Kdelta^c))\leq \ith \nu 1(\Kdelta^c)+\ith \nu 1(\Kdelta^c)\leq 2\epsilon.
	\end{equation*}
	Therefore $\pi^{-1}(K)$ is tight, and thus compact. 
	Thus, by Theorem~\ref{theorem: intro weak LDP for Ln2} and Lemma~\ref{lemma: contraction principle}, $(\ith L1_n)$ satisfies the weak LDP with rate function $\ith I1$ defined by~\eqref{eq: I1 = inf I2}.
\end{proof}
\begin{proposition}
	\label{prop: rate functions for occupation time}
	Let $\mu\in \Aun$. Then,
	\begin{equation}
		(\ith \Lambda1)^*(\mu)=\ith \idv 1(\mu)=\ith R1(\mu).
	\end{equation}
\end{proposition}
\begin{proof}
	The equality $\ith \idv 1(\mu)=\ith R1(\mu)$ holds in all generality in $\mathcal P(\spaces)$.
	For the proof of this equality, we refer the reader to Theorem~13.1 of~\cite{rassoul}. To apply Equation~(13.5) of~\cite{rassoul} in our context, take $I_\mathcal X$ as our $\ith \idv 1$ and $I$ as our $\idv=\ith \idv 2$. Then, Equation~(13.5) yields 
	\begin{equation*}
		\ith \idv 1(\mu)=\inf_{\nu\in\pi^{-1}(\mu)}\ith \idv2 (\nu)=\inf_{\nu\in\pi^{-1}(\mu)}\ith R2 (\nu)=\ith R1(\mu).
	\end{equation*}
	
	Let us show that $(\ith\Lambda1)^*(\mu)=\ith \idv 1(\mu)$. The first inequality $(\ith \Lambda1)^*(\mu)\geq \ith \idv 1(\mu)$ is handled as usual:\footnote{See the first step of the proof of Proposition~\ref{prop: equality of all rate functions on Adebal}, or Lemma 4.1.36 of~\cite{deuschel1989} and Theorem~4.1 of~\cite{deacosta2022}} for any $\uf\in \expbounded(\spaces)$, we set $V=\log \frac \uf{p\uf}\in\bounded(\spaces)$, which satisfies $p^V\uf =\uf$. Thus, as in the first step of the proof of Proposition~\ref{prop: equality of all rate functions on Adebal}, by~\eqref{eq: expression of Lambda1 with operator pV}, we have 
	\begin{align*}
		\ith \Lambda1(V)&=\limsup_{n\to \infty}\frac1n\log \langle\beta , (p^V)^n1\rangle\\
		&= \limsup_{n\to \infty}\frac1n\log\langle\beta,(p^V)^n \uf\rangle\\
		& = \limsup_{n\to \infty}\frac1n\log\langle\beta, \uf\rangle=0.
	\end{align*}
	Therefore,
	\begin{equation*}
		(\ith \Lambda1)^*(\mu)\geq  \langle \mu, V\rangle=\Big\langle \mu,\log \frac{f}{pf}\Big \rangle,
	\end{equation*}
	hence $ (\ith \Lambda1)^*(\mu)\geq \ith \idv 1(\mu)$.
	We turn to the converse inequality.\footnote{To prove the converse inequality, we wish to reproduce the proof of Proposition~\ref{prop: equality of all rate functions on Adebal}, \emph{i.e.}~to find a quantity $\alpha$ such that $(\ith \Lambda1)^{*}(\mu)\leq \alpha\leq \ith J1(\mu)$. In the proof of Proposition~\ref{prop: equality of all rate functions on Adebal}, this quantity was $(\ith \Lambda2_\infty)^*(\mu)$, satisfying $(\ith \Lambda2)^*(\mu)\leq (\ith \Lambda2_\infty)^*(\mu)\leq \ith J2(\mu)$. However, due to the lack of counterpart of Lemma~\ref{lemma: balanced measure decomposition} in $\mathcal P(\spaces)$, the argument here to handle measures with infinite support is slightly different. This argument is more standard; see~\cite{wu2005}. The same argument could be used to prove that $(\ith \Lambda2)^*(\mu)\leq \ith J2(\mu)$ in Proposition~\ref{prop: equality of all rate functions on Adebal}, but it would not provide any equality implying $(\ith\Lambda2_\infty)^*(\mu)$.} Let $\spaces_\beta$ be the set of elements $x\in \spaces$ such that $\beta\leadsto x$. Since $\mu \in \Aun$, we have $\supp\mu\subseteq\spaces_\beta$. For every $V\in \bounded(\spaces)$, we define
	\begin{equation}
		\ith {\widetilde\Lambda}1(V)=\sup_{x\in\spaces_\beta}\limsup_{n\to \infty}\frac1n\log \mathbb E_x\left[e^{n\langle \ith L1_n, V\rangle }\right],
	\end{equation}
	where $\mathbb E_x[\cdot]$ denotes the expectation with respect to the law of the Markov chain with kernel $p$ and initial measure $\delta_x$. We will show that 
	\begin{equation}
		\label{eqloc: Lambda ^* < Lambda tilde^* < IDV}
		\ith \idv 1(\mu)\geq (\ith {\widetilde \Lambda}1)^*(\mu)\geq (\ith \Lambda1)^*(\mu).
	\end{equation}
	For the second inequality, it suffices to show that $ \ith \Lambda 1(V)\geq\ith{\widetilde \Lambda}1(V)$ for all $V\in \bounded(\spaces)$.
	Let $V\in \bounded(\spaces)$ and $\delta>0$. There exists $x\in\spaces_\beta$ such that
	\begin{equation*}
		\limsup_{n\to \infty}\frac 1n\log \mathbb E_{x}\left[e^{\langle \ith L1_{n},V\rangle}\right] \geq \ith {\widetilde \Lambda}1(V)-\delta.
	\end{equation*}
	Since $\beta\leadsto x$, there exists a finite sequence $\supp \beta \ni x_1,x_2,\ldots, x_k=x$ satisfying $\mathbb P_k(x_1\ldots x_k)>0$.
	By the Markov property, we get
	\begin{align*}
		\mathbb E\left[e^{\langle \ith L1_n,V\rangle }\right]
		&\geq \mathbb E\left[e^{\langle \ith L1_n,V\rangle }\indic _{\{X_1=x_1,\ldots,X_k=x_k\}}\right]\\
		&=\mathbb E_{x}\left[e^{\langle \ith L1_{n-k+1},V\rangle}\right]e^{V(x_1)+\ldots+V(x_{k-1})}\mathbb P_{k}(x_1\ldots x_k).
	\end{align*}
	Hence,
	\begin{align*}
		\ith \Lambda 1(V)
		\geq \limsup_{n\to \infty}\frac 1n\log \mathbb E_{x}\left[e^{\langle \ith L1_{n-k+1},V\rangle}\right] +0
		\geq \ith {\widetilde \Lambda}1(V)-\delta.
	\end{align*}
	This holds for any $\delta>0$, thus $\ith \Lambda 1(V)\geq \ith{\tilde \Lambda}1(V)$.\footnote{Lemma~3.3 of~\cite{wu2005} provides an alternative proof of this inequality. This inequality can be strict; see Example~\ref{ex: right only random walk}.} We turn to the first inequality in~\eqref{eqloc: Lambda ^* < Lambda tilde^* < IDV}.
	We want to optimize the bracket in the definition of $\ith \idv 1$ in~\eqref{eq: definition IDV1}. Since $\supp\mu\subseteq \spaces_\beta$, we can consider functions of $\expbounded(\spaces_\beta)$ instead of functions of $\expbounded(\spaces)$ in~\eqref{eq: definition IDV1}. 
	Let $V\in \bounded(\spaces)$ and $\lambda>\ith{\widetilde\Lambda}1(V)$. Let 
	\begin{equation*}
		f_n(x)=\sum_{k=0}^n e^{-\lambda k}(p^V)^k1(x),\qquad x\in \spaces_\beta.
	\end{equation*}
	Since the sum is finite, we have $f_n\in\expbounded(\spaces_\beta)$, with the bound $f_n\geq f_0=1$.
	The sequence $(f_n)$ converges pointwise over $\spaces_\beta$ because, for all $x\in \spaces_\beta$,
	\begin{align*}
		\limsup_{k\to \infty}\frac 1k\log \left (e^{-k\lambda} (p^V)^k1(x)\right)
		&=\limsup_{k\to \infty}\frac 1k\log \left (e^{-k\lambda}\mathbb E_{x}\left[e^{n\langle {\ith L1_n},V\rangle}\right] \right)\\
		&\leq \ith{ \widetilde \Lambda}1(V)-\lambda\\&<0.
	\end{align*}
	We call $f$ the pointwise limit of $(f_n)$ on $\spaces_\beta$. 
	It is not necessarily bounded so we cannot use it in the optimization of the bracket in~\eqref{eq: definition IDV1}. However, for every $n\geq 1$, we have
	$$
	\ith \idv1(\mu)\geq \Big \langle \mu, \log \frac{f_n}{pf_n}\Big \rangle=\sum_{x\in\spaces_\beta}\mu(x)\log \frac{f_n(x)}{p f_n(x)} .
	$$
	Moreover, we have, for all $n\in \N$,
	\begin{equation*}
		pf_n=e^{-V}p^Vf_n=e^{V-\lambda}(f_{n+1}-1).
	\end{equation*}
	Furthermore, for all $n\in \N$,
	\begin{equation*}
		f_{n+1}-f_n=e^{-\lambda(n+1)}(p^V)^{n+1}1\leq e^{-\lambda +\|V\|}e^{-\lambda n}(p^V)^n1\leq e^{-\lambda +\|V\|}f_n,
	\end{equation*}
	thus 
	\begin{equation*}
		\frac{f_n}{pf_n}=e^{\lambda-V}\frac{f_n}{f_{n+1}-1}\geq e^{\lambda-V}\frac{f_n}{(1+e^{-\lambda +\|V\|})f_n-1}\geq 1.
	\end{equation*}
	Thus, by Fatou's lemma,
	\begin{align*}
		\ith \idv1(\mu)
		&\geq \liminf_{n\to\infty }\sum_{x\in \spaces_\beta} \mu(x)\log\frac{f_n(x)}{pf_n(x)}\\
		&\geq \sum_{x\in \spaces_\beta} \mu(x) \liminf_{n\to\infty}\log\frac{f_n(x)}{pf_n(x)}\\
		&= \langle \mu, V\rangle -\lambda +\sum_{x\in \spaces_\beta} \mu(x) \lim_{n\to\infty}\log\frac{f_n(x)}{f_{n+1}(x)-1}\\
		&=\langle \mu, V\rangle -\lambda +\sum_{x\in \spaces_\beta} \mu(x) \log\frac{f(x)}{f(x)-1}
		\geq \langle \mu, V\rangle -\lambda.
	\end{align*}
	This holds\footnote{The equality line holds because $f$ is never infinite. If there exists $x$ in $\supp mu$ such that $f(x)=1$, the sum is infinite and the final inequality still holds.} for every $\lambda>\ith{\widetilde\Lambda}1$, thus $\ith \idv1(\mu)\geq \langle \mu, V\rangle -\ith{\widetilde\Lambda}1$ and ultimately $\ith \idv 1(\mu)\geq (\ith {\widetilde \Lambda}1)^*(\mu)$. The proof is complete.
\end{proof}
Before leaving this section, let us add a corollary of Lemma~\ref{lemma: Ln1 satisfies a weak LDP} as a remark about $\ith I1$. Like $\ith I2$, the function $\ith I1$ can be decomposed over irreducible classes. 
\begin{corollary}
	Let $\mu\in \Aun$ and let $\tilde \mu_j={\mu(C_j)}^{-1}\mu|_{C_j}$ for all $j\in \mathcal J_\mu$. We have
	\begin{equation*}
		\ith I 1(\mu)=\sum_{j\in\mathcal J_\mu}\mu(C_j) \ith I1(\tilde \mu_j).
	\end{equation*}
\end{corollary}
\begin{proof} 
	We use the expression of $\ith I1(\mu)$ provided by~\eqref{eq: I1 = inf I2} in Lemma~\ref{lemma: Ln1 satisfies a weak LDP}.
	The set $\pi^{-1}(\mu)$ is exactly the set of all $\nu\in\Adebal$ such that $\mathcal J_\nu=\mathcal J_\mu$ and for all $j\in \mathcal J_\mu$, $\nu(C_j^2)=\mu(C_j)$ and $\pi({\mu(C_j)}^{-1}\nu|_{C_j^2})=\tilde\mu_j$.
	Therefore, by Corollary~\ref{coro: I affine between classes},
	\begin{align*}
		\ith I1(\mu)
		&=\inf\bigg\{\sum_{j\in\mathcal J_\mu}\mu(C_j)\ith I2\left(\frac1{\mu(C_j)}\nu|_{C_j^2}\right)\ \Big|\ \nu\in \pi^{-1}(\mu)\bigg\}\\
		&=\inf\bigg\{\sum_{j\in\mathcal J_\mu}\mu(C_j)\ith I2(\nu_j)\ \Big|\ \forall j\in\mathcal J_\mu,\ \nu_j\in \pi^{-1}\left(\tilde \mu_j\right)\bigg\}\\
		&=\sum_{j\in\mathcal J_\mu}\mu(C_j)\inf \{\ith I2(\nu_j)\ |\ \nu_j\in \pi^{-1}\left(\tilde \mu_j\right)\}\\
		&=\sum_{j\in\mathcal J_\mu}\mu(C_j) \ith I1(\tilde \mu_j).
	\end{align*}
\end{proof}
\section{Level-3 weak LDP}
\label{section: weak LDP on the process level}
In this section, we prove Theorem~\ref{theorem: intro weak LDP for Lninfty}.
We follow the projective limit approach of Section 4.4 of~\cite{deuschel1989}, and of Sections 6.5.2 and 6.5.3 of~\cite{DZ}. Since we work with weak LDPs and without the uniformity assumption {\bf(U)} used in these books, we need to slightly adapt their propositions. In Section~\ref{section: weak LPD for Lnk}, we derive the weak LDPs for the $k$-th empirical measures $(\ith Lk_n)$, and derive convenient expressions for the associated rate functions. In Section~\ref{section: dawson gartner}, we turn these weak LDPs into the weak LDP for the sequence $(\ith L\infty_n)$.
\subsection{Weak LDP for the $k$-th empirical measures}
\label{section: weak LPD for Lnk}
In this section, we work in $\mathcal P(\spaces^k)$ with $k\geq 3$.
We use the notations of Section~\ref{section: notations}.
For all $n\in \N$, the $k$-th empirical measure of a word $w\in \spaces^{n+k-1}$ is defined as 
\label{not: Lnk}
\begin{equation*}
	\ith Lk[w]=\frac 1n\sum_{i=1}^n\delta_{(w_i,\ldots, w_{i+k-1})}.
\end{equation*}
The $k$-th empirical measure of $(X_n)$ at time $n$ is denoted $\ith Lk_n$ and is defined as the $k$-th empirical measure of the word $(X_1,\ldots, X_{n+k-1})$. The sequence $(\ith Lk_n)$ is the object of interest in this section.

Our goal is to translate the weak LDP for  $(\ith L2_n)$ of Theorem~\ref{theorem: intro weak LDP for Ln2} to the weak LDP for $(\ith Lk_n)$. To achieve this, we consider the Markov chain $(Y_n)$ on $\spaces^{k-1}$ defined by $Y_n=(X_n,\ldots ,X_{n+k-2})$, whose pair empirical measures carry the same information as $(\ith Lk_n)$, as we will see. Using the notation of~\eqref{eq: definition of p over words}, the initial measure of the Markov chain $(Y_n)$ is given by $\beta_Y(u)=\beta(u_1)p(u)$ and its transition kernel by
\begin{equation*}
	p_Y(u,v)=p(u_{k-1},v_{k-1})\prod_{i=1}^{k-2}\indic_{\{u_{i+1}=v_i\}}, \qquad u,v\in \spaces^{k-1}.
\end{equation*}
Notice that a word $u\in \spaces^{k-1}$ is reachable from $v\in \spaces^{k-1}$ for $p_Y$ if and only if there exists a word $w\in \wordspos$ of which $u$ is prefix and $v$ is suffix ($u$ and $v$ may overlap in $w$).
The pair empirical measure associated to this Markov chain is
\begin{equation}
	\ith L{2}_{Y,n}=\frac{1}{n}\sum_{i=1}^n\delta_{(Y_n,Y_{n+1})}=\frac{1}{n}\sum_{i=1}^n\delta_{((X_n,\ldots, X_{n+k-2}),(X_{n+1},\ldots ,X_{n+k-1}))}.
\end{equation}
According to Theorem~\ref{theorem: intro weak LDP for Ln2}, $(\ith L{2}_{Y,n})$ satisfies the weak LDP with an explicitly known rate function. In Proposition~\ref{prop: weak LDP for Lnk} below, we will prove that this weak LDP is actually the weak LDP for $(\ith Lk_n)$. 
Before stating Proposition~\ref{prop: weak LDP for Lnk}, let us introduce some notations. As in \eqref{eq: definition absolutely continuous measures}, let\label{not: abscont k}
\begin{equation}
	\label{eq: definition abscont k}
	\ith \abscont k=\{\mu\in \spaces^k\ |\ \forall u\in \spaces^k,\, p(u)=0 \Rightarrow \mu(u)=0\}.
\end{equation}
Let $\mu\in \mathcal P(\spaces^k)$. For any $k'\leq k$, we define a measure $\ith \mu {k'}\in \mathcal P(\spaces^{k'})$ by setting 
\begin{equation*}
	\ith \mu {k'}(A)=\mu(A\times \spaces^{k-k'}),\qquad A\subseteq \spaces^{k'}.
\end{equation*}
We denote by $\pi_{k,k'}:\mathcal P(\spaces^k)\to \mathcal P(\spaces^{k'})$ the continuous map $\nu\mapsto\ith \nu {k'}$.\label{not: pikk}
As in~\eqref{eq: definition mu balanced in S2}, we say that the measure $\mu$ is balanced, and we write $\mu\in \balanced(\spaces ^k)$ if\label{not: Pbalk}
\begin{equation}
	\label{eq: definition balanced on Sk}
	\mu(A\times \spaces)=\mu(\spaces\times A),\qquad A\subseteq \spaces ^{k-1}.
\end{equation}
Like previously, we denote by $\ith \Abal k$ the set of measures of $\mathcal P(\spaces^k)$ that are both admissible and balanced. Note that $\balanced(\spaces^k)$ and $\ith \Abal k$ are both closed. If $\mu$ is balanced, 
by induction on~\eqref{eq: definition balanced on Sk}, one can show that $\ith \mu {k'}$ satisfies
\begin{equation*}
	\ith \mu{k'}(A)=\mu(\spaces^{k_1}\times A\times \spaces^{k_2}),\qquad A\subseteq \spaces^{k'},
\end{equation*}
for all $k',k_1,k_2\geq 0$ such that $k_1+k_2=k-k'$.
It follows that $\ith \mu {k'}$ is balanced in $\mathcal P(\spaces^{k'})$ when $\mu$ is balanced in $\mathcal P(\spaces^k)$. 

Before stating~Proposition~\ref{prop: weak LDP for Lnk}, we also need to define some rate functions.
Like previously, we define the $k$-th SCGF by\label{not: Lambda k}
\begin{equation}
	\ith \Lambda k(V)=\limsup_{n\to\infty}\frac 1n\log \mathbb E\left[e^{n\langle\ith Lk_n,V\rangle}\right],
	\qquad V\in \bounded(\spaces^k).
\end{equation}
Its convex conjugate $(\ith \Lambda k)^*$ is given by Definition~\ref{def: convex conjugate}.
Let $\expbounded(\spaces^k)$ denote the set of positive functions on $\spaces^k$ that are bounded away from $0$ and $\infty$.\label{not: expbounded k}
We define, for all $\mu\in \mathcal P(\spaces^k)$,\footnote{The operator $\ith Pk$ corresponds to the operator $P$ of~\eqref{eq: definition operator p}, but for the Markov chain $((X_n,\ldots, X_{n+k-1}))$. Hence $\ith Jk$ is just the DV entropy relative to this Markov chain. Also notice that the quantity $\ith Pkf(u)$ does not depend on $u_1$.}\label{not: IDV k}
\begin{align}
	\ith \idv k(\mu)&=\sup_{f\in\expbounded(\spaces^k)}
	\Big \langle\mu,\log \frac{f}{\ith Pk f}\Big \rangle,
	\qquad \ith Pk f: u\mapsto {\sum_{y\in\spaces}p(u_k,y)f(u_2,\ldots ,u_k,y)},
\end{align}
and \label{not: Rk}
\begin{align}
	\ith Rk(\mu)&=\ent (\mu|\ith \mu {k-1}\otimes p).
\end{align} 
Note that when $\mu\in \ith \abscont k$, we have $\mu\ll \ith \mu {k-1}\otimes p$ and 
\begin{equation}
	\label{eq: expression Rk}
	\ith Rk (\mu)=\sum_{u\in \spaces^{k-1}}\sum_{y\in \spaces}\ith \mu {k}(uy)\log\frac{\ith\mu k(uy)}{\ith\mu{k-1}(u)p(u_{k-1},y)}.
\end{equation}
In the case $k=2$, the rate functions $(\ith \Lambda 2)^*$, $\ith \idv 2$ and $\ith R2$ are necessarily infinite outside $\ith \abscont 2$; see Remark~\ref{remark: I infinite on Adebal}. As one can see in Remark~\ref{remark: infinite rate functions k}, in the general case, the set $\ith \abscont k$ plays the same role for $(\ith \Lambda k)^*$, $\ith \idv k$ and $\ith Rk$.
\begin{remark}
	\label{remark: infinite rate functions k}
	As for the $k=2$ case, it is easy to check that for all $\mu\in \mathcal P(\spaces^k)$ that does not belong to $\ith \abscont k$, we have
	\begin{equation*}
		(\ith \Lambda k)^*(\mu)
		=\ith \idv k(\mu)=\ith Rk(\mu)=\infty.
	\end{equation*}
\end{remark}
We can now get to the main proposition of this section.
\begin{proposition}
	\label{prop: weak LDP for Lnk}
	The sequence $(\ith L k_n)$ satisfies the weak LDP with rate function $\ith Ik$ given by
	\begin{equation}
		\label{eq: weak LDP for Lnk}
		\ith Ik (\mu)=
		\begin{cases}
			(\ith \Lambda k)^*(\mu)
			=\ith\idv k(\mu)=\ith Rk(\mu),\quad &\hbox{if $\mu\in\ith\Abal k$},\\
			\infty,\quad &\hbox{otherwise}.
		\end{cases}
	\end{equation}
\end{proposition}
The proof of this proposition is rather technical. It consists in applying a contraction principle to a map $\Phi$ that turns $\ith L2_{n,Y}$ into $\ith Lk_n$. Then, it must be checked that the obtained rate function matches the claim~\eqref{eq: weak LDP for Lnk}.\footnote{A brief analysis of the reducibility structure of the Markov chain $(Y_n)$ reveals that its irreducible classes are slightly more intricate than $C_j^{k-1}$, which complicates the definition of admissible measures for $(Y_n)$.}
\begin{proof}
	Let
	\begin{equation*}
		\phi:
		\begin{cases}
			\quad U&\to \ \spaces^k,\\
			(u,v)&\mapsto \ u_1v,
		\end{cases}
		\qquad U=\{(u,v)\in( \spaces^{k-1})^2\ |\ u_{i+1}=v_i\ 1\leq i\leq k-2\}.
	\end{equation*}
	The application $\varphi$ is a one-to-one correspondence between $U$ and $\spaces^k$. For all $\nu\in \mathcal P(U)$, let $\Phi(\nu):=\nu\circ\varphi^{-1}\in \mathcal P(\spaces^k)$.
	The application $\Phi$ is a one-to-one correspondence between $\mathcal P(U)$ and $\mathcal P(\spaces^k)$
	(the inverse application is defined by $\Phi^{-1}(\mu)=\mu\circ\varphi\in \mathcal P(U)$ for all $\mu\in \mathcal P(\spaces^k)$).
	It is even a homeomorphism since both the maps $\Phi$ and $\Phi^{-1}$ preserve the TV distance.
	Moreover,
	if $v\in \spaces^{k}$, we have $\Phi(\delta_{\phi^{-1}(v)})=\delta_v$.
	It follows that
	\begin{equation}
		\label{eq: Phi(Lnk)=Ln2}
		\ith Lk _n=\frac{1}{n}\sum_{i=1}^{n}\Phi(\delta_{\varphi^{-1}(X_i\ldots X_{i+k-1})})=\Phi(\ith L2_{Y,n}).
	\end{equation}
	By Theorem~\ref{theorem: intro weak LDP for Ln2}, $(\ith L 2_{Y,n})$ satisfies the weak LDP on $\mathcal P((\spaces^{k-1})^2)$. Since the measure $\ith L2_{Y,n}$ always belongs to the closed set $\mathcal P(U)$, this weak LDP is actually a weak LDP on $\mathcal P(U)$ equipped by the induced weak topology. We denote by $I_Y$ the associated rate function. Therefore, by~\eqref{eq: Phi(Lnk)=Ln2} and since $\Phi$ is a homeomorphism, $(\ith Lk_n)$ satisfies the weak LDP with rate function $\ith I k=I_Y\circ\Phi^{-1}$.\footnote{This can be seen as a consequence of Lemma~\ref{lemma: contraction principle}.}\label{not: Ik}
	
	It remains to identify $\ith Ik$, \emph{i.e.}~to establish~\eqref{eq: weak LDP for Lnk}.
	Let $\mathcal A_{Y,\mathrm{bal}}^{(2)}$ denote the set of admissible, balanced measures of $\mathcal P((\spaces^{k-1})^2)$ relative to $(Y_n)$, which naturally appears in the expression of $I_Y$ given by~\eqref{eq: rate function I2 for Ln2}.
	We must first precisely describe this set of measures. We now prove that
	\begin{align}
		\label{eqloc: Phi(A2Y)= Ak}
		\Phi(\mathcal A_{\mathrm{bal},Y}^{(2)}\cap\ith \abscont 2_Y)&= \ith \Abal k\cap \ith \abscont k,
		\\
		\nonumber
		\ith \abscont 2_Y&:=\{\nu\in \mathcal P(U)\ |\ \forall (u,v)\in U,\, p_Y(u,v)=0\Rightarrow\nu(u,v)=0\}.
	\end{align}
	For convenience, we set $T_{k-1}:=\wordspos\cap \spaces^{k-1}=\{u\in \spaces^{k-1}\ |\ p(u)>0\}$. 
	Let $\nu\in \mathcal P(U)$. The following holds.
	\begin{itemize}
		\item $\nu\in \ith \abscont 2_Y$ if and only if $\Phi(\nu)\in \ith \abscont k$. Indeed, 
		for all $(u,v)\in U$,
		\begin{equation*}
			\nu(u,v)=0\Leftrightarrow \Phi(\nu)(\varphi(u,v))=0,
			\qquad p(u)p(v)=0 \Leftrightarrow p(\varphi(u,v))=0.
		\end{equation*}
		\item 
		$\nu$ is balanced in $\mathcal P((\spaces^{k-1})^2)$ if and only if $\Phi(\nu)$ is balanced in $\mathcal P(\spaces^k)$.\footnote{The question of whether $\nu$ is balanced requires seeing $\nu$ as a measure in $\mathcal P((\spaces^{k-1})^2)$ with support in $U$ instead of a measure of $\mathcal P(U)$.}
		Indeed, for all $A\subseteq \spaces^{k-1}$, 
		\begin{equation*}
			\begin{split}
				\Phi(\nu)(A\times \spaces)&=\nu(\phi^{-1}(A\times \spaces))=\nu((A\times \spaces^{k-1})\cap U)=\nu(A\times \spaces^{k-1}),\\
				\Phi(\nu)(\spaces\times A)&=\nu(\phi^{-1}( \spaces\times A))=\nu((\spaces^{k-1}\times A)\cap U)=\nu( \spaces^{k-1}\times A).
			\end{split}
		\end{equation*} 
		Hence, $\nu$ satisfies~\eqref{eq: definition mu balanced in S2} if and only if $\Phi(\nu)$ satisfies~\eqref{eq: definition balanced on Sk}. 
		\item Let $(C_{Y,j})_{j\in \mathcal J_Y}$ be the sequence of irreducible classes of the Markov chain $(Y_n)$, as in~\eqref{eq: decomposition of S in irreducible classes}. We have $\mathcal J_Y=\mathcal J$ and 
		\begin{equation}
			\label{eqloc: Cjy}
			C_{Y,j}=C_{j}^{k-1}\cap T_{k-1}, \qquad j\in \mathcal J_Y.
		\end{equation}
		Indeed, let $u,v\in C_j^{k-1}\cap T_{k-1}$. Then, $u_{k-1}\leadsto v_1$, thus there exists a word $\xi$ such that $p(u_{k-1}\xi v_1)>0$. Since $p(u)>0$ and $p(v)>0$, we have $w:=u\xi v\in \wordspos$, implying that $v$ is reachable from $u$ for $p_Y$. In the same way, $u$ is reachable from $v$ for $p_Y$, thus $u$ and $v$ are in the same irreducible class for $p_Y$.
		Now let $u,v\in \spaces^{k-1}$ be in the same class for $p_Y$. 
		There exist $w,w'\in \wordspos$ such that $u$ is a prefix of $w$ and a suffix of $w'$ and $v$ is a prefix of $w'$ and a suffix of $w$. Using these two words, we can construct a word $w''\in \wordspos$ of the form $w''=u\xi v\zeta u\xi v$. The fact that $p(w'')>0$ implies that $u,v\in T_{k-1}$ and that any letter of $u$ or $v$ is reachable form any letter of $u$ or $v$. Therefore, all the letters of $u$ and $v$ are in the same irreducibility class. This proves~\eqref{eqloc: Cjy}. 
		Therefore, $\nu$ is admissible if and only if
		\begin{equation}
			\label{eqloc: nu admissible in P(U)}
			\sum_{j\in \mathcal J_\nu}^{}\nu((C_{j}^{k-1})^2\cap T_{k-1}^2)=1,
		\end{equation}
		where $\mathcal J_\nu$ satisfies the admissibility condition of Definition~\ref{def: admissible measures}.
	\end{itemize}
	Let $\mu\in \mathcal P(\spaces^k)$ and let $\nu=\Phi^{-1}(\mu)\in \mathcal P(U)$. 
	If $\mu\in \ith \Abal k\cap \ith \abscont k$, then $\nu\in \balanced((\spaces^{k-1})^2)\cap \ith \abscont 2_Y$ and $
	\sum_{j\in \mathcal J_\mu}\mu(C_j^k)=1.$
	Since $\nu\in \ith \abscont 2_Y$, we have $\supp\nu\subseteq T_{k-1}^2$.
	It follows that 
	\begin{equation*}
		\sum_{j\in \mathcal J_\mu}\nu((C_{j}^{k-1})^2\cap T_{k-1}^2)=\sum_{j\in \mathcal J_\mu}\nu((C_{j}^{k-1})^2)=\sum_{j\in \mathcal J_\mu}\mu(C_j^k)=1.
	\end{equation*}
	Since $\mathcal J_\mu$ satisfies the admissibility conditions of Definition~\ref{def: admissible measures}, the measure $\nu$ is admissible. 
	Reciprocally, if $\nu \in \ith {\mathcal A}{2}_{\mathrm{bal},Y}\cap \ith \abscont 2_Y$, then $\mu\in \balanced(\spaces^k)\cap \ith \abscont k$ and~\eqref{eqloc: nu admissible in P(U)} holds.
	Since $\nu \in \ith \abscont 2_Y$, we have $\supp\nu\subseteq T_{k-1}^2$.
	It follows that 
	\begin{equation*}
		\sum_{j\in \mathcal J_\nu}\mu(C_j^k)=\sum_{j\in \mathcal J_\nu}\nu((C_{j}^{k-1})^2)=
		\sum_{j\in \mathcal J_\nu}\nu((C_{j}^{k-1})^2\cap T_{k-1}^2)=1.
	\end{equation*}
	Since $\mathcal J_\nu$ satisfies the admissibility conditions of Definition~\ref{def: admissible measures}, the measure $\mu$ is admissible.
	This achieves to prove~\eqref{eqloc: Phi(A2Y)= Ak}.
	
	Now we can prove~\eqref{eq: weak LDP for Lnk}.
	Let $\mu \in \mathcal P(\spaces^k)$ and let $\nu=\Phi^{-1}(\mu)$. If $\mu\notin\ith \Abal k\cap \ith \abscont k$, then by~\eqref{eqloc: Phi(A2Y)= Ak}, $\nu\notin\mathcal A_{\mathrm{bal},Y}^{(2)}\cap\ith \abscont 2_Y$, thus $\ith Ik(\mu)=I_Y(\nu)=\infty$. Indeed, by Proposition~\ref{prop: identification I2} and Remark~\ref{remark: I infinite on Adebal}, $I_Y$ is infinite outside of $\mathcal A_{\mathrm{bal},Y}^{(2)}\cap\ith \abscont 2_Y$.
	By Remark~\ref{remark: infinite rate functions k}, this proves~\eqref{eq: weak LDP for Lnk} for $\mu\notin \Adebal\cap \ith \abscont k$.
	Assume now that $\mu\in \ith \Abal k\cap \ith \abscont k$. Applying Proposition~\ref{prop: identification I2} to the chain $(Y_n)$ shows that
	\begin{equation} 
		I_Y(\nu) = (\ith \Lambda2_{\infty,Y})^*(\nu) =(\ith \Lambda2_Y)^*(\nu)=\ith \idv2_Y(\nu)=\ith R2_Y(\nu),
	\end{equation}
	where the subscript $Y$ was added to distinguish the rate functions corresponding to $(Y_n)$ instead of $(X_n)$. Thus, in order to complete the proof of \eqref{eq: weak LDP for Lnk}, we need to identify $(\ith \Lambda2_{\infty,Y})^*(\nu) =(\ith \Lambda k_\infty)^*(\mu)$, $(\ith\Lambda2_{Y})^*(\nu) =(\ith \Lambda k)^*(\mu)$, $\ith\idv2_Y(\nu) = \ith\idv k(\mu)$ and $\ith R2_Y(\nu)=\ith Rk(\mu)$.
	\begin{enumerate}
		\item We begin with $(\ith \Lambda2_{Y})^*(\nu) =(\ith \Lambda k)^*(\mu)$. For all $V\in \bounded((\spaces^{k-1})^2)$ and all $\gamma\in \mathcal P(U)$, we have 
		\begin{equation}
			\label{eqloc: change var V gamma}
			\langle \gamma,V\rangle=\langle \Phi(\gamma),V\circ \varphi^{-1}\rangle.
		\end{equation}
		By definition, 
		\begin{equation}
			\label{eqloc: def Lambda Y *}
			(\ith\Lambda2_{Y})^*(\nu)=\sup_{ V\in\bounded((\spaces^{k-1})^2)}\left(\langle \nu,V\rangle -\limsup_{n\to\infty}\frac 1n\log\mathbb E\left[e ^{n\langle \ith L2_{n,Y}, V\rangle}\right]\right).
		\end{equation}
		Since both $\ith L2_{n,Y}$ and $\nu$ belong to $\mathcal P(U)$, the quantity being optimized in~\eqref{eqloc: def Lambda Y *} does not depend on the values of $V$ outside $U$. Hence, the supremum can be taken over $V\in \bounded (U)$. Thus, 
		since the application $V\mapsto V\circ \varphi ^{-1}$ is a bijection between $\bounded(U)$ and $\bounded(\spaces^k)$,
		Identity~\eqref{eqloc: change var V gamma} shows that the expressions of $(\ith \Lambda2_{Y})^*(\nu)$ and of $(\ith \Lambda k)^*(\mu)$ coincide.
		\item We now show $\ith \idv 2_Y(\nu)=\ith Jk(\mu)$. 
		By definition, 
		\begin{align}
			\label{eqloc: def idv Y }
			\ith \idv 2_Y(\nu)=\sup_{\uf \in \expbounded((\spaces^{k-1})^2)}
			\Big \langle \nu, \log \frac f{P_Y f}\Big \rangle,
		\end{align}
		where\footnote{The function $P_Yf$ does not depend on its first variable.}
		\begin{equation*}
			P_Yf:(u,v)\mapsto \sum_{y\in \spaces}p(v_{k-1},y)f(v,v_2\ldots v_{k-1}y).
		\end{equation*}
		Since $\nu$ belongs to $\mathcal P(U)$, the quantity being optimized in~\eqref{eqloc: def idv Y } does not depend on the values of $f$ outside $U$.
		Hence, the supremum can be taken over $f\in \expbounded (U)$. 
		Since $f\mapsto f\circ \varphi^{-1}$ is a bijection between $\expbounded(U)$ and $\expbounded(\spaces^k)$, and after noticing that
		\begin{equation*}
			(P_Y f)\circ \varphi^{-1}(w)=\sum_{y\in \spaces}p(w_k,x)f(w_2\ldots w_k,w_3\ldots w_ky)=\ith Pk (f\circ \varphi^{-1})(w),
			\qquad w\in \spaces^k,
		\end{equation*}
		Identity~\eqref{eqloc: change var V gamma} shows again that the expressions of $\ith \idv 2_Y(\nu)$ and of $\ith \idv k(\mu)$ coincide. 
		\item
		We now show $\ith R2_Y(\nu)=\ith Rk(\mu)$. The measure $\nu$ belongs to $\ith \abscont 2_Y$, thus by~\eqref{eq: expression Rk},
		\begin{align*}
			\ith R2_Y(\nu)
			&=\sum_{(u,v)\in U}\nu(u,v)\log\frac{\nu(u,v)}{\ith \nu {1}(u) p_Y(u,v)}.
		\end{align*}
		For all $(u,v)\in U$, we have $p_Y(u,v)=p(u_{k-1},v_{k-1})$. Moreover, for all $u\in \spaces^{k-1}$, 
		\begin{equation*}
			\ith \nu1(u)=\sum_{\substack{v\in \spaces^{k-1}\\(u,v)\in U}}\nu(u,v)=\sum_{y\in \spaces}\nu(u,u_2\ldots u_{k-1}y)=\sum_{y\in \spaces}\mu(uy)=\ith \mu {k-1}(u).
		\end{equation*}
		By replacing $p_Y(u,v)$ and $\ith \nu1(u)$ by these values in the expression of $\ith R2_Y(\nu)$, since $(u,v)\mapsto (u,v_k)$ is a bijection between $U$ and $\spaces^{k-1}\times \spaces$, we find that the expression of $\ith R2_Y(\nu)$ and the expression~\eqref{eq: expression R2} of $\ith Rk (\mu)$ coincide.
	\end{enumerate}
\end{proof}
Let us also mention a relation between $\ith R k$ and relative entropy, which will be useful in the proof of Proposition~\ref{prop: alternative expression for R infty} below. 
Notice that both sides in~\eqref{eq: relative entropy for Lnk} can be $\infty$.
\begin{lemma}
	\label{lemma: relative entropy for Lnk}
	Let $\mu\in \balanced(\spaces^{k})$ be absolutely continuous with respect to $\mathbb P_{k}$. Then, 
	\begin{equation}
		\label{eq: relative entropy for Lnk}
		\ent (\mu|\mathbb P_{k})=\ent(\ith \mu{k-1}|\mathbb P_{k-1})+\ith R k(\mu).
	\end{equation}
\end{lemma}
\begin{proof} 
	For all $u\in \supp\ith \mu {k-1}$, let $q(u,\cdot)$ be the probability measure on $\spaces$ defined by the relation $\mu(uy)=\ith \mu {k-1}(u)q(u,y)$ for all $y\in\spaces$. 
	Since $\mu\in \ith \abscont k$, Equation~\eqref{eq: expression Rk} holds.
	Equation~\eqref{eq: relative entropy for Lnk} is the result of the following computations, which we justify below.
	\begin{align} 
		\ent (\mu|\mathbb P_k)
		&=\sum_{u}\ith \mu {k-1}(u)\sum_{y}q(u,y)\log\frac{\mu (uy)}{\mathbb P_k (uy)}
		\nonumber
		\\&=\sum_{u}\ith \mu {k-1}(u)\sum_{y}q(u,y)\left(\log\frac{\ith \mu {k-1}(u)}{\mathbb P_{k-1}(u)}+\log\frac{q(u,y)}{p(u_{k-1},y)}\right)
		\label{eqloc: S(mu Pk) 1}
		\\&=\sum_{u}\ith \mu {k-1}(u)\bigg(\log\frac{\ith \mu{k-1}(u)}{\mathbb P_{k-1}(u)}+\sum_{y}q(u,y)\log\frac{q(u,y)}{p(u_{k-1},y)}\bigg)
		\label{eqloc: S(mu Pk) 2}
		\\&=\sum_{u}\ith \mu {k-1}(u)\log\frac{\ith \mu{k-1}(u)}{\mathbb P_{k-1}(u)}
		+\sum_{u}\ith \mu {k-1}(u)\sum_{y}q(u,y)\log\frac{q(u,y)}{p(u_{k-1},y)}
		\label{eqloc: S(mu Pk) 3}
		\\&=\ent (\ith \mu {k-1}(u)|\mathbb P_{k-1})+\ith R{k}(\mu).
		\nonumber
	\end{align}
	In these computations, the sums are taken for $u\in\supp\ith \mu {k-1}$ and $y\in\supp q(u,\cdot)$.
	The transitions from \eqref{eqloc: S(mu Pk) 1} to \eqref{eqloc: S(mu Pk) 2} and from \eqref{eqloc: S(mu Pk) 2} to \eqref{eqloc: S(mu Pk) 3} must be justified. Let $u\in \supp \ith \mu {k-1}$.
	Then, the series 
	\begin{equation*}
		\sum_{y\in\supp q(u,\cdot)}q(u,y)\log \frac{\ith \mu {k-1}(u)}{\mathbb P_{k-1}(u)}= \log \frac{\ith \mu {k-1}(u)}{\mathbb P_{k-1}(u)}
	\end{equation*}
	is absolutely convergent in $(-\infty,+\infty)$, since $\ith \mu {k-1}$ is absolutely continuous with respect to $\mathbb P_{k-1}$. Moreover, the sum
	\begin{equation*}
		\sum_{y\in \supp q(u,\cdot)}q(u,y)\log \frac{q(u,y)}{p(u_{k-1},y)}=\ent(q(u,\cdot)|p(u_{k-1},\cdot))
	\end{equation*}
	is properly defined in $[0,+\infty]$ as a relative entropy, because $q(u,\cdot)$ is absolutely continuous with respect to $p(u_{k-1},\cdot)$.
	This justifies that the expression within brackets in~\eqref{eqloc: S(mu Pk) 2} is properly defined in $(-\infty,+\infty]$ for every $u$, which justifies the transition from \eqref{eqloc: S(mu Pk) 1} to \eqref{eqloc: S(mu Pk) 2}. 
	In~\eqref{eqloc: S(mu Pk) 3}, notice that both 
	\begin{equation*}
		u\mapsto\frac{\ith \mu {k-1}(u)}{\mathbb P_{k-1}(u)}\log \frac{\ith \mu {k-1}(u)}{\mathbb P_{k-1}(u)},
		\qquad 
		u\mapsto \sum_{y\in\supp q(u,\cdot)}q(u,y)\log\frac{q(u,y)}{p(u_{k-1},y)},
	\end{equation*}
	are bounded below; the former by $\inf_{s>0}s\log s=-e^{-1}$, and the latter by 0 as a relative entropy. Hence the two sums in~\eqref{eqloc: S(mu Pk) 3} are well defined in $(-\infty,+\infty]$, and the transition from~\eqref{eqloc: S(mu Pk) 2} to~\eqref{eqloc: S(mu Pk) 3} is justified.
\end{proof}
\subsection{Weak LDP for the empirical process}
\label{section: dawson gartner}
In this section, we prove Theorem~\ref{theorem: intro weak LDP for Lninfty}. We follow the projective limit approach of Sections 6.5.2 and 6.5.3 of~\cite{DZ}, Section 4.4 of~\cite{deuschel1989}, or originally Theorem~3.3 of~\cite{dawsongartner1987}. 
We equip $\spaces^{\N}$ with the product topology, which is metrized by the distance $\ith d{\infty}$ given by $\ith d\infty (\omega,\omega')=2^{-n}$, where $n=\inf\{k\in \N\ |\ \omega_k\neq \omega'_k\}$. 
Since $\spaces$ is separable, the Borel $\sigma$-algebra of $ \spaces^\N$ is generated by the product of Borel $\sigma$-algebras of $\spaces$; see Theorem~1.10 of~\cite{parthasarathy2005}.
Let $\mathcal P(\spaces^{\N})$ denote the space of probability measures on $\spaces^\N$. \label{not: Pcal infty}\label{not: Mcal infty} Let $\continuous(\spaces^\N)$ and $\unifcont(\spaces^\N)$ respectively denote the space of bounded continuous functions and of bounded, uniformly continuous, functions on $\spaces^\N$.\label{not: Ccal}\label{not: Cu}
We define the dual pairing\label{not: pairing infty} 
\begin{equation}
	\label{eq: dual pariing process level}
	(V,\mu)\mapsto \langle \mu, V\rangle=\int_{\spaces^\N}V\d \mu,\qquad V\in \mathcal V,\ \mu\in \mathcal P(\spaces^\N),
\end{equation}
where $\mathcal V=\continuous(\spaces^\N)$ or $\mathcal V=\unifcont(\spaces^\N)$.
Functions $V\in \bounded(\spaces^k)$ for some $k\in \N$ can be seen as functions of $\unifcont(\spaces^k)$ that depend only on their first $k$ variables, satisfying $\langle \mu, V\rangle =\langle \ith \mu k, V\rangle$ for $\mu \in\mathcal P(\spaces^\N)$. We have
\begin{equation}
	\label{eq: inclusions C(SN)}
	\bigcup_{k\geq 2}\bounded(\spaces^k)\subseteq \unifcont(\spaces^\N)\subseteq\continuous (\spaces^\N).
\end{equation} 
The weak topology on $\mathcal P(\spaces^\N)$ is the coarsest topology that makes each $\langle \cdot, V\rangle$ continuous, for all $V\in \mathcal V$ (by the Portmanteau theorem, both $\mathcal V=\continuous(\spaces^\N)$ and $\mathcal V=\unifcont(\spaces^\N)$ define the same weak topology; see Theorem~D.10 of~\cite{DZ} or Theorem~6.1 of~\cite{parthasarathy2005}).
We equip $\mathcal P(\spaces^\N)$ with this topology.
By Lemma 6.5.14 of~\cite{DZ}, the topological space $\mathcal P(\spaces^\N)$ is the projective limit of $(\mathcal P(\spaces^k))_{k\geq 2}$. This means that the weak topology is generated by sets 
\begin{equation*}
	\left\{\mu\in\mathcal P(\spaces^\N),\ \left|\langle \ith \mu k, V\rangle -x\right|<\epsilon\right\},
	\qquad k\geq 2,\ V\in \bounded(\spaces^k),\ x\in \mathbb R,\ \epsilon>0.
\end{equation*}

Let $T$ denote the shift map on $\spaces^\N$ defined by $T(x_1,x_2,\ldots)=(x_2,x_3,\ldots)$, and let $\balanced(\spaces^\N)$ denote the set of shift-invariant measures of $\mathcal P(\spaces^{\N})$. \label{not: Pbalinfty}
For $\mu\in\mathcal P(\spaces^{\N})$ and $k\geq 1$, we define $\ith \mu k\in\mathcal P(\spaces^k)$ by setting $\ith \mu k(A)=\mu(A\times \spaces^\N)$ for all $A\subseteq \spaces^k$. When $\mu$ is shift-invariant, $\ith \mu k$ is balanced. Conversely, if every $\ith \mu k$ is balanced, then $\mu$ is shift-invariant because, for all $A\subseteq \spaces^k$, by~\eqref{eq: definition balanced on Sk},
\begin{equation*}
	\mu(T^{-1}(A\times \Omega))=\mu(\spaces \times A\times \Omega)=\ith\mu{k+1}(\spaces\times A)=\ith \mu{k+1}(A\times \spaces)=\mu(A\times \Omega).
\end{equation*}
We denote by $\pi_k: \mathcal P(\spaces^\N)\to \mathcal P(\spaces^k)$ the continuous map $\mu\mapsto \ith \mu k$. Notice that $\pi_{k'}=\pi_{k,k'}\circ \pi_k$ for all $1\leq k'\leq k$, where $\pi_{k,k'}$ is a continuous map defined in Section~\ref{section: weak LPD for Lnk}.\label{not: pik}
Notice that two measures $\mu,\nu$ are equal in $\balanced (\spaces^\N)$ if and only if $\pi_k(\mu)=\pi_k(\nu)$ for all $k\geq 2$.
Recall that the empirical process $(\ith L\infty_n)$ is defined as an element of $\mathcal P(\spaces^\N)$ by~\eqref{eq: definition empirical measures level 1 2 3}.
Clearly, $\pi_k(\ith L\infty_n)=\ith Lk_n$.
Our goal is to prove Theorem~\ref{theorem: intro weak LDP for Lninfty}, \emph{i.e.}~to derive the weak LDP in $\mathcal P(\spaces^\N)$ for $(\ith L\infty_n)$ and prove Equation~\eqref{eq: rate functions Iinfty}. 

We shall first define the rate functions appearing in~\eqref{eq: rate functions Iinfty}. Let $\mu\in\mathcal P(\spaces^\N)$. We define the level-3 SCGF by\label{not: Lambda level 3}
\begin{align}
	\ith \Lambda \infty(V)&=\limsup_{n\to\infty}\frac 1n\log \mathbb E\left[e^{n\langle \ith L\infty_n, V\rangle}\right],
	\qquad V\in \continuous(\Omega).
	\label{eq: definition lambda infty}
\end{align}
We let\footnote{If we equip $\continuous(\Omega)$ with the weak topology induced by $\langle\cdot, \cdot\rangle$ and $\mathcal P(\Omega)$, the function $(\ith \Lambda \infty)^*$ is the convex conjugate of $\ith \Lambda \infty$.
}
\label{not: Lambda * level 3}
\begin{align}
	\label{eq: definition Lambda infty *}
	(\ith \Lambda\infty)^*(\mu)&
	=\sup_{V\in\continuous(\spaces^\N)}\left(\langle \mu, V\rangle -\ith \Lambda \infty (V)\right),\qquad \mu\in \mathcal P(\spaces^\N).
\end{align}
Let $\spaces^{-\N_0}=\{\ldots, \omega_{-2},\omega_{-1},\omega_{0}\ |\ \omega_i\in \spaces, \, i\in \N_0\}$, equipped with the product topology and the distance $\ith d\infty_-$ given by $\ith d\infty_-(\omega,\omega')= 2^n$ where $n=\sup\{k\in -\N_0\ |\ \omega_k\neq \omega_k'\}$.
Since $\mu$ is shift invariant, Kolmogorov's extension theorem ensures that we can define a unique measure $\mu^*\in \mathcal P(\spaces^{-\N_0})$ satisfying
\begin{equation*}
	\mu^*(\spaces^{-\N}\times\{u\})=\ith\mu{k}(u),
	\qquad u\in \spaces^k,\ k\geq 1.
\end{equation*}
Let $\expbounded(\spaces^{-\N_0}):=\exp(\continuous(\spaces^{-\N_0}))$ denote the set of continuous positive functions on $\spaces^{-\N_0}$ that are bounded away from $0$ and $\infty$.
Let\label{not: IDV infty}
\begin{equation}
	\label{eq: definition idv infty}
	\ith \idv \infty(\mu)=\sup_{\uf\in \expbounded(\spaces^{-\N_0})}\Big\langle\mu^*, \log \frac{f}{\ith P\infty f}\Big \rangle,
\end{equation}
where
\begin{equation*}
	\ith P\infty g(\omega)=\sum_{y\in \spaces}p(\omega _0,y)g(\omega y), \qquad g\in \expbounded(\spaces^{-\N_0}).
\end{equation*}
In the definition of $\ith P\infty$ above, $\omega y$ denotes the word indexed by $-\N_0$ whose letters are $(\omega y)_i=y$ if $i=0$ and $(\omega y)_i=\omega_{i+1}$ if $i<0$. The symbol $\langle\cdot,\cdot\rangle$ in~\eqref{eq: definition idv infty} denotes the same integral as in~\eqref{eq: dual pariing process level} but indexed over $\spaces^{-\N_0}$.
The function $\ith \idv \infty$ is the level-3 DV entropy.\footnote{This function is not common in the literature.}
We also define $\ith R\infty$ like in in Sections 4.4 of~\cite{deuschel1989} and~6.5.3 of~\cite{DZ}. Using again Kolmogorov's extension Theorem, we define a measure $\nu^*\in\mathcal P(\spaces^{-\mathbb N})$ by
\begin{equation*}
	\nu^*(\spaces^{-\N_0}\times \{u\})
	=\ith \mu {k-1}(u_1\ldots u_{k-1})p(u_{k-1},u_k), \qquad u\in\spaces^k, \ k\geq 2.
\end{equation*}
If $\mu^*$ is absolutely continuous with respect to $\nu^*$, let $\frac {\d \mu^*}{\d \nu^*}$ denote the Radon-Nykodym derivative of $\mu^*$ with respect to $\nu^*$.
We set $\ith R\infty(\mu)$ as the relative entropy of $\mu^*$ with respect to $\nu^*$, that is\label{not: R infty}
\begin{equation}
	\label{eq: definition R infty}
	\ith R\infty (\mu)
	=
	\begin{cases}
		\int_{\spaces^{-\N_0}}\log \frac {\d \mu^*}{\d \nu^*} \d \mu^*,\quad &\hbox{if $\mu^*\ll \nu^*$},\\
		\infty,\quad &\hbox{otherwise}.
	\end{cases}
\end{equation}
These three definitions of rate functions are natural extensions of~\eqref{eq: definition Lambda}, \eqref{eq: definition DV entropy}, and~\eqref{eq: definition R}.
The last definition needed in the statement of Theorem~\ref{theorem: intro weak LDP for Lninfty} is the definition of $\ith {\mathcal A} \infty$ and $\ith \Abal\infty$. This definition is also a natural extension of Definition~\ref{def: admissible measures}. In fact, $\ith \Abal \infty$ is the projective limit of the sequence $(\ith \Abal k)$; see Appendix~\ref{section: admissible measures on the process level}.
\begin{definition}[Admissibility]
	\label{def: admissibility on the process level}
	\label{not: Abal infty}
	\label{not: Ainfty}
	Let $\mu\in \mathcal P (\spaces^{\N})$. We say that $\mu$ is \emph{pre-admissible} if there exists a set of indices $\mathcal J_\mu\subseteq \mathcal J$ such that 
	\begin{equation}
		\label{eq: admissibility on the process level}
		\mu=\sum_{j\in \mathcal J_\mu}\mu|_{C_j^\N},
	\end{equation} 
	in which case we impose that $\mathcal J_\mu$ is minimal. 
	If $\mu$ is pre-admissible, we say that $\mu$ is \emph{admissible} if the order $\leadsto$ is total on $(C_j)_{j\in\mathcal J_\mu}$ and $\beta\leadsto C_j$ for all $j\in {\mathcal J}_\mu$. 
	We denote by $\ith \Ab \infty$ the set of all admissible measures of $\mathcal P (\spaces^{\N})$, and we set $\ith \Abal \infty=\ith \Ab \infty\cap\balanced(\spaces ^{\N})$.
\end{definition}
Now we can turn to the proof of Theorem~\ref{theorem: intro weak LDP for Lninfty}. We first use a variant of the Dawson-Gärtner Theorem to derive the weak LDP with an abstract rate function and second compute this rate function with Lemma~\ref{lemma: rate functions for Lninfty} and Proposition~\ref{prop: alternative expression for R infty}, which are stated and proved below.
\begin{proof}[Proof of Theorem~\ref{theorem: intro weak LDP for Lninfty}]
	Let $\ith I\infty :\mathcal P(\spaces^\N)\to [0,+\infty]$ be defined by
	\begin{equation}
		\label{eq: dawson gartner rate function}
		\ith I\infty(\mu)= 
		\sup_{k\geq 2}\ith I k (\ith \mu k),\qquad \mu\in \mathcal P(\spaces^\N),
	\end{equation}
	where $\ith I k$ is the rate function of Proposition~\ref{prop: weak LDP for Lnk}.
	Since it is the supremum of a family of lower semicontinuous functions, this function is lower semicontinuous. We begin by proving that $(\ith L\infty_n)$ satisfies the weak LDP with rate function $\ith I\infty$, by a variant of the Dawson-Gärtner Theorem.
	The proof of the LDP lower bound with $\ith I\infty$ is the classical proof for the lower bound of the Dawson-Gärtner Theorem, as in Theorem 4.6.1 of~\cite{DZ} and Exercise 2.1.21 of~\cite{deuschel1989}, and we do not repeat it here.
	It remains to show that the weak LDP upper bound holds with $\ith I\infty $.
	The proof of the upper bound is similar to the one of~\cite{DZ} or~\cite{deuschel1989}, up to some adaptations required by the fact that we only have weak LDPs and no good rate function.\footnote{By definition of the weak LDP upper bound, we do not need the compactness usually brought by the goodness of the rate function.} Let $K$ be a compact subset of $\mathcal P(\spaces^\N)$, and $\alpha\in \mathbb R$. For $k\in \N$, we set $\mathcal L_k=\{\mu \in \mathcal P(\spaces^k)\ |\ \ith I k(\mu )\leq \alpha \}$. We also set $\mathcal L_\infty=\{\mu \in \mathcal P(\spaces^\N)\ |\ \ith I \infty(\mu )\leq \alpha \}$.
	By definition of $\ith I\infty$, a measure $\mu\in \mathcal P¨(\spaces^\N)$ satisfies $\ith I\infty(\mu)\leq \alpha$ if and only if it satisfies $\ith Ik(\ith \mu k)\leq \alpha$ for all $k$. Therefore, $K\cap \mathcal L_\infty$ is the projective limit of closed sets $\pi_k(K)\cap\mathcal  L_k$,~\emph{i.e.}
	\begin{equation*}
		K\cap \mathcal L_\infty= \bigcap_{k\geq 2}\pi_k^{-1}(\pi_k(K)\cap\mathcal  L_k).
	\end{equation*}
	Assume that $\alpha<\inf_{\mu\in K} \ith I\infty(\mu) $. It follows that $K\cap\mathcal L_ \infty$ is empty.
	Each $\pi_k(K)\cap\mathcal  L_k$ is compact, thus by property of the projective limit of compact sets, one of them must be empty; see Theorem~B.4 in~\cite{DZ} or TG~I.6.8 in~\cite{bourbakiTG}.
There exists $k\geq 2$ such that $\pi_k(K)\cap\mathcal  L_k$ is empty, implying that $\ith Ik(\mu)>\alpha$ for all $\mu \in \pi_k(K)$.
By Proposition~\ref{prop: weak LDP for Lnk}, the weak LDP upper bound for $(\ith L k_n)$ holds with $\ith Ik$. 
Thus,
\begin{align*}
	\limsup_{n\to\infty}\frac 1n\log \mathbb P(\ith L\infty_n\in K)
	&\leq \limsup_{n\to\infty}\frac 1n\log \mathbb P(\ith Lk_n\in \pi_k(K))
	\leq-\alpha.
\end{align*} 
This holds true as long as $\alpha<\inf_{\mu \in K} \ith I\infty$, thus
\begin{equation*}
	\limsup_{n\to\infty}\frac 1n\log \mathbb P(\ith L\infty_n\in K)\leq-\inf_{\mu \in K} \ith I\infty.
\end{equation*} 
Both the lower and upper bound hold with $\ith I \infty$, hence $(\ith L\infty_n)$ satisfies the weak LDP with rate function $\ith I\infty$.
We now prove that $\ith I\infty$ satisfies~\eqref{eq: rate functions Iinfty} and \eqref{eq: level 3 relative entropy}.
If $\mu\notin\ith \Abal \infty$, then by Proposition~\ref{lemma: A infty projective limit of Ak}, there exists $k$ such that $\ith \mu k\notin \ith \Abal k$, therefore $\ith I \infty(\ith \mu \infty)\geq \ith I k(\ith \mu k)=\infty$. Assume now that $\mu\in \ith \Abal\infty$. By  Proposition~\ref{prop: weak LDP for Lnk} and Lemma~\ref{lemma: rate functions for Lninfty}, we have
\begin{align*}
	\ith I\infty(\mu)
	&=\sup_ {k\geq 2}\sup_ {V_k\in\bounded(\spaces^k)}(\langle \mu, V_k\rangle -\ith \Lambda k(V_k))
	=(\ith\Lambda\infty)^*(\mu),
	\\
	\ith I\infty(\mu)
	&=\sup_ {k\geq 2}\sup_ {\uf_k\in\expbounded(\spaces^k)}\Big\langle\ith \mu k,\log \frac{f_k}{\ith Pkf_k}\Big\rangle
	=\ith \idv\infty (\mu),
	\\
	\ith I\infty(\mu)
	&=\sup_{k\geq 2}\ith Rk(\ith \mu k)=\ith R\infty(\mu),
\end{align*}
which establishes \eqref{eq: rate functions Iinfty}.
Finally, the identity \eqref{eq: level 3 relative entropy} is proved in Proposition~\ref{prop: alternative expression for R infty} below.
\end{proof}
\begin{lemma}
\label{lemma: rate functions for Lninfty}
Let $\mu\in \mathcal P(\spaces^\N)$. 
\begin{enumerate}
	\item $\ith \Lambda\infty$ satisfies
	\begin{equation}
		(\ith \Lambda\infty)^*(\mu)
		=\sup_{V\in\unifcont(\spaces^\N)}\left(\langle \mu, V\rangle -\ith \Lambda \infty (V)\right)
		=\sup_{k\geq 2}\sup_{V_k\in\bounded(\spaces^k)} \left(\langle\ith \mu k, V_k\rangle -\ith \Lambda k (V_k)\right).
		\label{eq: alternative expression level-3 convex conjugation}
	\end{equation}
	\item 
	Let $\expbounded_\mathrm{u}(\spaces^{-\N_0}):=\exp(\unifcont(\spaces^{-\N_0}))$ denote the set of uniformly continuous positive functions on $\spaces^{-\N_0}$ that are bounded away from $0$ and $\infty$. Then, 
	$\ith \idv \infty$ satisfies
	\begin{align}
		&\ith \idv \infty(\mu)
		=\sup_{\uf\in \expbounded_\mathrm{u}(\spaces^{-\N_0})}\Big\langle \mu^*,\log \frac{f}{\ith P\infty f}\Big \rangle
		=\sup_{k\geq 2}\sup_{\uf_k\in \expbounded(\spaces^k)}\Big \langle \ith \mu k,\log \frac {f_k}{\ith P k f_k}\Big\rangle.
		\label{eq: alternative expression level-3 IDV}
	\end{align}
	\item
	\label{item: pinsker lemma}
	The sequence $(\ith Rk(\ith \mu k))$ is nondecreasing and $\ith R\infty$ satisfies 
	\begin{equation}
		\ith R\infty(\mu)=\lim_{k\to\infty}\ith Rk(\ith \mu k).
	\end{equation}
\end{enumerate}
\end{lemma}
\begin{proof}
\begin{enumerate}
	\item 
	By~\eqref{eq: inclusions C(SN)}, since $\ith \Lambda \infty (V)=\ith \Lambda k(V)$ for all $k\geq 2$ and $V\in \bounded (\spaces^k)$, each quantity in~\eqref{eq: alternative expression level-3 convex conjugation} is no larger that the one on its left. We show show that the rightmost quantity in~\eqref{eq: alternative expression level-3 convex conjugation} is as large as the leftmost one.
	Let $V\in \continuous(\spaces^\N)$, and set, for all $x\in \spaces^k$,
	\begin{equation*}
		V_k(x)=\inf\{V(\omega)\ |\ \omega_1=x_1,\ldots,\omega_k=x_k\}, \quad k\geq 2.
	\end{equation*}
	The functions $V_k$ belongs to $\bounded(\spaces^k)$. When seen as an element of $\continuous(\spaces^\N)$, it satisfies $V_k\leq V$, thus $\ith\Lambda k(V_k)\leq \ith \Lambda \infty(V)$. Moreover, the sequence $(V_k)$ is nondecreasing, thus by monotone convergence $\langle \mu, V_k\rangle\to \langle \mu, V\rangle $ as $k\to\infty$. 
	Therefore, the rightmost term in~\eqref{eq: alternative expression level-3 convex conjugation} is bounded below by
	\begin{equation*}
		\limsup_{k\to \infty}\big(\langle \ith \mu k, V_k\rangle -\ith \Lambda k(V_k)\big)\geq \langle \mu, V\rangle -\ith \Lambda \infty (V).
	\end{equation*}
	Taking the supremum over $V\in \continuous(\spaces^\N)$ yields the desired bound.
	\item 
	For all $k\geq 2$ and all $f\in \expbounded(\spaces^k)$, define $f^{*,k}$ as a function on $\spaces^{-\N_0}$ by $f^{*,k}(\omega)=f(\omega_{1-k},\omega_{-k},\ldots,\omega_0)$. This is an uniformly continuous function that is bounded from above and below, hence we have
	\begin{equation*}
		\bigcup_{k\geq 2}\expbounded(\spaces^k)\subseteq \expbounded_\mathrm{u}(\spaces^{-\N_0})\subseteq \expbounded(\spaces^{-\N_0}).
	\end{equation*}
	Also notice that 
	\begin{equation}
		\begin{split}
			\langle \mu^*,\log f^{*,k}\rangle
			=\langle \ith \mu k,\log f\rangle,
			\qquad
			\langle\mu^*,\log\ith P\infty f^{*,k}\rangle
			=\langle\ith \mu k\log \ith Pkf\rangle .
		\end{split}
	\end{equation}
	This proves that each quantity in~\eqref{eq: alternative expression level-3 IDV} is no larger than the one on its left.
	We now show that the rightmost quantity in~\eqref{eq: alternative expression level-3 IDV} is as large as the leftmost one.
	Let $f\in\expbounded(\spaces^{-\N_0})$, and set, for all $x\in\spaces^k$,
	\begin{equation*}
		f_k(x)=\inf\{f(\omega)\ |\ 
		\omega_{1-k}=x_1,\omega_{2-k}=x_{2},\ldots ,\omega_0=x_k,\}.
	\end{equation*}
	The function $f_k$ belongs to $\expbounded(\spaces^k)$ and satisfies $f_k^{*,k}\leq f$, thus 
	\begin{align*}
		\langle\mu^*,\log \ith P\infty f\rangle
		\geq \langle\mu^*,\log \ith P\infty f_k ^{* ,k}\rangle
		=\langle\ith \mu k,\log \ith Pkf_k\rangle .
	\end{align*}
	Moreover, the sequence $(\log f_k^{*,k} )$ is nondecreasing, thus by monotone convergence,
	\begin{equation*}
		\langle \ith \mu k,\log f_k\rangle=\langle\mu^*, \log f_k^{*,k}\rangle \to \langle \mu^*,\log f\rangle,
		\qquad k\to \infty.
	\end{equation*}
	Therefore, the rightmost term in~\eqref{eq: alternative expression level-3 IDV} is bounded below by
	\begin{equation*}
		\limsup_{k\to \infty}\big( \langle \ith \mu k,\log f_k\rangle -\langle \ith \mu k,\log \ith Pkf_k\rangle \big) \geq \langle\mu^*, \log f\rangle -\langle \mu^*,\log \ith P\infty f\rangle.
	\end{equation*}
	Taking the supremum over $f\in \expbounded(\spaces^{-\N_0})$ yields the desired bound.
	\item 
	Property~\ref{item: pinsker lemma} is known as Pinsker's Lemma.
	We refer the reader to Lemma~6.5.13 in~\cite{DZ} or to~\cite{pinsker1964} for a proof.
\end{enumerate}
\end{proof}
\begin{proposition}
\label{prop: alternative expression for R infty}
Let $\mu\in\balanced(\spaces^\N)$, and assume that $\ent (\ith \mu 1|\beta)<\infty$. Then,
\begin{equation}
	\label{eq: alternative expression for R infty}
	\ith R\infty(\mu) =\lim_{k\to\infty}\frac 1k\ent  (\ith \mu k |\mathbb P_k),
\end{equation}
where $\ent (\ith \mu k|\mathbb P_k)$ denotes the relative entropy of $\ith \mu k$ with respect to $\mathbb P_k$. 
\end{proposition}
\begin{proof}
Assume first that $\ith Rk(\ith \mu k)=\infty$ for some $k$. Then, Property~\ref{item: pinsker lemma} of Lemma~\ref{lemma: rate functions for Lninfty}, $\ith R{l}(\ith \mu {l})=\ith R\infty( \mu)=\infty$ for all $l\geq k$. Since $\ent(\ith \mu k|\mathbb P_k)\geq 0$, by Lemma~\ref{lemma: relative entropy for Lnk}, $\ent (\ith \mu {l}|\mathbb P_{l})=\infty$ for all $l\geq k$, and \eqref{eq: alternative expression for R infty} holds. 
Assume now that $\ith R k(\ith \mu k)$ is finite for all $k$.
Then, by Lemma~\ref{lemma: relative entropy for Lnk}, the relative entropy $\ent (\ith \mu k|\mathbb P_k)$ is either finite for all $k\geq 1$, or infinite for all $k\geq 1$. As we assumed that $\ent (\ith \mu 1|\mathbb P_1)=\ent (\ith \mu 1|\beta)<\infty$, the second case is ruled out. Therefore $\ith R k(\ith \mu k)$ can always be expressed as the telescopic difference between $\ent (\ith \mu k|\mathbb P_k)$ and $\ent (\ith \mu {k-1}|\mathbb P_{k-1})$. Hence
$\lim_{k\to \infty}\frac 1k \ent  (\ith \mu k|\mathbb P_k)$ is the Cesaro limit of the sequence $\ith Rk(\ith \mu k)$, which is also its regular limit $\ith R\infty(\mu)$ by Property~\ref{item: pinsker lemma} of Lemma~\ref{lemma: rate functions for Lninfty}.
\end{proof}

\appendix
\renewcommand\thesection{\Alph{section}}
\renewcommand\thetheorem{\thesection.\arabic{theorem}}
\section{Admissible measures}
\subsection{Balanced measures}
\label{section: balanced measures}
We call \emph{minimal cycle} any word of the form $u=vv_1$ with $v\in\words$, such that every letter of $v$ appears exactly once in $v$.
Denote by $\mathcal C$ the set of all minimal cycles $u$ satisfying $p(u)>0$.\footnote{Or equivalently, the set of all minimal cycles $u$ satisfying $L[u]\in \ith \abscont 2$.}
In this section, we present a construction that was mentioned by de La Fortelle and Fayolle in Propositions~2 and~6 of~\cite{fortelle2002} under the name of balanced measure decomposition. We provide here a detailed proof, including the  following crucial lemma, which seems to have been overlooked in~\cite{fortelle2002}. Notations are defined in Section~\ref{section: notations}.

\begin{lemma}
	\label{lemma: existence of a minimal cycle}
	Let $\nu\in \mathcal M(\spaces^2)$ be a balanced non-negative measure. If $\nu\neq 0$, then there exists a minimal cycle $v$ and $\alpha >0$ such that $\nu\geq \alpha M[v]$ in $\mathcal M(\spaces^2)$.
\end{lemma}
\begin{proof}
	Since $\nu$ is finite and non-zero, we can assume that $\nu\in \balanced(\spaces^2)$ without loss of generality. Let $q$ be a probability kernel defined by $q(x,y)=\nu(x,y)/\ith \nu1(x)$ if $x\in \supp \ith \nu 1$ and arbitrarily otherwise. We have $\ith \nu 1 \otimes q=\nu$.\footnote{$\ith \nu1\otimes q$ denotes the measure on $\spaces^2$ defined by $\ith \nu 1\otimes q(x,y)=\ith \nu 1(x)q(x,y)$, as in Section~\ref{section: usual rate functions}.}
	Let $(Z_n)$ be the Markov chain defined on $\spaces$ by the initial measure $\ith \nu 1$ and the kernel $q$. Since $\nu$ is balanced, $\ith \nu1$ is stationary for the kernel $q$.
	Consider a state $z\in \supp\ith \nu 1\neq\emptyset$. Denoting by $\mathbb E_{\nu}[\cdot]$ the expectation associated to the Markov chain $(Z_n)$, we have
	\begin{equation*}
		\label{eqloc: expectation sum an}
		\mathbb E_\nu\Bigg[\sum_{n=1}^\infty \indic_{\{Z_n=z\}}\Bigg]=\sum_{n=1}^\infty \ith \nu 1 (z)=\infty.
	\end{equation*}
	Therefore, there is a positive probability that $(Z_n)$ visits $z$ several times. This means that there exists a word $u$, whose first and last letter are $z$, satisfying $q(u_i,u_{i+1})>0$ for all $i\leq |u|-1$.
	In particular, there exists a minimal cycle $v$ satisfying $q(v_{i},v_{i+1})>0$ for all $i\leq |v|-1$. Hence, by definition of $\ith \nu 1$ and by induction, for all $i\leq |v|-1$,
	\begin{equation*}
		\ith \nu 1(v_i)\geq \ith \nu 1(v_{i-1})q(v_{i-1},v_i)\geq \ldots\geq\ith \nu 1(v_1)q(v_1,v_2)\ldots q(v_{i-1},v_i)>0.
	\end{equation*}
	Therefore,
	\begin{equation*}
		\alpha:=\min _{1\leq i\leq |v|-1} \ith \nu 1(v_i)q(v_{i},v_{i+1})>0,
	\end{equation*}
	and we have 
	$\nu\geq \alpha M[v]$.
\end{proof}
Lemma~\ref{lemma: balanced measure decomposition} provides a useful way to approximate certain balanced measures.\footnote{As mentioned by de La Fortelle and Fayolle, since measures of the form $L[v]$, where $v\in \mathcal C$, are extreme points of the convex set $\balanced(\spaces)$, this result can also be seen as an extreme point decomposition. When $\spaces$ is finite, $\balanced(\spaces)$ is compact and the decomposition is simply an application of the Krein-Milman theorem.} Recall that $\ith \abscont 2$ was defined in~\eqref{eq: definition absolutely continuous measures} and $\ith R2$ in \eqref{eq: definition R}.
Notice that, by Property~\ref{item: support of mun}, each $\mu_n$ is (pre-)admissible whenever $\mu$ is (pre-)admissible.
\begin{lemma}
	\label{lemma: balanced measure decomposition}
	Let $\mu\in\balanced(\spaces^2)\cap \ith \abscont 2$.
	There exists a sequence $(\mu_n)\in \proba^\N$ satisfying the following properties:
	\begin{enumerate}
		\item 
		\label{item: support of mun}
		For all $n\in \N$, the measure $\mu_n$ is balanced, absolutely continuous with respect to $\mu$, and of finite support.
		\item $\mu_n\to \mu$ and $\ith R2(\mu_n)\to \ith R2(\mu)$ as $n\to \infty$.
		\label{item: mun to mu and R(mun) to R(mu}
		\item 
		\label{item: mun decomposed in minimal admissible cycle}
		For all $n\in \N$, $\mu_n$ is of the form 
		\begin{equation}\label{eq: cycle decomposition}
			\mu_n=\sum_{k=1}^n\alpha_{n,k}M[ u^k],
		\end{equation}
		where $u^k\in \mathcal C$ and $\alpha_{k,n}\in [0,1]$ for all $1\leq k\leq n$. 
	\end{enumerate} 
\end{lemma}
\begin{proof}
	We first construct the words $u^k$ and coefficients $\alpha_{n,k}$ of Property~\ref{item: mun decomposed in minimal admissible cycle}.
	Since $\mathcal C$ is countable, we can index its elements and write $\mathcal C$ as $\{u^n\ |\ n\in \N\}$. 
	We define by induction a sequence of nonnegative numbers $(a_n)$ and a sequence of balanced measures\footnote{Not necessarily probability measures.} $(\nu_n)_{n\geq 0}$ by $\nu_0=0$ and
	\begin{align*}
		\nu_n&=\sum_{k=1}^n{a_k}M[u^k],\qquad n\geq 1\\
		a_{n+1}&=\min\left\{\mu(u ^n_i,u ^n_{i+1})-\nu_n(u^ n_i, u ^n_{i+1})\,\middle|\, 1\leq i\leq | u^n|-1\right\},\qquad n\geq 0.
	\end{align*}
	Since minimal cycles go through each of their letters exactly once, $M[u^{n+1}](x,y)\leq 1$ for all $n\in \N$ and all $x,y\in \spaces$. Therefore, $\nu_{n+1}(x,y)$ is equal either to $\nu_n(x,y)$ or $\nu_n(x,y)+a_{n+1}$, implying by induction that $0\leq \nu_n\leq \mu$ in $\mathcal M(\spaces^2)$ for all $n\geq 0$.
	Let $\nu\in \mathcal M(\spaces^2)$ be the limit of $(\nu_n)$. It must satisfy $0\leq \nu \leq \mu$.
	If $\nu\neq\mu$, by Lemma~\ref{lemma: existence of a minimal cycle}, there exists a minimal cycle $v$ and $\alpha>0$ such that $\mu-\nu\geq \alpha M[v]$. Since then $\mu\geq \alpha M[v]$ and $\mu\in \ith \abscont 2$, the minimal cycle $v$ must be an element of $\mathcal C$, thus $v=u^n$ for some $n\in \N$.
	Therefore, there exists an index $i$ such that $a_{n+1}=\mu(u ^n_{i}, u ^n_{i+1})-\nu_n( u ^n_{i}, u^ n_{i+1})$, thus
	\begin{align*}
		\mu(u^n_i,u^n_{i+1})-\nu(u^n_i,u^n_{i+1})
		&\geq \mu(u^n_i,u^n_{i+1})-\nu_{n+1}(u^n_i,u^n_{i+1})
		\\&=\mu(u^n_i,u^n_{i+1})-\nu_{n}(u^n_i,u^n_{i+1})-a_{n+1}
		\\&=0.
	\end{align*}
	Since $M[u^n](u^n_i,u^n_{i+1})=1$, this contradicts $\mu-\nu\geq \alpha M[u^n]$. Thus $\nu=\mu$.
	
	We now define the measures $\mu_n$ of the form \eqref{eq: cycle decomposition}. For technical reasons, they are obtained from $\nu_n$ not by simple normalization, but by adding the ``missing probability'' $1-\nu_n(\spaces^2)$ to the first cycle $u^{1}$ instead.
	Up to reindexing of $(u^n)$, we can assume $a_1>0$ without loss of generality, which ensures that $M[u^{1}]$ is absolutely continuous with respect to $\mu$. Let, for all $n\in \N$,
	\begin{equation}
		\label{eqloc: definition mu_n}
		\begin{split}
			\alpha_{n,1}= a_1+\frac{1-\nu_n(\spaces^2)}{|u^1|-1},\qquad
			\alpha_{n,k}=a_k,\qquad 2\leq k\leq n,
		\end{split}
	\end{equation}
	and let $\mu_n$ be the measure defined by~\eqref{eq: cycle decomposition} with these coefficients.
	By construction, the measure $\mu_n$ satisfies Properties~\ref{item: support of mun} and~\ref{item: mun decomposed in minimal admissible cycle}. 
	
	We now prove Property~\ref{item: mun to mu and R(mun) to R(mu}. The first part of Property~\ref{item: mun to mu and R(mun) to R(mu} is also a direct consequence of the construction of $(\mu_n)$, because 
	\begin{equation*}
		\mu_n=\nu_n+\frac{1-\nu_n(\spaces^2)}{|u^1|-1}M[u^1]\to \mu+0,\qquad n\to \infty.
	\end{equation*}
	For later purposes, note that if $U$ denotes the set of letters in the word $u ^1$, then for all $n\in \N$ and all $x,y\in \spaces$,
	\begin{equation} 
		\label{eqloc: mun < mu}
		\begin{split}
			(x,y)\notin U^2\ \Longrightarrow&\ \mu_n(x,y)=\nu_n(x,y)\leq \mu(x,y),\\
			x\notin U\ \Longrightarrow&\ \ith \mu 1_n(x)=\ith \nu1_n(x)\leq \ith \mu 1(x).
		\end{split}
	\end{equation}
	It remains to prove the convergence of $\ith R2(\mu_n)$ to $\ith R 2(\mu)$.
	If $\ith R2(\mu)=\infty$, by lower semicontinuity of $\ith R2$, it follows from $\mu_n\to \mu$ that $\ith R2(\mu_n)\to \infty$ as $n\to \infty$. Now, assume that $\ith R2(\mu)$ is finite. In particular, $\mu\in \ith \abscont 2$ and~\eqref{eq: expression R2} holds. Equation~\eqref{eq: expression R2} also holds for $\mu_n$ for all $n\in \N$ because $\mu_n\ll\mu$.
	For all $n\in \N$, we have
	\begin{align}
		\ith R2(\mu_n)&
		= R_1(\mu_n)-R_2(\mu_n),\label{eqloc: decompsoition R(mu_n)}\\
		R_1(\mu_n)&:=\ent (\mu_n|\ith \mu 1\otimes p)=\sum_{(x,y)\in\spaces ^2}\mu_n(x,y)\log\frac{\mu_n(x,y)}{\ith \mu 1(x)p(x,y)},
		\nonumber\\
		R_2(\mu_n)&:=\ent (\ith \mu1_n |\ith \mu1)=\sum_{x\in\spaces}\ith \mu1_n(x)\log\frac{\ith \mu1_n(x)}{\ith \mu 1(x)}.
		\nonumber
	\end{align}
	Since $\mu_n$ is finitely supported, both $R_1(\mu_n)$ and $R_2(\mu_n)$ are finite sums, thus they are finite.
	Let us next compute the limit of $R_1(\mu_n)$. We have
	\begin{align}
		R_1(\mu_n)&=
		\sum_{(x,y)\in U^2}k_n(x,y)
		+\sum_{(x,y)\in (U^2)^c}k_n(x,y),\label{eqloc: decomposition premier term d'entropie}\\
		k_n(x,y)&=\mu_n(x,y)\log\frac{\mu_n(x,y)}{\ith \mu1(x)p(x,y)}.\nonumber
	\end{align}
	Since it involves a finite number of terms, the first sum converges as $n\to \infty$. 
	Let $f:s\mapsto s\log s$ and $g:s\mapsto \max(e^{-1},|f(s)|)$, defined on $[0,+\infty)$.\footnote{We recall the convention that $0\log0=0$.} We have $|f| \leq g$, and the function $g$ is nondecreasing. Hence, for all $(x,y)\in (U^2)^c$,
	\begin{align*}
		|k_n(x,y)|
		&=\ith \mu 1(x)p(x,y)\Big|f\Big(\frac{\mu_n(x,y)}{\ith \mu 1(x)p(x,y)}\Big)\Big|
		\\&\leq \ith \mu 1(x)p(x,y)g\Big(\frac{\mu_n(x,y)}{\ith \mu 1(x)p(x,y)}\Big)
		\\&\leq \ith \mu 1(x)p(x,y)g\Big(\frac{\mu(x,y)}{\ith \mu 1(x)p(x,y)}\Big).
	\end{align*}
	The last line used~\eqref{eqloc: mun < mu} and the monotonicity of $g$.
	Since $g$ coincides with $f$ on $[1,+\infty)$ and with $e^{-1}$ on $[0,1]$, we get, for all $(x,y)\in (U^c)^2$,\
	\begin{equation*}
		|k_n(x,y)|\leq k(x,y):=
		\begin{cases}
			\mu(x,y)\log\frac{\mu(x,y)}{\ith \mu 1(x)p(x,y)},
			\qquad &\hbox{if $\mu(x,y)> \ith \mu 1(x,y)p(x,y)$},\\
			e^{-1}\ith \mu 1(x)p(x,y),\qquad &\hbox{otherwise}.
		\end{cases} 
	\end{equation*} 
	Since $\ith \mu 1\otimes p$ is a probability measure, we have
	\begin{equation*}
		\sum_{\substack{(x,y)\in (U^2)^c\\\mu(x,y)\leq \ith \mu 1(x)p(x,y)}}e^{-1}\ith \mu 1(x)p(x,y)\leq e^{-1}<\infty,
	\end{equation*}
	and since the negative terms in the first sum of the expression of $\ith R2$ given in~\eqref{eq: expression R2} cannot get smaller than $-e^{-1}\ith \mu 1(x)p(x,y)$, we have
	\begin{equation*}
		\sum_{\substack{(x,y)\in (U^2)^c\\\mu(x,y)> \ith \mu 1(x)p(x,y)}}\mu(x,y)\log\frac{\mu(x,y)}{\ith \mu 1(x)p(x,y)}
		\leq \ith R2(\mu)+e^{-1}<\infty.
	\end{equation*}
	Hence the sum of $k(x,y)$ over ${(x,y)\in (U^c)^2}$ is finite. Therefore, by dominated convergence, we get $R_1(\mu_n)\to \ith R2(\mu)$ as $n\to\infty$.
	Let us compute the limit of $R_2(\mu_n)$.
	We have
	\begin{align}
		R_2(\mu_n)
		&=\sum_{x\in U}\ith \mu1_n(x)\log\frac{\ith \mu1_n(x)}{\ith \mu 1(x)}+
		\sum_{x\in U^c}\ith \mu1_n(x)\log\frac{\ith \mu1_n(x)}{\ith \mu 1(x)}.\label{eqloc: decomposition du 2eme terme dd'entropie}
	\end{align}
	Since it involves a finite number of terms which vanish when $n\to \infty$, the first sum vanishes as $n\to \infty$. 
	By~\eqref{eqloc: mun < mu} and since $f$ is bounded below by $-e^{-1}$, we have, for all $x\in U^c$, 
	\begin{equation*}
		-e^{-1}\ith \mu 1(x)\leq\ith \mu1_n(x)\log \frac{\ith \mu 1_n(x)}{\ith \mu 1(x)}\leq 0.
	\end{equation*}
	Thus, by dominated convergence, the second sum vanishes too as $n\to \infty$.
	Therefore, $R_2(\mu_n)\to 0$ as $n\to \infty$, which completes the proof of Property~\ref{item: mun to mu and R(mun) to R(mu}.
\end{proof}

\subsection{Admissible measures for finite $k$}
\label{section: admissible measures}
This appendix is dedicated to the study of the set $\Ak$ of Definition~\ref{def: admissible measures}. As mentioned in the introduction, the cases $k=1$ and $k=2$ are the most important ones for the main part of this paper. Larger values of $k$ are only considered to obtain the level-3 weak LDP in Section~\ref{section: weak LDP on the process level}. 

Properties of pre-admissibility and admissibility of a measure $\mu\in \mathcal P(\spaces^k)$ only depend on the support of $\mu$. In particular, $\mu$ is pre-admissible if and only if
\begin{equation}
	\label{eq: mu pre-admissible}
	\sum_{j\in \mathcal J}\mu(C_j^k)=1.
\end{equation}
If so, $\mu$ is admissible if and only if for any given $j_1,j_2\in \mathcal J_\mu$, we have $\beta\leadsto C_{j_1}\leadsto C_{j_2}$ or $\beta\leadsto C_{j_2}\leadsto C_{j_1}$. 
For all $k\geq 2$, we denote by $\ith \Abal k$ the set of measures that are admissible and balanced; see Section~\ref{section: weak LPD for Lnk}. We will also consider the set $\ith \abscont k$ defined in~\eqref{eq: definition absolutely continuous measures} and~\eqref{eq: definition abscont k}.
\begin{remark} 
	\label{remark: other definitions of admissibility}
	In \cite{fortelle2002}, \cite{wu2005} and Corollary 13.6 of~\cite{rassoul}, admissible measures are defined as elements of $\Ade\cap \ith \abscont 2$ rather than simply elements of $\Ade$. 
	Elements of $\Ade\cap \ith \abscont 2$ and of $\Ade$ differ in the fact that the latter may have transitions that are not `allowed' under $p$.
	This apparent conflict between definitions of admissibility is without consequences, since $R(\mu)=\idv(\mu)=\Lambda^*(\mu)=\infty$ anyway for all $\mu \in \Ade\setminus \ith \abscont2$ (see Remark~\ref{remark: I infinite on Adebal}). In particular, the value of the right-hand side of~\eqref{eq: rate function I2} and~\eqref{eq: rate function I2 for Ln2} remains unchanged when $\Adebal$ is replaced by $\Adebal\cap \ith \abscont2$.
\end{remark}
\begin{proposition}
	\label{prop: Ade closed}
	$\Ak$ and $\ith {\mathcal A}k_{\mathrm {bal}} $ are closed in $\mathcal P(\spaces^k)$.
\end{proposition}
\begin{proof}
	Since $\balanced(\spaces^k)$ is closed in $\mathcal P(\spaces^k)$, it suffices to show that $\Ak$ is closed.
	Let $(\mu_n)$ be a sequence in $\Ak$ that converges to some $\mu\in\mathcal P(\spaces ^k)$ as $n\to \infty$. 
	For all $u\in \spaces^k$ that does not belong to $C_j^k$ for any $j\in \mathcal J$, we have $\mu_n(u)=0$ for all $n$, thus $\mu(u)=0$. Therefore, $\mu$ is pre-admissible.
	Let $j_1,j_2\in \mathcal J_\mu$. There exists $n\in \N$, such that $\mu_n(C_{j_1}^k)>0$ and $\mu_n(C_{j_2}^k)>0$. Since $\mu_n$ is admissible, $\beta\leadsto C_{j_1}\leadsto C_{j_2}$ or $\beta\leadsto C_{j_2}\leadsto C_{j_1}$. Therefore, $\mu$ is admissible. 
\end{proof}
\begin{remark}
	Sets of the form $\{\nu\in \Ak\ |\ \mathcal J_\nu\subseteq \mathcal J_\mu\}$ for $\mu\in \Ak$ are convex faces of the (possibly infinite dimensional) simplex $\mathcal P(\spaces^k)$. The set $\Ak$, which may be nonconvex, can always be written as a union of faces:
	\begin{equation*}
		\Ak=\bigcup_{\mu\in \Ak}\{\nu\in \Ak\ |\ \mathcal J_\nu\subseteq \mathcal J_\mu\}.
	\end{equation*}
	A visual representation of this decomposition is provided in Figure~\ref{figure: Ab for the diamond shaped MC}.
	\begin{figure}[thb]
		\centering
		\includegraphics{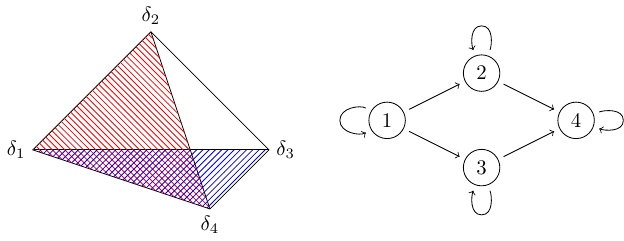}
		\caption{Let $(X_n)$ be the Markov chain on $\spaces=\{1,2,3,4\}$ whose allowed transitions are given by the graph on the right-hand side of the figure, and whose initial measure is $\beta=\delta_1$. The values of transition probabilities do not matter for this example as long as they are positive on each edge of the graph and null otherwise. The set $\mathcal P(\spaces)$ is a (finite-dimensional) simplex with extreme points $\{\delta_1,\delta_2,\delta_3,\delta_4\}$. The set $\Aun$ is the hatched domain, that is the union of the faces $\{\nu\in \Aun\ |\ \mathcal J_\nu\subseteq \mathcal J_{\frac13(\delta_1+\delta_2+\delta_4)}\}$ and $\{\nu\in \Aun\ |\ \mathcal J_\nu\subseteq \mathcal J_{\frac13(\delta_1+\delta_3+\delta_4)}\}$.}
		\label{figure: Ab for the diamond shaped MC}
	\end{figure}
\end{remark}
The following two propositions are characterizations of $\Aun$ and of $\Adebal\cap \ith \abscont 2$.
Recall that, if $\mu\in \Adebal$, the measure $\ith \mu 1\in \Aun$ is defined by~\eqref{eq: definition mu balanced in S2}.
\begin{proposition}
	\label{prop: projection of A2bal}
	We have $\Aun=\{\ith \mu 1\ |\ \mu\in \Adebal\}$.
\end{proposition}
\begin{proof}
	If $\nu\in\Aun$, consider the measure $\mu\in \balanced (\spaces^2)$ defined by $\mu(x,y)=\indic_{\{x=y\}}\nu(x)$. The measure $\mu$ satisfies $\ith \mu 1=\nu$. Moreover, $\mu(C_j^2)=\nu(C_j)$ hence $\mu$ is pre-admissible with $\mathcal J_{\mu}=\mathcal J_{\nu}$, and therefore admissible.\footnote{Recall that the definition of admissibility of $\mu$ does not require $\mu\in \ith \abscont 2$, which is why checking admissibility is so easy. See Remark~\ref{remark: other definitions of admissibility}.}
	Conversely, let $\mu\in \Adebal$. Since $\ith \mu 1(C_j)=\mu(C_j^2)$, the measure $\ith \mu 1$ is pre-admissible with $\mathcal J_{\ith \mu 1}=\mathcal J_{\mu}$, thus it is admissible.
\end{proof}
Proposition~\ref{prop: characterization A2bal} characterizes $\Adebal$ in an intuitive way. Although we prove here the full equivalence for the sake of completeness, we will only use the simpler implication (\ref{item: sequence vn L[vn] to mu} $\Rightarrow$ \ref{item: mu admissible and absolutely continuous}). 
\begin{proposition}
	\label{prop: characterization A2bal} 
	Let $\mu\in\mathcal P(\spaces^2)$. The following properties are equivalent.
	\begin{enumerate}
		\item There exists a sequence of words $(w^m)$ with $|w^m|\to \infty$ such that $\mathbb P_{|w^m|}(w^m)>0$ for all $m$ and $L[w^m]\to \mu$ as $m\to \infty$.\footnote{The notation $L[w^m]$ is introduced in~\eqref{eq: notation L[u] M[u]}}
		\label{item: sequence vn L[vn] to mu}
		\item $\mu\in \Adebal\cap \ith \abscont 2$.
		\label{item: mu admissible and absolutely continuous}
	\end{enumerate}
\end{proposition}
\begin{proof}
	Let $\mu\in \mathcal P(\spaces^2)$ and assume Property~\ref{item: sequence vn L[vn] to mu}. 
	Let us show that $\mu\in\ith \abscont 2$. Let $(x,y)\in \spaces^2$ be such that $p(x,y)=0$. Since $\mathbb P_{|w^m|}(w^m)>0$, the letters $x,y$ cannot be consecutive letters in $w^m$. Therefore, $L[w^m](x,y)=0$, thus $\mu(x,y)=0$.
	We now prove that $\mu$ is balanced.
	First, observe that 
	\begin{equation*}
		\inf_{\nu\in \balanced(\spaces^2)}\tvnorm{L[w^m]-\nu}\leq \tvnorm{L[w^m]-L[w^mw^m_1]}\leq \frac 2{|w^m-1|}\to 0,
		\qquad m\to \infty.
	\end{equation*}
	By the closedness of $\balanced(\spaces^2)$, $\mu$ is balanced. 
	Let us show that $\mu$ is admissible. Let $xy\in\spaces^2$ be a two-letter word that does not belong to $C_j^2$ for any $j\in\mathcal J$. Since $\mathbb P_{|w^m|}(w^m )>0$, the subword $xy$ appears at most once in $w^m$. As a result, $\mu(x,y)=\lim_{n\to\infty} L[w^m](x,y)= 0$. 
	It follows that $\mu$ is pre-admissible. 
	Let $j_1,j_2\in {\mathcal J}_\mu$. 
	There exists some $m\geq 1$ such that $\hat L[w^m](C_{j_1}^2)>0$ and $ L[w^m](C_{j_2}^2)>0$. Hence there exists a letter $x_1\in C_{j_1}$ and a letter $x_2\in C_{j_2}$ in $w^m$, and since $w^m$ has positive probability under $\mathbb P$, we have $\beta\leadsto x_1\leadsto x_2$ or $\beta\leadsto x_2\leadsto x_1$. It follows that $\beta\leadsto C_{j_1}\leadsto C_{j_2}$ or $\beta\leadsto C_{j_2}\leadsto C_{j_1}$, hence $\mu$ is admissible, so  Property~\ref{item: mu admissible and absolutely continuous} holds.
	
	Conversely, assume now Property~\ref{item: mu admissible and absolutely continuous}. Consider the sequence $(\mu_n)$,  the words $u^k$ and the coefficients $\alpha_{n,k}$ given by Lemma~\ref{lemma: balanced measure decomposition}, and let $m\in \N$. We define the word $w^m$ as follows.
	By Properties~\ref{item: support of mun} and~\ref{item: mun to mu and R(mun) to R(mu} of Lemma~\ref{lemma: balanced measure decomposition}, $\mu_n\in \Adebal$, and there exists a certain $n$ such that $\tvnorm{\mu_n-\mu}\leq 1/m$. 
	By the admissibility condition, up to reordering words $u^k$, we can assume that the first letter of each word $u^{k+1}$ is reachable from the last letter of its predecessor $u^k$, and we fix a word $\xi^{k+1}$ such that $p(u^k\xi ^{k+1}u ^{k+1})>0$. Additionally, $u^1_1$ is reachable from $\beta$, and we can choose a word $\xi^1$ such that $\beta(\xi^1)p(\xi^1u^1)>0$. 
	Let $N\in \N$ be such that
	\begin{equation*}
		\lambda_{N,k}:=\left\lfloor{\alpha_{n,k}N}\right\rfloor\geq 1,
		\qquad 1\leq k\leq n.
	\end{equation*}
	Since each $u^k$ is a minimal cycle, we can set $u^k=\tilde u^k\tilde u^k_1$. We set
	\begin{equation*}
		v^N=\xi^1\big(\tilde u^1\big)^{\lambda_{N,1}}u^1_1\xi^2\big(\tilde u^2\big)^{\lambda_{N,2}}u^2_1\cdots\xi^k\big(\tilde u^n\big)^{\lambda_{N,n}}u^n_1,
		\qquad 
		\ell_N=|v^N|-1.
	\end{equation*} 
	Let $\tau$ be the maximal length of words $\xi^k$. 
	We have
	\begin{align}
		\tvnorm{L[v^N]-\mu}
		&\leq \frac {\tau n}{\ell_N}+\frac n{\ell_N}
		+\tvnorml{\sum_{k=1}^n\frac {\lambda_{N,k}}{\ell_N}M[u^k]-\mu}
		\nonumber
		\\&\leq\frac {(\tau+1) n}{\ell_N}
		+\tvnorm{\mu_n-\mu}+\tvnorml{\sum_{k=1}^n\bigg(\frac {\lambda_{N,k}}{\ell_N}-\alpha_{n,k}\bigg)M[u^k]}
		\nonumber
		\\&\leq \frac {(\tau+1) n}{\ell_N}
		+\frac 1m+\sum_{k=1}^n\Big|\frac {\lambda_{N,k}}{\ell_N}-\alpha_{n,k}\Big|(|u^k|-1).
		\label{eqloc: majoration v_N mu}
	\end{align}
	When $N\to \infty$, the first term vanishes. Moreover, we have $|v^N|=N+o(N)$, thus for all $1\leq k\leq n$, the coefficient $\lambda_{N,k}/\ell_N$ converges to $\alpha_{n,k}$, thus each term of the sum vanishes. It follows that the sum vanishes because $n$ is fixed. 
	Consequently, there exists $N(m)\in \N$ such that for all $N\geq N(m)$,
	\begin{equation*}
		\tvnorm{L[v^N]-\mu}\leq \frac 1m+\frac 1m+\frac 1m.
	\end{equation*}
	The word $w^m:=v^{N(m)}$ satisfies $\mathbb P_{|w^m|}(v^m)>0$ and $\tvnorm{L[w^m]-\mu}\leq 3/m$. Since it is always possible to choose $N(m+1)\geq N(m)+1$, the length of $w^m$ is at least $m$, thus $|w^m|\to \infty$ as $m\to \infty$. 
\end{proof}
\begin{remark}\label{remark: set M fortelle}
	Proposition~\ref{prop: characterization A2bal} is adapted from Proposition~1 of~\cite{fortelle2002}, which is, unfortunately, false as formulated. Indeed, the set described in this proposition actually coincides with $\Adebal\cap \ith \abscont2$, which sometimes differs from the set $\mathcal M$ defined in~\cite{fortelle2002}. See Example~\ref{ex: right only random walk} for instance. This overlook seems to have no consequence in the proofs in \cite{fortelle2002}, provided we replace the definition of $\mathcal M$ with that of $\Adebal\cap \ith \abscont2$.
\end{remark}	
\subsection{Admissible measures for infinite $k$}
\label{section: admissible measures on the process level}
In this appendix, we study the set $\ith \Abal \infty$ introduced in Definition~\ref{def: admissibility on the process level}. We show that $\ith \Abal \infty$ is the projective limit of $(\ith \Abal k)_{k\geq 1}$, defined in Definition~\ref{def: admissible measures}. 
The objects of this appendix, for instance measures $\ith \mu k$, where $\mu\in \mathcal P(\spaces^\N)$, and the applications $\pi_k:\mu\mapsto \ith \mu k$, were introduced at the beginning of Section~\ref{section: dawson gartner}.
\begin{lemma}
	\label{lemma: A infty projective limit of Ak}
	The set $\ith \Abal \infty$ is the projective limit of $(\ith \Abal k)$. In other words, $\ith \Abal \infty$ is closed and satisfies
	\begin{equation}\label{eq: proj lim Abalk}
		\ith \Abal \infty=\bigcap_{k=1}^\infty \pi_k^{-1}(\ith \Abal k).
	\end{equation}
\end{lemma}
\begin{proof}
	Since every $\ith \Abal k$ is closed and every $\pi_k$ is continuous, it suffices to prove~\eqref{eq: proj lim Abalk}.
	Let $\mu \in\ith \Abal \infty$ and $k\geq 1$. The measure $\ith \mu k$ is balanced, and satisfies
	\begin{equation}
		\label{eqloc: decomposition muk}
		\ith \mu k=\pi_k\Bigg(\sum_{j\in\mathcal J_\mu}\mu|_{C_j^\N}\Bigg)=\sum_{j\in\mathcal J_{\mu}}\pi_k\left(\mu|_{C_j^\N}\right)=\sum_{j\in\mathcal J_{\mu}}\ith \mu k|_{C_j^k}.
	\end{equation}
	Thus $\ith \mu k$ is pre-admissible and $\mathcal J_{\ith \mu k}\subseteq\mathcal J_\mu$.
	Moreover, for all $j\in {\mathcal J}$,
	\begin{equation*}
		\ith \mu {k}(C^{k}_j)=\mu (C^{k}_j\times \Omega)=\mu (C^\N_j),
	\end{equation*}
	thus $\mathcal J_\mu$ is minimal in~\eqref{eqloc: decomposition muk}
	and we have $\mathcal J_{\ith \mu k}=\mathcal J_\mu$. 
	Since the order $\leadsto $ is total on $(C_j)_{j\in {\mathcal J}_\mu}$ and $\beta\leadsto C_j$ for all $j\in {\mathcal J}_\mu$, it follows that $\ith \mu k\in \ith \Abal k$. 
	We now prove the converse inclusion. Let $\mu \in \bigcap_{k=1}^\infty \pi_k^{-1}(\ith \Abal k)$. Then, $\mu$ is shift-invariant.
	It remains to prove that $\mu\in \ith {\mathcal A} \infty$.
	More precisely, we will show that $\mu$ is pre-admissible with $\mathcal J_\mu=\mathcal J_{\ith \mu 1}$, and therefore admissible.
	It suffices to show that 
	\begin{equation}\label{eqloc: equivalence mu infty admissible}
		\sum_{j\in \mathcal J_{\ith \mu 1}}\mu(C^\N_j)=1.
	\end{equation}
	Let $k\in \N$. For all $j\in\mathcal J$, we have
	\begin{equation*}
		\ith \mu {1}(C^{}_j)= \mu (C^{}_j\times \spaces ^{\N})=\ith \mu k(C_j\times \spaces^{k-1})=\ith \mu k(C_j^k),
	\end{equation*}
	where the last equality uses that $\ith \mu k$ is pre-admissible. Thus, the sequence $(\ith \mu {k}(C_j^{k}))_{k\geq 1}$ is constant for every $j\in \mathcal J$.
	Therefore, 
	\begin{equation*}
		\mu(C_j^\N)=\mu\bigg(\bigcap_{k\in \N}C_j^k \times \spaces^\N\bigg)=\lim_{k\to \infty}\ith \mu k(C_j^k)=
		\ith \mu 1(C_j).
	\end{equation*}
	By definition of $\mathcal J_{\ith \mu1}$, \eqref{eqloc: equivalence mu infty admissible} is satisfied.
\end{proof}
\section{Convex conjugates and duality}
\label{section: duality}
Let $k\in \N$. We recall that $\bounded(\spaces^k)$ denotes the set of all bounded functions on $\spaces^k$ and $\mathcal M(\spaces^k)$ denotes the set of all finite signed measures on $\spaces^k$.
In addition to the weak topology on $\mathcal M(\spaces^k)$ defined in the introduction, we define the \emph{weak topology} on $\bounded(\spaces^k)$ as the coarsest topology such that $\langle \mu, \cdot\rangle$ is continuous for all $\mu\in \mathcal M(\spaces^k)$.
The sets $\mathcal M(\spaces^k)$ and $\bounded(\spaces^k)$ are equipped with their respective weak topologies.\footnote{We follow the same philosophy as in Chapter 2.3 of~\cite{zalinescu2002}: 
	rather than taking a Banach space and considering its dual equipped with the operator norm, we define simultaneously two topologies on two sets that make them the dual of each other. Beware that this weak topology on $\bounded(\spaces^k)$ is not the topology of uniform convergence!}
We recall that $\mathcal P(\spaces^k)$ denotes the set of probability measures on $\spaces^k$.
\begin{definition}
	\label{def: convex conjugate}
	Let $\Lambda:\bounded(\spaces^k)\to (-\infty,+\infty]$ be a function. The \emph{convex conjugate} of $\Lambda$ is defined over $\mathcal M(\spaces^k)$ by
	\begin{align*}
		\Lambda^*(\mu)=\sup_{V\in\bounded(\spaces^k)}\left(\langle \mu, V\rangle -\Lambda(V)\right).
	\end{align*}
	Let $J:\mathcal M(\spaces^k)\to (-\infty,+\infty]$ be a function.
	The convex conjugate of $J$ is defined over $\bounded(\spaces^k)$ by
	\begin{equation*}
		J^*(V)=\sup_{\mu\in\mathcal M(\spaces^k)}\left(\langle \mu, V\rangle-J(\mu)\right).
	\end{equation*} 
	Let $I:\mathcal P(\spaces^k)\to (-\infty,+\infty]$ be a function. Let $J$ be the extension of $I$ to $\mathcal M(\spaces^k)$ defined by setting $J(\mu)=\infty$ for all $\mu\notin \mathcal P(\spaces^k)$. The convex conjugate of $I$ is defined over $\bounded(\spaces^k)$ as the convex conjugate of $J$ and is still denoted $I^*$.
\end{definition}
The convex conjugate of a function on $\bounded(\spaces^k)$, $\mathcal M(\spaces^k)$, or $\mathcal P(\spaces^k)$ is always a lower semicontinuous function because it is the supremum of a family of continuous functions.
The notion of convex conjugation is designed for vector spaces, which is the reason why we need to extend the function $I:\mathcal P(\spaces^k)\to (-\infty,+\infty]$ to $\mathcal M(\spaces^k)$ before defining its convex conjugate. Nevertheless, we still have
\begin{equation*}
	I^*(V)=\sup_{\mu\in\mathcal P(\spaces^k)}\left(\langle \mu, V\rangle-I(\mu)\right),
	\qquad V\in \bounded(\spaces^k),
\end{equation*}
as a consequence of Definition~\ref{def: convex conjugate}.

Let $I:\mathcal P(\spaces^k)\to(-\infty,+\infty]$ be a function.
If $I$ is finite at one point (at least), $I^*$ is a function from $\bounded(\spaces^k)$ to $(-\infty,+\infty]$.
The convex biconjugate of $I$ is the function $I^{**}:=(I^*)^*$ defined on $\mathcal M(\spaces^k)$. 
Otherwise, $I^{**}$ is set to be the infinite function.
The function $I^{**}$ is described by the Fenchel-Moreau Theorem; see for instance Theorems~2.3.3 and~2.3.4 of~\cite{zalinescu2002}. In this paper, we only use the following corollary. For a function $I:\mathcal P(\spaces^k)\to (-\infty,+\infty]$, we denote by $\cl I$ the lower semicontinuous envelope of $I$ and by $\conv I$ the convex envelope of $I$.
\begin{corollary}
	\label{coro: fenchel moreau}
	Let $I:\mathcal P(\spaces^k)\to (-\infty,+\infty]$ be a function. 
	Then, for all $\mu\in \mathcal P(\spaces^k)$,
	\begin{equation*}
		I^{**}(\mu)=\cl(\conv I)(\mu).
	\end{equation*}
\end{corollary}
\begin{proof}
	If $I$ is infinite everywhere, we have $I^{**}=\infty=\conv I$ everywhere.
	Otherwise, let $J$ be the extension of $I$ to $\mathcal M (\spaces^k)$ defined as in Definition~\ref{def: convex conjugate}. 
	By the Fenchel-Moreau Theorem (see Theorem~2.3.4 of~\cite{zalinescu2002}), and by definition of $I^{**}$, we have
	\begin{equation*}
		I^{**}(\mu)=J^{**}(\mu)=\cl_{\mathcal M(\spaces^k)}(\conv _{\mathcal M(\spaces^k)}J)(\mu),
		\qquad \mu\in \mathcal P(\spaces^k),
	\end{equation*}
	where $\cl_{\mathcal M(\spaces^k)}(\cdot)$ denotes the lower semicontinuous envelope taken on $\mathcal M(\spaces^k)$ and $\conv _{\mathcal M(\spaces^k)}(\cdot)$ denotes the convex envelope taken on $\mathcal M(\spaces^k)$. 
	Since $\mathcal P(\spaces^k)$ is convex in $\mathcal M(\spaces^k)$, by definition of $J$, we have, for all $\mu \in \mathcal M(\spaces^k)$,
	\begin{equation*}
		\conv_{\mathcal M(\spaces^k)}J(\mu)=
		\begin{cases}
			\conv I(\mu),\quad &\hbox{if $\mu\in \mathcal P(\spaces^k)$},
			\\\infty,\quad &\hbox{otherwise}.
		\end{cases}
	\end{equation*}
	Since $\mathcal P(\spaces^k)$ is closed in $\mathcal M(\spaces^k)$, it follows that for all $\mu\in \mathcal P(\spaces^k)$,
	\begin{equation*}
		\cl_{\mathcal M(\spaces^k)}(\conv _{\mathcal M(\spaces^k)}J)(\mu)=\cl(\conv I)(\mu),
	\end{equation*}
	which concludes the proof.
\end{proof}

\section{The finite case}\label{section: finite case}
In this section, we consider the case of a finite state space $\spaces$. This assumption, as it implies that $\mathcal P(\spaces)$ is compact, provides tightness and goodness of rate functions. By doing so, it greatly simplifies many proofs in the paper.
Nevertheless, we stress that finiteness does not imply any kind of irreducibility.
In the following theorem, functions $(\ith \Lambda1)^*$, $\ith \idv1$, $\ith R1$, $(\ith \Lambda\infty)^*$, $\ith \idv\infty$ and $\ith R\infty$ and sets $\Aun$ and $\ith \Abal \infty$ are as in Theorems~\ref{theorem: intro weak LDP for Ln1} and~\ref{theorem: intro weak LDP for Lninfty}.
\begin{theorem}\label{theorem: full LDP finite case}
	Assume that $\spaces$ is finite. Then, the following hold:
	\begin{enumerate}
		\item\label{item: level 1 full LDP finite case} {(Level-1 full LDP)} For all $f:\spaces\to \mathbb R^d$, the sequence $(A_nf)$ satisfies the full LDP with good rate function $I_f$, which satisfies
		\begin{equation*}
			I_f(a)=\inf\{\ith I1(\mu)\ |\ \mu\in \mathcal P(\spaces),\ \langle \mu, f\rangle =a\},
			\qquad a\in \mathbb R^d,
		\end{equation*}
		where $\ith I1$ is the function of Property~\ref{item: level 2 full LDP finite case};
		\item \label{item: level 2 full LDP finite case}(Level-2 full LDP) The sequence $(\ith L1_n)$ satisfies the full LDP with good rate function $\ith I1$, which satisfies, for all $\mu \in \mathcal P(\spaces)$,
		\begin{equation*}
			\ith I1(\mu)=
			\begin{cases}
				(\ith \Lambda1)^*(\mu)=\ith \idv1(\mu)=\ith R1(\mu),\quad &\hbox{if $\mu\in \Aun$},\\
				\infty,\quad &\hbox{otherwise};
			\end{cases}
		\end{equation*}
		\item \label{item: level 3 full LDP finite case}(Level-3 full LDP) The sequence $(\ith L\infty_n )$ satisfies the full LDP with good rate function $\ith I\infty$, which satisfies, for all $\mu \in \mathcal P(\mathcal \spaces^\N)$,
		\begin{equation*}
			\ith I\infty(\mu)=
			\begin{cases}
				(\ith \Lambda \infty)^*(\mu)=\ith \idv \infty (\mu)=\ith R\infty(\mu),\quad &\hbox{if $\mu\in \ith \Abal \infty$},\\
				\infty,\quad &\hbox{otherwise}.
			\end{cases}
		\end{equation*}
		Moreover, assuming that $\mu \in \ith \Abal \infty$ and $\ent (\ith \mu1|\beta)<\infty$, we have the additional expression
		\begin{equation*}
			\ith I\infty(\mu)=\lim_{k\to \infty}\frac 1k \ent (\ith \mu k|\mathbb P_k).
		\end{equation*}
	\end{enumerate}
\end{theorem}
Theorem~\ref{theorem: full LDP finite case} could be obtained as a corollary of Theorems~\ref{theorem: intro weak LDP for Ln1} and~\ref{theorem: intro weak LDP for Lninfty}.
However, in order to highlight how the finiteness assumption simplifies many proofs, we propose a sketch of a direct proof by going through the paper again with this assumption in mind. 
\begin{proof}[Sketch of proof]
	Here are the main steps of the proof, and how much simpler they are.
	\begin{itemize}
		\item The slicing, stitching, coupling and decoupling  maps defined in Sections~\ref{section: slicing and stitching} and~\ref{section: coupling and decoupling maps} have the same definition as in the infinite case. However, their properties are simpler to prove. Indeed, since $\spaces$ is now finite, we do not need to introduce a finite set $K\subseteq \spaces$ to be able to bound probabilities and lengths of words (that is, we can simply take $K=\Kdelta=\spaces$ throughout). Moreover, the sums that appear, for example in~\eqref{eqloc: bound on the sum of Q(u)}, can be bounded by simply counting the number of terms and crudely replacing the probabilities being summed by 1; see Example 2.20 of~\cite{papierdecouple}.
		
		\item We then use the RL method of Section~\ref{section: ruelle lanford functions} to show that $(\ith L2_n)$ has a RL function. By Lemma~\ref{lemma: RL implies LDP}, $(\ith L2_n)$ satisfies the weak LDP with rate function $\Ide$. But since $\mathcal P(\spaces^2)$ is compact, the LDP is full and the rate function good.
		\item We derive Proposition~\ref{prop: I=I**} using convexity properties of $\Ide$ as in Section~\ref{section: convexity}. 
		\item By Varadhan's Lemma (the standard version for full LDPs; see Section III.3 in~\cite{denhollander2008} or Theorem 4.3.1 of~\cite{DZ}), we have $(\Ide)^*=\ith \Lambda2=\ith \Lambda2_\infty$ everywhere. Section~\ref{section: varadhan} can be skipped entirely.
		\item The content of Section~\ref{section: alternative expressions of the rate function} is standard in the finite case.
		To derive $\ith J2=\ith R2$, one can reproduce the proof of theorem 13.1 of~\cite{rassoul}, as we did in Section~\ref{section: alternative expressions of the rate function}, but an alternative proof, that gets simpler in finite dimension, consists in explicitly optimizing~\eqref{eq: definition DV entropy}.\footnote{The supremizer is the function $(x,y)\mapsto \indic _{p(x,y)>0}\mu(x,y)/\ith \mu 1(x)p(x,y)$, though it does not belong to $\expbounded(\spaces^2)$, and should be approximated to get $\ith \idv2$.}
		We then have to prove that $(\ith \Lambda2)^*=\ith\idv2$ on $\Adebal$. 
		We outline here a proof of this standard equality, following the same approach as in the proof of Proposition~\ref{prop: rate functions for occupation time}.\footnote{Since $\ith \Lambda2$ and $\ith \Lambda2_\infty$ are equal, there is no need to reproduce the full proof of Proposition~\ref{prop: equality of all rate functions on Adebal} for comparing $\ith J2$ with $(\ith \Lambda2_\infty)^*$.}
		Let $\mu\in \Adebal$.
		We deduce the inequality $\ith \idv 2(\mu)\leq (\ith \Lambda2)^*(\mu)$ like in step 1 of the proof of Proposition~\ref{prop: equality of all rate functions on Adebal}.
		The proof of the converse inequality is simplified in the following way.
		Since the set $\spaces_\beta:=\{x\in \spaces\ |\ \beta\leadsto x\}$ is finite, for all $V\in \bounded(\spaces^2)$,
		\begin{equation*}
			\begin{split}
				\ith \Lambda2(V)
				&=\max_{x\in \spaces_\beta}\left(\limsup_{n\to \infty}\frac 1n\log \mathbb E_x\left[e^{n\langle \ith L2_n,V\rangle}\right] \right)=:\ith{\widetilde \Lambda}2(V).
			\end{split}
		\end{equation*} 
		Here, $\mathbb E_x[\cdot]$ denotes the expectation associated with the Markov chain conditioned to $(X_1=x)$.
		For all $\lambda>\ith{\widetilde \Lambda}2(V)$, the expression
		\begin{equation*}
			u_n(x)=\sum_{k=0}^n e^{-\lambda k}(P^V)^k1(x),\qquad x\in \spaces,
		\end{equation*}
		defines a function $u_n\in \expbounded(\spaces^2)$. After observing that $Pu_n=e^{V-\lambda}(u_{n+1}-1)$ and that $(u_n)$ converges pointwise to a function $u$ satisfying $u(x)>1$ for all $x\in \spaces_\beta$, we obtain 
		\begin{equation*}
			\ith J2(\mu)\geq \langle \mu, V\rangle -\lambda+\sum_{x\in \spaces_\beta}\mu(x)\liminf_{n\to \infty}\log\frac{u_n(x)}{u_{n+1}-1}\geq \langle \mu, V\rangle -\lambda+0.
		\end{equation*}
		This shows $\ith \idv 2(\mu)\geq (\ith \Lambda2)^* (\mu)$.
		Notice that we did not use Lemma~\ref{lemma: balanced measure decomposition} in this reasoning.
		\item Since the LDP is full, the content of Section~\ref{section: contraction} can be entirely replaced by the simple use of the standard contraction principle (see Section III.5 in~\cite{denhollander2008} or Theorem 4.2.1 of~\cite{DZ}) to obtain the full LDP for $(\ith L1_n)$. Recall also that, in the infinite case, Remark~\ref{remark: no level 1 LDP} prevented us from using Lemma~\ref{lemma: contraction principle} to obtain the LDP for $(A_nf)$. In the finite case, the standard contraction principle yields the full LDP for $(A_nf)$.
		\item We use the same proofs as in Section~\ref{section: weak LDP on the process level} to get the level-3 full LDP. In the proof of Theorem~\ref{theorem: intro weak LDP for Lninfty}, we can use the Dawson-Gärtner Theorem.
	\end{itemize}
\end{proof}
\begin{remark}
	An even shorter route to the level-2 full LDP of Theorem~\ref{theorem: full LDP finite case} is to perform the same steps while replacing $\ith L2_n$ by $\ith L1_n$ everywhere. This makes the formulation of certain statements simpler and skips the use of a contraction principle. This approach would not have worked in the infinite $\spaces$ case, because we needed measures on $\spaces^2$ in Section~\ref{section: alternative expressions of the rate function}. More precisely, Step 3 of the proof of Proposition~\ref{prop: equality of all rate functions on Adebal} used the decomposition of balanced measures given in Lemma~\ref{lemma: balanced measure decomposition}.
\end{remark} 
\section{Examples and counterexamples}\label{section: examples and counterexamples}
\begin{example}\label{ex: simple example three states}
	\begin{figure}[thb]
		\centering
		\includegraphics{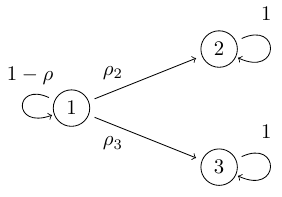}
		\caption{The 3-state Markov chain of Example~\ref{ex: simple example three states}.}
		\label{figure: simple example three states}
	\end{figure}
	Let $\spaces=\{1,2,3\}$, $\beta=\delta_1$, and 
	\begin{equation*}
		p=\begin{pmatrix}
			1-\rho&\rho_2&\rho_3\\
			0&1&0\\
			0&0&1
		\end{pmatrix},
	\end{equation*}
	where $\rho_2,\rho_3\in [0,1]$ are such that $\rho=\rho_2+\rho_3\leq1$.
	Although very simple and easily solvable by hand, the problem in this example does not fit into the irreducibility framework. We get
	\begin{equation*}
		\begin{split}
			\Aun&=\{\mu \in \probaun\ |\ \mu(2)\mu(3)=0\}\\
			&=\{\lambda \delta_1+(1-\lambda)\delta_2\ |\ \lambda\in [0,1]\}\cup\{\lambda \delta_1+(1-\lambda)\delta_3\ |\ \lambda\in [0,1]\},
		\end{split}
	\end{equation*}
	and $\ith L1_n$ satisfies the full LDP with rate function $I$, where
	\begin{equation*}
		I(\mu)=\begin{cases}
			-\mu(1)\log(1-\rho),\quad& \hbox{if $\mu\in \Aun$},\\
			\infty,\quad& \hbox{otherwise.}
		\end{cases}
	\end{equation*}
	$\ith I1$ is not convex, since $I(\frac12\delta_{2}+\frac12\delta_{3})=\infty$ and $I(\delta_{2})=I(\delta_{3})=0$. The same argument holds for $\ith I2$.
\end{example}
\begin{example}[The right-only one-dimensional random walk]\label{ex: right only random walk}
	\begin{figure}[thb]
		\centering
		\includegraphics{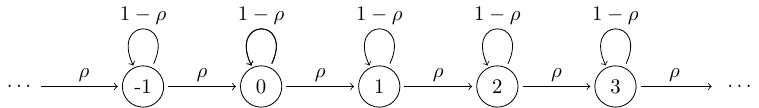}
		\caption{The right-only one-dimensional random walk on $\mathbb Z$.}
		\label{fig: right only one dimensional random walk}
	\end{figure}
	Let $\spaces=\mathbb Z$ be equipped with the stochastic kernel defined by $p(x,\cdot)=\rho\delta_{x+1}+(1-\rho)\delta_{x}$, where $\rho\in [0,1]$ is a parameter. This kernel is shown in Figure~\ref{fig: right only one dimensional random walk}. Let $\beta\in\mathcal P(\spaces)$ satisfy $\inf \supp\beta=-\infty$.
	If $\rho=1$, there are no irreducible classes. In this case $\Aun=\emptyset$ and $\ith L1_n$ satisfies the weak LDP with rate function $I_1= \infty$. If $\rho\in (0,1)$, all singletons are irreducible classes, and $\Aun=\mathcal P(\spaces)$. Then, $\ith L1_n$ satisfies the weak LDP with rate function $I_\rho$ defined as the constant function $-\log (1-\rho)$. If $\rho=0$, all singletons are irreducible classes, but none is reachable from another. Thus $\Aun=\{\delta_x\ |\ x\in\supp\beta\}$ and $(L_n^{(1)})$ satsfies the weak LDP with rate function $I_0$ defined by
	\begin{equation*}
		I_0(\mu)=
		\begin{cases}
			0,\quad &\hbox{if $\exists x\in \supp \beta$, $\mu=\delta_x$ },\\
			\infty,\quad &\hbox{otherwise.}
		\end{cases}
	\end{equation*}
	This example illustrates several interesting facts:
	\begin{enumerate}
		\item When $\rho=1$, the functions $\ith \Lambda1$ and $\ith {\widetilde \Lambda}1$ of the proof of Proposition~\ref{prop: rate functions for occupation time} may not coincide. Let $V(x)=\indic_{x\geq 0}$ and let $\beta = \sum_{k\geq 1}2^{-k} \delta_{-k}$. Then, $\ith \Lambda1 (V)= 1-\log 2>0=\ith {\widetilde \Lambda}1(V)$.
		\item When $\rho\in (0,1)$, no state is recurrent, yet the rate function is not everywhere infinite.
		\item When $\rho\in (0,1)$, $\Aun=\mathcal P(\spaces)$ and $I_\rho$ is convex, but the Markov chain is not irreducible. Although the convexity of the rate function is closely related to the irreducibility of the Markov chain, it is not a sufficient condition.
		\item When $\rho\in (0,1)$, the set $\mathcal M$ defined in~\cite{fortelle2002} and mentioned in Remark~\ref{remark: set M fortelle} is the set of balanced measures whose support is bounded from below. It does not coincide with $\Adebal\cap \ith \abscont 2$. As a consequence, the statement provided in Theorem~7 of~\cite{fortelle2002} is false. For instance, the measure $\mu=\sum_{x<0}2^x\delta_{(x,x)}$ is admissible, thus $I_\rho(\mu)=-\log (1-\rho)<\infty$, but $\mu\notin \mathcal M$. As mentioned in Remark~\ref{remark: set M fortelle}, Theorem~7 of~\cite{fortelle2002} holds if the definition of $\mathcal M$ is replaced by that of $\Adebal\cap \ith \abscont2$.
		\item When $\rho\in (0,1)$, one can compute that $\Lambda_\infty(0)=\log \rho<0=\Lambda(0)$, showing that the inequality $\Lambda_\infty\leq \Lambda$ is not an equality in general.
		\item When $\rho\in (0,1)$, $\Lambda$ is not lower semicontinuous in the weak topology on $\bounded(\spaces^2)$.\footnote{Yet it is $1$-Lipschitz, thus continuous with respect to the uniform convergence topology. The weak topology is defined in Appendix~\ref{section: duality}.} If it were, since it is also convex, the Fenchel-Moreau Theorem would cause $\Lambda^{**}=\Lambda$ on $\bounded(\spaces^2)$ (see Theorem~2.3.3 of~\cite{zalinescu2002} for instance). Yet this equality is false because $\Lambda^{**}(0)=I_\rho^*(0)=\log(1-\rho)<0=\Lambda(0)$.
		\item When $\rho=0$, let $\conv I_0$ denote the convex envelope of $I_0$. One can compute that $\conv I_0(\mu)$ is zero if the support of $\mu$ is finite and $\infty$ otherwise. This function is not lower semicontinuous. Moreover,  both $I_0^{**}$ and the $\sigma$-convex envelope of $I_0$ are the zero function. Hence, $\conv I_0$ does agree with $I_0^{**}$ on $\Aun$ but not on $\mathcal P(\spaces)$.
	\end{enumerate}
\end{example}
\begin{example}\label{ex: climber's markov chain}
	\begin{figure}[thb]
		\centering
		\includegraphics{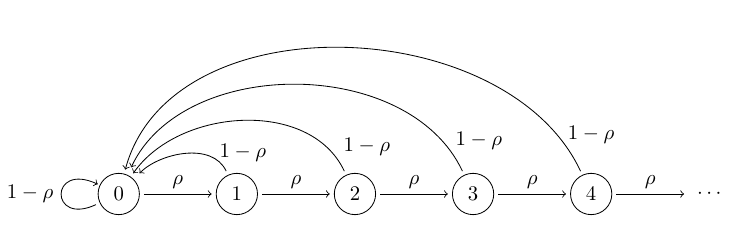}
		\caption{The Markov chain of Example~\ref{ex: climber's markov chain}. The state $0$ is reachable from every other state in one step and the state $x$ is reachable from $0$ in $x$ steps, thus the Markov chain is matrix-irreducible.}
		\label{fig: climber's markov chain}
	\end{figure}
	Let $\spaces=\mathbb N_0=\{0,1,2,\ldots\}$, and $p(x,\cdot)=\rho \delta_{x+1}+(1-\rho)\delta_{0}$ for all $x$, where $\rho$ is a parameter in $(0,1)$. The kernel $p$ is represented in Figure~\ref{fig: climber's markov chain}. This Markov chain is matrix-irreducible. As mentioned in the introduction, this example shows that, in absence of exponential tightness, it is a very subtle question to determine whether the LDP is full or not. For every initial measure $\beta$, the empirical measure satisfies the weak LDP. If $\beta=\delta_0$, one can show that this LDP is actually full. If $\beta\neq \delta_0$, the empirical measure fails to satisfy the full LDP, as proved in detail in Example~10.3 of~\cite{deacosta2022} and in~\cite{baxter1991}.
\end{example}
\begin{example}
	\label{ex: infinite entropy}
	Let $\spaces=\N$ and $p$ be defined by $p(n,\cdot)=e^{-n}\delta_{n}+(1-e^{-n})\delta_{n+1}$ for all $n\in \N$. Let $\beta=\delta_1$. Then, we have
	\begin{equation*} 
		\Adebal\cap \ith \abscont 2=\Adebal=\mathcal P(\{(n,n)\ |\ n\in\N\}).
	\end{equation*}
	Let
	\begin{equation*} 
		\mu=\sum_{n=1}^\infty\frac{6}{\pi^2n^2}\delta_{(n,n)}\in \Adebal \cap \ith \abscont2.
	\end{equation*}
	By Theorem~\ref{theorem: intro weak LDP for Ln2}, $\mu$ satisfies
	\begin{equation*}
		\Ide(\mu)=\ith R2(\mu)=\sum_{n=1}^\infty\frac{6}{\pi ^2n^2}n=\infty.
	\end{equation*}
	In this example, the inclusion of the domain of $\Ide$~---~that is the set $\{\mu\in \proba\ |\ \Ide(\mu)<\infty\}$~---~in $\Adebal \cap \ith \abscont 2$ is strict.
\end{example}
In the following example, from~\cite{brycdembo1996}, the Markov chain satisfies a level-2 weak LDP but fails to satisfy a level-1 weak LDP. By doing so, it disproves any general statement of a contraction principle from a weak level-2 LDP to a weak level-1 LDP. 
\begin{example}
	\label{ex: bryc and dembo's counterexample for level 1}
	\begin{figure}[thb]
		\centering
		\includegraphics{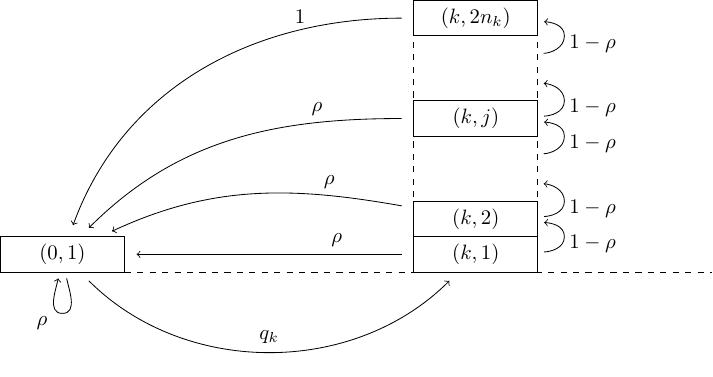}
		\caption{Visual representation of the Markov chain of example~\ref{ex: bryc and dembo's counterexample for level 1}. There is one `tower' as represented on the right of the picture per value of $k\in \N$.}
		\label{figure: climber's markov chain}
	\end{figure}
	Let $n_k=3^k$ if $k\geq 1$ and $n_0=1/2$, and let $\spaces=\{(k,j)\in \mathbb N_0^2\ |\ 1\leq j\leq 2n_k\}$. Let $p$ be defined by Figure~\ref{figure: climber's markov chain}, with the following parameters:
	$\rho\in (0,1)$ such that $\alpha:=1-4\log (1-\rho)>0$, and $q_k=C_1e^{-\alpha n_k}$, where $C_1$ is such that $q_1+q_2+\ldots=1-\rho$.
	Let $\beta(k,j)=C_2(1-\rho)^{j-1}q_k$ where $C_2$ is a normalizing constant. 
	This is a matrix-irreducible Markov chain. We consider the following observable:
	\begin{equation*}
		f(k,j)=
		\begin{cases}
			+1,\quad &j\geq n_k+1,\\
			-1,\quad &j\leq n_k.
		\end{cases}
	\end{equation*}
	Suppose that the Markov chain satisfies a level-1 weak LDP. Denote $A_n=A_nf=\frac 1n\sum_{i=1}^nf(X_i)$. Since $(A_n)$ takes its values in the compact set $[-1,1]$, it must satisfy the full LDP. Therefore, by Varadhan's lemma, the limit 
	\begin{equation*}
		\lim_{n\to\infty}\frac 1n\log\mathbb E\left[e^{3\alpha A_n}\right]
	\end{equation*} 
	exists. But~\cite{brycdembo1996} provided a proof that it actually does not converge. 
	Eventually, this shows by contradiction that the level-1 weak LDP is not valid. The level-2 weak LDP remains valid by Theorem~\ref{theorem: intro weak LDP for Ln1}.
\end{example}
\begin{example}[The Wright-Fisher model for genetic drift]
\label{ex: wright fisher}
In genetics, a gene can have several versions called \emph{alleles}. 
We consider a gene that has $k$ different alleles $A_1,\ldots , A_k$, and study the evolution of the frequencies of alleles $A_1,\ldots, A_k$ in a population over time. We consider a population of $N$ individuals, each one being characterized by their allele.\footnote{For simplicity, we study haploid cells, which are cells having a single copy of each chromosome; the gametes are typically haploid.} At the beginning, each allele is present, at least once, in the population. At each time step, a new generation of $N$ children is created and replaces the previous generation. Each child inherits the allele of their parent, which is chosen uniformly 
among the $N$ individuals of the previous generation. Let $\spaces=\{x\in \N_ 0^k\ |\ x_1+\ldots +x_k=N\}$ and let $X_{n,i}\in \{0, 1,\ldots ,N\}$ denote the number of individuals having the allele $A_i$ at the $n$-th generation. Then, $(X_n)$ is a Markov chain on $\spaces$ with stochastic kernel
\begin{equation*}
	p(x,y)=N!\prod_{i=1}^k\Big(\frac1{y_i !}\Big(\frac {x_i}N\Big)^{y_i}\Big),
	\qquad (x,y)\in \spaces^2.
\end{equation*}

Let us consider at first the case $k=2$, which is the most common formulation of the Wright-Fisher model. Then the state space reduces to $\spaces =\{0, 1, \ldots, N\}$ and $X_n\in \spaces$ denotes the number of individuals having the allele $A_1$ at the $n$-th generation. Almost surely, $X_n=0$ or $X_n=N$ for some $n$, which correspond to the extinction among the population of the allele $A_1$ or $A_2$, respectively.
Extinction scenarios provide the reducibility of the Markov chain. The irreducible classes are $\{0\}$, $\{N\}$, and $\{1,\ldots, N-1\}$. Hence, $\Aun=\{\mu\in \probaun \ |\ \mu(0)\mu(N)=0\}$, and the empirical measure satisfies the (full) LDP with non-convex rate function given by~\eqref{eq: rate functions I1}.

The reducibility structure becomes even more complex in cases $k\geq 3$; there are $2^k-1$ irreducible classes $(C_E)$, indexed by non-empty subsets of $\{1,\ldots, k\}$ defining which alleles exist in the population. If $E_1$ and $E_2$ are non-empty subsets of $\{1,\ldots, k\}$, the class $C_{E_2}$ is reachable from the class $C_{E_1}$ if and only if $E_2\subseteq E_1$. A measure $\mu\in \probaun$ is admissible if and only if there are no distinct indices $i,j\in\{1,\ldots, k\}$ such that $\langle \mu, \indic_{\{x_i>0,\,x_j=0\}}\rangle>0$ and $\langle \mu, \indic_{\{x_i=0,\,x_j>0\}}\rangle>0$ simultaneously. The empirical measure satisfies the (full) LDP with non-convex rate function given by~\eqref{eq: rate functions I1}.

Both the LDPs of the cases $k=2$ and $k\geq 3$ are out of the range of the theory of large deviations of irreducible Markov chain.
\end{example}

\section{Glossary}
\setlength\columnsep{20pt}
\begin{multicols}{2}
	\subsubsection*{Markov chain}
	\begin{itemize}
		\setlength\itemsep{-0.3em}
		\item[$\spaces$,] countable state space, \pageref{not: S}
		\item[$(X_n)$,] Markov chain, \pageref{not: X}
		\item[$\beta$,] initial measure, \pageref{not: beta}
		\item[$p$,] transition kernel, \pageref{not: p}, \pageref{not: p words}
		\item[$\mathbb P$,] law of the Markov chain, \pageref{not: P}
		\item[$\mathbb P_t$,] induced law on $\spaces^t$, \pageref{not: Pt}
		\item[$\mathbb E$,] expectation, \pageref{not: E}
		\item[$\leadsto$,] reachability relation, \pageref{not: reachable}
		\item[$(C_j)$,] irreducible classes, \pageref{not: Cj}
		\item[$\mathcal J$,] \pageref{not: Jcal}
		\item[$L_n$,] short for $\ith L2_n$, \pageref{not: Ln2}
		\item[$\ith Lk_n$,] empirical measure, \pageref{not: Ln1}, \pageref{not: Ln2}, \pageref{not: Lnk}
	\end{itemize}
	\subsubsection*{Sets}
	\begin{itemize}
		\setlength\itemsep{-0.3em}
		\item[$\bounded(\cdot)$,] bounded measurable functions, \pageref{not: Bounded}
		\item[$\expbounded(\cdot)$,] \pageref{not: expbounded}, \pageref{not: expbounded 1}, \pageref{not: expbounded k}, \pageref{eq: definition idv infty}
		\item[$\continuous(\cdot)$,] continuous functions, \pageref{not: Ccal}
		\item[$\unifcont(\cdot)$,] uniformly continuous functions, \pageref{not: Cu}
		\item[$\expbounded_\mathrm{u}(\cdot)$,] \pageref{eq: alternative expression level-3 IDV}
		\item[$\mathcal M(\cdot)$,] finite signed measures, \pageref{not: Mcal}
		\item[$\mathcal P(\cdot)$,] probability measures, \pageref{not: Pcal}, \pageref{not: Pcal infty}
		\item[$\balanced(\cdot)$,] balanced probability measures, \pageref{not: Pbal}, \pageref{not: Pbalinfty}
		\item[$ \ith {\mathcal A} k$,] admissible measures, \pageref{not: Ak}, \pageref{not: Ainfty}
		\item[$ \ith \Abal k$,] admissible balanced measures, \pageref{not: Akbal}, \pageref{not: Abal infty}
		\item[$\ith \abscont k$,] \pageref{not: abscont}, \pageref{not: abscont k}	
	\end{itemize}
	\subsubsection*{Functions and rate functions}
	\begin{itemize}
		\setlength\itemsep{-0.3em}
		\item[$I$,] short for $\ith I2$, \pageref{not: I2}
		\item[$\ith I k$,] rate function, \pageref{not: I2}, \pageref{eq: I1 = inf I2}, \pageref{not: Ik}, \pageref{eq: dawson gartner rate function}
		\item[$\underline s,\overline s$,] \pageref{not: sbar}
		\item[$s$,] RL function, \pageref{not: s}
		\item[$\langle \cdot, \cdot\rangle$,] dual pairing, \pageref{not: pairing}, \pageref{not: pairing infty}
		\item[$\|\cdot\|$,] supremum norm, \pageref{not: uniform norm}
		\item[$\cdot^*$,] convex conjugate,  \pageref{eq: definition Lambda infty *}, \pageref{section: duality}
		\item[$\ent (\cdot|\cdot)$,] relative entropy, \pageref{not: relative entropy}
		\item[$\idv$,] short for $\ith \idv2$, \pageref{not: IDV}
		\item[$\ith \idv k$,] DV entropy, \pageref{not: IDV}, \pageref{not: IDV 1}, \pageref{not: IDV k}, \pageref{not: IDV infty}
		\item[$R$,] short for $\ith R2$, \pageref{not: R}
		\item[$\ith R k$,] \pageref{not: R}, \pageref{not: R1}, \pageref{not: Rk}, \pageref{not: R infty}
		\item[$\Lambda$,] short for $\ith \Lambda2$, \pageref{not: Lambda}
		\item[$\ith \Lambda k$,] SCGF, \pageref{not: Lambda}, \pageref{not: Lambda 1}, \pageref{not: Lambda k}, \pageref{not: Lambda level 3}
		\item[$(\ith \Lambda k)^*$,] \pageref{not: Lambda *}, \pageref{not: Lambda 1}, \pageref{not: Lambda k}, \pageref{not: Lambda * level 3}, \pageref{section: duality}
		\item[$\Lambda_\infty$,] short for $\ith \Lambda2_\infty$, \pageref{not: Lambda infty}
		\item[$(\ith \Lambda 2_\infty )^*$,] \pageref{not: Lambda infty *}, \pageref{section: duality}
	\end{itemize}
	\subsubsection*{Measures}
	\begin{itemize}
		\setlength\itemsep{-0.3em}
		\item[$\tvnorm{\mu}$,] total variation norm of $\mu$, \pageref{not: TV}
		\item[$\mathcal B(\mu, \rho)$,] ball around $\mu$, \pageref{not: ball}
		\item[$\supp \mu$,] support of $\mu$, \pageref{not: supp}
		\item[$\ith \mu k$,] marginal of $\mu$, \pageref{not: mu1}, \pageref{not: Pbalk}, \pageref{not: Pbalinfty}
		\item[$\mu\otimes q$,] \pageref{not: muotimeq}
		\item[$\mu|_A$,] \pageref{not: muA}
		\item[$\pi_k$,] \pageref{not: pik}
		\item[$\pi_{k, k'}$,] \pageref{not: pi1}, \pageref{not: pikk}
	\end{itemize}
	\subsubsection*{Words and operations on words,\\ slicing and stitching}
	\begin{itemize}
		\setlength\itemsep{-0.3em}
		\item[$e$,] empty word, \pageref{not: e}
		\item[$p(u)$,] \pageref{not: p words}
		\item[$M\mbox{[$u$]}$,] \pageref{not: M(u)}
		\item[$L\mbox{[$u$]}$,] \pageref{not: L(u)}
		\item[$|u|$,] length of $u$, \pageref{not: length}
		\item[$|\underline u|$,] total length of $\underline u$, \pageref{not: |u|}
		\item[$\kv u$,] number of items in $\underline u$, \pageref{not: ku}
		\item[$\words$,] set of words, \pageref{not: words}
		\item[$\wordspos$,] set of words, \pageref{not: wordspos}
		\item[$\closeword_t$,] set of words, \pageref{not: closeword1}
		\item[$\closeword_{t,S'}$,] set of words, \pageref{not: closeword2}
		\item[$\mathcal J_0$,] finite set of indices, \pageref{section: slicing and stitching}
		\item[$r$,] number of classes, \pageref{eq: class naming convention}
		\item[$K$,] finite subset of $\spaces$, \pageref{section: slicing and stitching}
		\item[$K_j$,] finite subset of a class, \pageref{eq: class naming convention}
		\item[$\Gamma$,] \pageref{eq: definition Gamma}
		\item[$\xi_{x,y}$, $\xi^{i}$,] transition words, \pageref{lemma: stitching words}
		\item[$\eta$,] minimal transition probability, \pageref{lemma: stitching words}
		\item[$\tau$,] maximal transition length, \pageref{lemma: stitching words}
		\item[$\stitchable$,] stitchable lists, \pageref{def: the stitching map}
		\item[$\sigma$,] reordering map, \pageref{lemma: reoredring the list}
		\item[$F_{\mathcal J'}$, $F$, $F_n$,] slicing map, \pageref{def: the slicing map}, \pageref{lemma: proba ineq for the slicing map}
		\item[$G_t$,] stitching map, \pageref{def: the stitching map}
		\item[$\Psi_{n,t}$,] coupling map, \pageref{def: the coupling map}
		\item[$\widetilde \Psi_n$,] decoupling map, \pageref{def: the decoupling map}
	\end{itemize}
	
\end{multicols}
\section*{Acknowledgments}
The author would like to thank Noé Cuneo for many pivotal conversations on this topic. The author is also grateful to	Orphée Collin, Armen Shirikyan and Ofer Zeitouni for helpful discussions and advice, to Arnaud de La Fortelle, who, altough he has long since left academia, was kind enough to answer some questions about his 2002 article, and to Andrea Agazzi and Léo Micollet for sharing interesting examples and references.

\bibliographystyle{elsarticle-num} 
\bibliography{papier.bib}
\end{document}